\documentclass[12pt]{article}
\usepackage{amssymb,amsmath,amsthm, amsfonts}
\usepackage{graphicx}
\usepackage{epsfig}
\usepackage{tikz}
\usepackage{mathrsfs}

\def\disp{\displaystyle}

\def\dref#1{(\ref{#1})}

\theoremstyle{plain}
\newtheorem{theorem}{Theorem}[section]
\newtheorem{lemma}{Lemma}[section]

\theoremstyle{definition}
\newtheorem{definition}{Definition}[section]

\newtheorem{remark}{Remark}[section]

\numberwithin{equation}{section}

\begin{document}

\title{\bf A new result for the global existence (and boundedness), regularity and stabilization of a three-dimensional
Keller-Segel-Navier-Stokes system modeling coral fertilization
}

\author{
Jiashan Zheng\thanks{Corresponding author.   E-mail address:
 zhengjiashan2008@163.com (J.Zheng)}
 \\
    School of Mathematics and Statistics Science,\\
     Ludong University, Yantai 264025,  P.R.China \\
}
\date{}


\maketitle \vspace{0.3cm}
\noindent
\begin{abstract}
This paper deals with the following quasilinear Keller-Segel-Navier-Stokes system modeling coral fertilization
$$
 \left\{
 \begin{array}{l}
   n_t+u\cdot\nabla n=\Delta n-\nabla\cdot(nS(x,n,c)\nabla c)-nm,\quad
x\in \Omega, t>0,\\
    c_t+u\cdot\nabla c=\Delta c-c+m,\quad
x\in \Omega, t>0,\\
 m_t+u\cdot\nabla m=\Delta m-nm,\quad
x\in \Omega, t>0,\\
u_t+\kappa(u \cdot \nabla)u+\nabla P=\Delta u+(n+m)\nabla \phi,\quad
x\in \Omega, t>0,\\
\nabla\cdot u=0,\quad
x\in \Omega, t>0\\
 \end{array}\right.\eqno(*)
 $$
 under no-flux boundary conditions in a bounded domain $\Omega\subset \mathbb{R}^3$ with smooth boundary, where $\phi\in W^{2,\infty} (\Omega)$.
   Here
    $S(x,n,c)$ denotes the rotational effect which
satisfies $|S(x,n,c)|\leq S_0 (c)(1 + n)^{-\alpha}$  with $\alpha\geq0$ and
some nonnegative nondecreasing
function $S_0$.  
Based on a new  weighted estimate and some carefully analysis, if $\alpha>0$, then for any $\kappa\in\mathbb{R},$ system $(*)$ possesses a global weak  solution  for which there exists
$T > 0$ such that $(n,c,m , u)$ is smooth in $\Omega\times( T ,\infty)$.
Furthermore, for any $p>1,$ this
solution is uniformly bounded in with respect to the norm
in $L^p(\Omega)\times L^\infty(\Omega) \times L^\infty(\Omega)\times  L^2 (\Omega; \mathbb{R}^3)$.
Building on this boundedness
property and some other analysis, it can finally even be proved that in the large time limit, any such solution approaches
the spatially homogeneous equilibrium $(\hat{n},\hat{m},\hat{m},0)$ in an appropriate sense, where $\hat{n}=\frac{1}{|\Omega|}\{\int_{\Omega}n_0-\int_{\Omega}m_0\}_{+}$ and $\hat{m}=\frac{1}{|\Omega|}\{\int_{\Omega}m_0 -\int_{\Omega}n_0\}_{+}$.
\end{abstract}

\vspace{0.3cm}
\noindent {\bf\em Key words:}~
Navier-Stokes system; Keller-Segel model; 
Global existence; Large time behavior; 
Tensor-valued
sensitivity

\noindent {\bf\em 2010 Mathematics Subject Classification}:~ 35K55, 35Q92, 35Q35, 92C17

\newpage
\section{Introduction}

This work is concerned with  the following chemotaxis-fluid system modelling coral fertilization:
\begin{equation}
 \left\{\begin{array}{ll}
   n_t+u\cdot\nabla n=\Delta n-\nabla\cdot(nS(x,n,c)\nabla c)-nm,\quad
x\in \Omega, t>0,\\
    c_t+u\cdot\nabla c=\Delta c-c+m,\quad
x\in \Omega, t>0,\\
 m_t+u\cdot\nabla m=\Delta m-nm,\quad
x\in \Omega, t>0,\\
u_t+\kappa(u \cdot \nabla)u+\nabla P=\Delta u+(n+m)\nabla \phi,\quad
x\in \Omega, t>0,\\
\nabla\cdot u=0,\quad
x\in \Omega, t>0\\
 \disp{(\nabla n-nS(x, n, c))\cdot\nu=\nabla c\cdot\nu=\nabla m\cdot\nu=0,u=0,}\quad
x\in \partial\Omega, t>0,\\
\disp{n(x,0)=n_0(x),c(x,0)=c_0(x),m(x,0)=m_0(x),u(x,0)=u_0(x),}\quad
x\in \Omega,\\
 \end{array}\right.\label{334451.1fghyuisda}
\end{equation}
where  $\Omega\subset \mathbb{R}^3$ is a bounded domain with smooth boundary $\partial\Omega$
 and
the matrix-valued function $S(x,n,c)$ indicates the rotational effect which
satisfies
\begin{equation}\label{x1.73142vghf48rtgyhu}
S\in C^2(\bar{\Omega}\times[0,\infty)^2;\mathbb{R}^{3\times3})
 \end{equation}
 and
 \begin{equation}\label{x1.73142vghf48gg}|S(x,n,c)|\leq (1+n)^{-\alpha}S(c)~~ \mbox{for all}~~ (x,n,c)\in\Omega\times[0,\infty)^2~~\mbox{with}~~ S(c)~~
\mbox{nondecreasing on}~~ [0,\infty)
 \end{equation}
 and
 $\alpha\geq0.$
 As described in \cite{Kiselevdd793,Kiselevsssdd793,EspejoEspejojainidd793,EspejojjEspejojainidd793}, problems of this type arise in the modeling of
 the phenomenon of coral broadcast spawning, where the sperm $n$ chemotactically moves toward the higher concentration of the chemical $c$
released by the egg $m$, while the egg $m$ is merely affected by random diffusion, fluid transport
and degradation upon contact with the sperm (see also \cite{Lidfff00}).
Here $\kappa,u,P$ and $\phi$ 
denote,
 respectively, the
strength of nonlinear fluid convection, the velocity field, the associated pressure of the fluid and the potential of the
 gravitational field.
We further
note that the sensitivity tensor $S(x,n,c)$ may take values that are matrices possibly containing
nontrivial off-diagonal entries, which reflects that the chemotactic migration may not necessarily
be oriented along the gradient of the chemical signal, but may rather involve rotational flux
components (see \cite{Xusddeddff345511215,XueXuejjainidd793} for the detailed model derivation).

Chemotaxis is the directed movement of the cells  as a response to gradients of the
concentration of the chemical signal substance in their environment, where the chemical signal substance
may be produced  or consumed by cells themselves (see e.g. Hillen and Painter \cite{Hillen} and \cite{Bellomo1216}).
The classical chemotaxis system was introduced in 1970 by Keller and
Segel (\cite{Keller2710}), which is called Keller-Segel system. Since then, the Keller-Segel model has attracted more and more attention, and also has been constantly modified by
various authors to characterize more biological phenomena (see Cie\'{s}lak and  Stinner \cite{Cie791}, Cie\'{s}lak and  Winkler \cite{Cie72},
Ishida et al. \cite{Ishida}, Painter and Hillen \cite{Painter55677},   Hillen and Painter \cite{Hillen},
Wang et al. \cite{Wang76,Wang79}, Winkler et al. \cite{Cie72,Horstmann791,Winkler72,Winkler792,Winkler21215,Winkler2233444ssdff51215,Winkler793},
Zheng \cite{Zheng00}
and references therein for detailed
results). For  related works in this direction, we mention that a corresponding quasilinear version ( see e.g. \cite{Tao794,Winkler72,Zhengssdefr23,Zheng00,Zhengsddfffsdddssddddkkllssssssssdefr23}),
the logistic damping or the signal consumed by the cells, has been deeply investigated by Cie\'{s}lak and Stinner \cite{Cie791,Cie201712791},
Tao and Winkler \cite{Tao794,Winkler79,Winkler72}, and Zheng et al. \cite{Zheng00,Zhengssssdefr23,Zhengssdefr23,Zhengssssssdefr23}.

In various situations, however, the interaction of chemotactic movement of the gametes and the surrounding fluid is not negligible (see Tuval et al. \cite{Tuval1215}). In 2005,
Tuval et al. (\cite{Tuval1215}) proposed
   the following prototypical {\bf signal consuming} model (with tensor-valued
sensitivity):
\begin{equation}
 \left\{\begin{array}{ll}
   n_t+u\cdot\nabla n=\Delta n-\nabla\cdot( nS(x,n,c)\nabla c),\quad
x\in \Omega, t>0,\\
    c_t+u\cdot\nabla c=\Delta c-nf(c),\quad
x\in \Omega, t>0,\\
u_t+\kappa (u\cdot\nabla)u+\nabla P=\Delta u+n\nabla \phi,\quad
x\in \Omega, t>0,\\
\nabla\cdot u=0,\quad
x\in \Omega, t>0,\\
 \end{array}\right.\label{1.1hhjffggjddssggtyy}
\end{equation}
where $f(c)$ 
denotes the {\bf consumption} rate of the oxygen by the cells.
Here  $S$ is a tensor-valued function or a scalar function which is the same as \dref{1.1fghyuisda}.
 The model \dref{1.1hhjffggjddssggtyy} describes the interaction of oxygen-taxis bacteria with a surrounding incompressible viscous fluid
in which the oxygen is dissolved.
  After
this, assume that the chemotactic sensitivity $S(x, n, c):=S(c)$ is  a scalar function. 
This kind of models have been studied by many researchers  by making use of energy-type functionals (see e.g.
Chae et. al. \cite{Chaexdd12176},
Duan et. al. \cite{Duan12186},
Liu and Lorz  \cite{Liu1215,Lorz1215},
 Tao and Winkler   \cite{Tao41215,Winkler31215,Winkler61215,Winkler51215}, Zhang and Zheng \cite{Zhang12176} and references therein).
 In fact, if $S(x, n, c):=S(c)$, Winkler (\cite{Winkler31215} and  \cite{Winkler61215}) proved  that in two-dimensional space \dref{1.1hhjffggjddssggtyy}
admits a unique global classical solution which stabilizes to the spatially homogeneous equilibrium $(\frac{1}{|\Omega|}\int_{\Omega}n_0 , 0, 0)$ in the large time limit. While in three-dimensional
setting, he (see \cite{Winkler51215}) also showed   that there exists a globally defined weak solution to \dref{1.1hhjffggjddssggtyy}.

Experiment  \cite{Xusddeddff345511215} show that the chemotactic movement could be not directly along the signal
gradient, but with a rotation,
so that, the the corresponding chemotaxis-fluid system with tensor-valued
sensitivity  
loses entropy-like functional structure, which gives rise to considerable mathematical difficulties during the process of
analysis.  The global solvability of corresponding initial value problem for chemotaxis-fluid system with tensor-valued
sensitivity have been deeply investigated by Cao, Lankeit \cite{CaoCaoLiitffg11},
Ishida \cite{Ishida1215},  Wang et al. \cite{Wang11215,Wang21215}  and Winkler \cite{Winkler11215}. 

If $-nf(c)$ in
the $c$-equation is replaced by $-c+n$, and the $u$-equation is a
(Navier-)Stokes equation, then \dref{1.1hhjffggjddssggtyy} becomes the following chemotaxis-(Navier-)Stokes system in the context of signal {\bf produced} other than consumed  by cells
(see \cite{Winkler444ssdff51215,Wddffang11215,Wang21215,Wangss21215,Zhenddddgssddsddfff00,Kegssddsddfff00})
\begin{equation}
 \left\{\begin{array}{ll}
   n_t+u\cdot\nabla n=\Delta n-\nabla\cdot( n S(x,n,c)\cdot\nabla c),\quad
x\in \Omega, t>0,\\
    c_t+u\cdot\nabla c=\Delta c-c+n,\quad
x\in \Omega, t>0,\\
u_t+\kappa (u\cdot\nabla)u+\nabla P=\Delta u+n\nabla \phi,\quad
x\in \Omega, t>0,\\
\nabla\cdot u=0,\quad
x\in \Omega, t>0. \\
 \end{array}\right.\label{1sdfdffgggggsxdcfffggvgb.1}
\end{equation}
Due to the presence of the tensor-valued sensitivity $S(x,n,c)$ as well as the strongly nonlinear term $(u\cdot\nabla)u$ and lower regularity for $n$,
 the analysis of \dref{1sdfdffgggggsxdcfffggvgb.1} with tensor-valued
sensitivity began to flourish
 (see \cite{Winkler444ssdff51215,Wddffang11215,Wang21215,Wangss21215,Zhenddddgssddsddfff00,Kegssddsddfff00}).
In fact,  the global boundedness of classical
solutions to the Stokes-version ($\kappa=0$ in the third equation of system \dref{1sdfdffgggggsxdcfffggvgb.1}) of system \dref{1sdfdffgggggsxdcfffggvgb.1} with the tensor-valued $S$
satisfying $|S(x,n,c)|\leq C_S(1+n)^{-\alpha}$
with some $C_S > 0$ and $\alpha > 0$ which implies that the
effect of chemotaxis is weakened when the cell density increases has been proved
for any $\alpha > 0$ in two dimensions (see Wang and Xiang \cite{Wang21215}) and for $\alpha >\frac{1}{2}$
in three dimensions (see Wang and Xiang \cite{Wangss21215}). Then Wang-Winkler-Xiang (\cite{Wddffang11215}) further
shows that when $\alpha > 0$ and $\Omega\subset R^2$ is a bounded {\bf convex} domain with smooth boundary,
system \dref{1sdfdffgggggsxdcfffggvgb.1} possesses a global-in-time classical and bounded solution.
Recently,  Zheng (\cite{Zhenddsdddddgssddsddfff00})   extends   the results of   \cite{Wddffang11215} to the  general bounded domain by  some new
entropy-energy estimates. More recently,
by using new entropy-energy estimates,
Zheng and Ke  (\cite{Kegssddsddfff00}) presented the existence of global and weak solutions for the system \dref{1sdfdffgggggsxdcfffggvgb.1} under the assumption that $S$ satisfies
\dref{x1.73142vghf48rtgyhu} and $$|S(x,n,c)|\leq (1+n)^{-\alpha}~~ \mbox{for all}~~ (x,n,c)\in\Omega\times[0,\infty)^2$$ with $\alpha > \frac{1}{3}$,
which, in light of the known results for the fluid-free system (see Horstmann and Winkler \cite{Horstmann791} and
Bellomo et al. \cite{Bellomo1216} ), is an optimal restriction on $\alpha$.
  For more works
about the chemotaxis-(Navier-)Stokes models \dref{1sdfdffgggggsxdcfffggvgb.1}, we mention that a corresponding quasilinear version
or the logistic damping has been deeply investigated by  Zheng
\cite{Zhengsdsd6}, Wang and Liu \cite{Liuddfffff}, Tao and Winkler \cite{Tao41215},
 Wang et. al.
\cite{Wang21215,Wangss21215}.

Other variants of the model \dref{1sdfdffgggggsxdcfffggvgb.1} has been used in the mathematical study of coral broad-
cast spawning.  In fact, Kiselev and Ryzhik (\cite{Kiselevdd793} and \cite{Kiselevsssdd793}) introduced
   the following Keller-Segel type system to model coral fertilization:
   \begin{equation}
 \left\{\begin{array}{ll}
 \rho_t+u\cdot\nabla\rho=\Delta \rho-\chi\nabla\cdot( \rho\nabla c)-\varepsilon\rho^q,
 \quad
\\
 \disp{ 0=\Delta c +\rho,}\quad\\
 \end{array}\right.\label{722dff344101.dddddgghggghhff2ffggffggx16677}
\end{equation}
where $\rho,u,\chi$ and $-\varepsilon \rho^q$, respectively,  denote  the density of egg (sperm) gametes, the smooth divergence free sea fluid
velocity as well as 
 the positive chemotactic sensitivity constant
 and the
reaction (fertilization) phenomenon.  In fact,  under suitable conditions,  the global-in-time existence of the solution to
\dref{722dff344101.dddddgghggghhff2ffggffggx16677} is presented by  Kiselev and Ryzhik in \cite{Kiselevdd793}. Moreover, they proved that the total
mass $m_0(t) =\int_{R^2}\rho(x,t)dx$ approaches a positive constant whose lower bound is $C(\chi,\rho_0 ,u)$
as $t\rightarrow\infty$ when $q>2$.  In the critical case of $N = q = 2$,
%
 a corresponding weaker
but yet relevant effect within finite time intervals is detected (see \cite{Kiselevsssdd793}).

In order  to analyze a further refinement of the model \dref{722dff344101.dddddgghggghhff2ffggffggx16677} which
explicitly distinguishes between sperms and eggs,
Espejo and Winkler (\cite{EspejojjEspejojainidd793}) have recently considered the Navier-Stokes version of \dref{1.1fghyuisda}:
\begin{equation}
 \left\{\begin{array}{ll}
   n_t+u\cdot\nabla n=\Delta n-\nabla\cdot(n\nabla c)-nm,\quad
x\in \Omega, t>0,\\
    c_t+u\cdot\nabla c=\Delta c-c+m,\quad
x\in \Omega, t>0,\\
 m_t+u\cdot\nabla m=\Delta m-nm,\quad
x\in \Omega, t>0,\\
u_t+\kappa(u \cdot \nabla)u+\nabla P=\Delta u+(n+m)\nabla \phi,\quad
x\in \Omega, t>0,\\
\nabla\cdot u=0,\quad
x\in \Omega, t>0\\
 \end{array}\right.\label{1.dddddffdffg1}
\end{equation}
in a bounded domain $\Omega\subset  \mathbb{R}^2$. If $N=2,$
Espejo and Winkler (\cite{EspejojjEspejojainidd793}) established the global existence of classical solutions to the associated initial-boundary
value problem \dref{1.dddddffdffg1}, which tend towards a spatially homogeneous equilibrium in the large time limit. 
Furthermore, if $S(x, n, c)$ satisfying \dref{x1.73142vghf48rtgyhu} and \dref{x1.73142vghf48gg} with $\alpha \geq \frac{1}{3}$ or $\alpha \geq0$ and the initial data satisfy a certain smallness condition, Li-Pang-Wang (\cite{Lidfff00}) proved the same result
for the
three-dimensional Stokes ($\kappa=0$ in the fourth equation of \dref{334451.1fghyuisda}) version of system \dref{334451.1fghyuisda}. 
From  \cite{Lidfff00}, we know that
 $\alpha\geq\frac{1}{3}$
is enough to warrant  the boundedness  of solutions to system \dref{1.1fghyuisda} for any large
data (see Li-Pang-Wang \cite{Lidfff00}). We should point that the core step of \cite{Lidfff00} is to establish the estimates of the
functional
$$\|n(\cdot,t) \|_{L^2(\Omega)}^2+\|\nabla c(\cdot,t) \|_{L^2(\Omega)}^2+\|u(\cdot,t) \|_{W^{1,2}(\Omega)}^2,$$
which strongly relies on $\alpha\geq\frac{1}{3} $ and $\kappa=0$ (see the proof of Lemma 3.1 of \cite{Lidfff00}). To the best
of our knowledge, it is yet unclear whether for $\alpha<\frac{1}{3}$ or $\kappa\neq0$, the solutions of \dref{1.1fghyuisda}  exist (or even bounded) or not.
Recently, relying on  the functional
$$
\left\{
\begin{array}{rl}
\disp{\int_{\Omega} n_{\varepsilon}^{4\alpha+\frac{2}{3}}+\int_{\Omega}   |\nabla c_{\varepsilon}|^2+\int_{\Omega}  | {u_{\varepsilon}}|^2~~~\mbox{if}~~\alpha\neq\frac{1}{12},}\\
\disp{\int_{\Omega} n_{\varepsilon}\ln n_{\varepsilon}+\int_{\Omega}   |\nabla c_{\varepsilon}|^2+\int_{\Omega}  | {u_{\varepsilon}}|^2~~~~\mbox{if}~~\alpha=\frac{1}{12},}
\end{array}
\right.
$$
we  (\cite{Zhenssssssdffssdddddddgssddsddfff00}) presented the existence of global {\bf  weak} solutions for the system \dref{334451.1fghyuisda} under the assumption that $S$ satisfies
\dref{x1.73142vghf48rtgyhu} and \dref{x1.73142vghf48gg} with $\alpha > 0$.
However, the existence of  global (stronger than the result of \cite{Zhenssssssdffssdddddddgssddsddfff00}) {\bf  weak} solutions is still open.
In this paper, by using a new  weighted estimate (see Lemma \ref{lemmaghjffggssddgghhmk4563025xxhjklojjkkk}),
we try to obtain  enough regularity and compactness properties (see Lemmas \ref{lemmaghjffggssddgghhmk4563025xxhjklojjkkk},
\ref{ssdddlemmaghjffggssddgghhmk4563025xxhjklojjkkk}, and \ref{lemmddaghjsffggggsddgghhmk4563025xxhjklojjkkk}), then show that system \dref{334451.1fghyuisda} possesses a globally defined  {\bf weak} solution,
which improves the result of \cite{Zhenssssssdffssdddddddgssddsddfff00}.
Therefore, collecting the above results, it is meaningful to analyze the following question:

Whether or not the assumption of $\alpha$ is optimal? Can we further relax the
restriction on $\alpha$, say, to $\alpha > 0 $? Moreover,  can we consider  the regularity of global solution for system \dref{334451.1fghyuisda}?

Inspired by the above works, the  first result of paper is to
%
 prove the existence of global
(and bounded) solution for any $\alpha > 0.$ Moreover, we also show that the corresponding
solutions converge to a spatially homogeneous equilibrium exponentially as $t \rightarrow\infty$ as well.

Throughout this paper,
we assume that
\begin{equation}
\phi\in W^{2,\infty}(\Omega)
\label{x1.73142vghf481}
\end{equation}
 and the initial data
$(n_0, c_0, u_0)$ fulfills
\begin{equation}\label{ccvvx1.731426677gg}
\left\{
\begin{array}{ll}
\displaystyle{n_0\in C(\bar{\Omega})~~~~ \mbox{with}~~ n_0\geq0 ~~\mbox{and}~~n_0\not\equiv0},\\
\displaystyle{c_0\in W^{1,\infty}(\Omega)~~\mbox{with}~~c_0\geq0~~\mbox{in}~~\bar{\Omega},}\\
\displaystyle{m_0\in C(\bar{\Omega})~~~~ \mbox{with}~~ m_0\geq0 ~~\mbox{and}~~m_0\not\equiv0},\\
\displaystyle{u_0\in D(A^\gamma_{r})~~\mbox{for~~ some}~~\gamma\in ( \frac{3}{4}, 1)~~\mbox{and any}~~ {r}\in (1,\infty),}\\
\end{array}
\right.
\end{equation}
where $A_{r}$ denotes the Stokes operator with domain $D(A_{r}) := W^{2,{r}}(\Omega)\cap  W^{1,{r}}_0(\Omega)
\cap L^{r}_{\sigma}(\Omega)$,
and
$L^{r}_{\sigma}(\Omega) := \{\varphi\in  L^{r}(\Omega)|\nabla\cdot\varphi = 0\}$ for ${r}\in(1,\infty)$
 (\cite{Sohr}).

In the context of these assumptions, the first of our main results  can be read as follows.
%
\begin{theorem}\label{theorem3}
Let $\Omega \subset \mathbb{R}^3$ be a bounded  domain with smooth boundary.
Suppose that the assumptions \dref{x1.73142vghf48rtgyhu}--\dref{x1.73142vghf48gg} and
 \dref{x1.73142vghf481}--\dref{ccvvx1.731426677gg}
 hold.
 If 
\begin{equation}\label{x1.73142vghf48}\alpha>0,
\end{equation}
then for any $\kappa\in \mathbb{R}$, there exist
\begin{equation}
 \left\{\begin{array}{ll}
 n\in L^{\infty}_{loc}([0,\infty),L^p(\Omega))\cap L^{2}_{loc}([0,\infty),W^{1,2}(\Omega))~~~\mbox{for any}~~p>1,\\
  c\in  L^{\infty}(\Omega\times(0,\infty))\cap L^{2}_{loc}([0,\infty),W^{2,2}(\Omega))\cap L^{4}_{loc}([0,\infty),W^{1,4}(\Omega)),\\
   m\in  L^{\infty}(\Omega\times(0,\infty))\cap L^{2}_{loc}([0,\infty),W^{2,2}(\Omega))\cap L^{4}_{loc}([0,\infty),W^{1,4}(\Omega)),\\
  u\in  L^{2}_{loc}([0,\infty),W^{1,2}_{0,\sigma}(\Omega))\cap L^{\infty}_{loc}([0,\infty),L^{2}(\Omega)),\\
   \end{array}\right.\label{1.1ddfghyuisdsdddda}
\end{equation}
such that $(n,c,m,u)$ is a global weak solution of the problem \dref{1.1fghyuisda} in the natural sense as specified
in \cite{Zhenssssssdffssdddddddgssddsddfff00}.
Moreover, if $\kappa=0$,
 the problem \dref{1.1fghyuisda} possesses at least
one global  classical solution $(n, c,m, u, P)$. Moreover, this solution is bounded in
$\Omega\times(0,\infty)$ in the sense that
\begin{equation}
\|n(\cdot, t)\|_{L^\infty(\Omega)}+\|c(\cdot, t)\|_{W^{1,\infty}(\Omega)}+\|m(\cdot, t)\|_{W^{1,\infty}(\Omega)}+\| u(\cdot, t)\|_{L^{\infty}(\Omega)}\leq C~~ \mbox{for all}~~ t>0.
\label{1.163072xggttyyu}
\end{equation}
\end{theorem}
\begin{remark}
(i)  Theorem \ref{theorem3} indicates that $\alpha > 0$ and $\kappa=0$ is enough to ensure the global existence and uniform boundedness of solution of the three-dimensional Keller-Segel-Stokes system \dref{334451.1fghyuisda}, which
improves the result obtained in \cite{Lidfff00}, therein $\alpha\geq
\frac{1}{3}$
is required.


(ii)This result also   improves the result of our recent paper (\cite{Zhenssssssdffssdddddddgssddsddfff00}), where the more {\bf weak} solution than our result  was obtained by using different method (see \cite{Zhenssssssdffssdddddddgssddsddfff00}).

\end{remark}

We can secondly prove that in fact any such {\bf weak} solution $(n,c,m,u)$ becomes
smooth ultimately, and that it approaches the unique spatially homogeneous steady
state $(\hat{n},\hat{m},\hat{m},0)$, where $\hat{n}=\frac{1}{|\Omega|}\{\int_{\Omega}n_0-\int_{\Omega}m_0\}_{+}$ and $\hat{m}=\frac{1}{|\Omega|}\{\int_{\Omega}m_0 -\int_{\Omega}n_0\}_{+}$.



\begin{theorem}\label{thaaaeorem3}
 Under the assumptions of Theorem \ref{theorem3}, then there are $T > 0$ and $\iota\in (0,1)$
such that the solution $(n,c,m,u)$ given by Theorem \ref{theorem3} satisfies
$$n,c,m\in C^{2+\iota,1+\frac{\iota}{2}}(\bar{\Omega}\times[T,\infty)),~~~u\in C^{2+\iota,1+\frac{\iota}{2}}(\bar{\Omega}\times[T,\infty);\mathbb{R}^3).$$
Moreover,
$$n(\cdot,t)\rightarrow \hat{n},~c(\cdot,t)\rightarrow \hat{m}
~~\mbox{as well as } ~~~m(\cdot,t)\rightarrow \hat{m}~~\mbox{and}~~~u(\cdot,t)\rightarrow0
~~\mbox{in}~~~L^\infty(\Omega),$$
where $\hat{n}=\frac{1}{|\Omega|}\{\int_{\Omega}n_0-\int_{\Omega}m_0\}_{+}$ and $\hat{m}=\frac{1}{|\Omega|}\{\int_{\Omega}m_0 -\int_{\Omega}n_0\}_{+}$.
\end{theorem}
\begin{remark}
 Theorem \ref{theorem3} indicates that if $\alpha > 0$, then for arbitrarily large initial data and for any $\kappa\in \mathbb{R}$, this problem
admits at least one global weak solution for which there exists
$T > 0$ such that $(n,c,m, u)$ is  smooth in $\Omega\times ( T ,\infty)$.
Moreover, it is asserted that such solutions are shown to approach a
spatially homogeneous equilibrium in the large time limit, which
improves the result obtained in \cite{Lidfff00}, therein $\kappa=0$
is required.

%


%
%

\end{remark}
{\bf Mathematical challenges for the regularity and stabilization of the solution for system  \dref{334451.1fghyuisda}.}
 System \dref{334451.1fghyuisda} incorporates  fluid and
rotational flux, which involves more complex cross-diffusion mechanisms and brings
about many considerable mathematical difficulties.
Firstly,  even when posed without any external influence, that is, $n=c=m\equiv0,$ the corresponding Navier-Stokes system \dref{334451.1fghyuisda} does not admit a satisfactory solution
theory up to now (see Leray \cite{LerayLerayer792} and Sohr \cite{Sohr}, Wiegner \cite{Wiegnerdd79}). As far as we know that
 the question of global solvability in classes of suitably regular functions yet remains open
except in cases when the initial data are appropriately small (see e.g.  Wiegner \cite{Wiegnerdd79}).
Moreover, the tensor-valued sensitivity functions result in new mathematical difficulties, mainly linked to the fact that a chemotaxis system with such rotational fluxes
thereby loses an energy-like structure (see e.g. \cite{Xusddeddff345511215}).
In \cite{Lidfff00} and \cite{EspejojjEspejojainidd793}, relying on  globally {\bf bounded} for the solution,  Espejo-Winkler
(\cite{EspejojjEspejojainidd793})
Li-Pang-Wang (\cite{Lidfff00}) proved  that all these solutions  of problem \dref{334451.1fghyuisda} are shown to approach a
spatially homogeneous equilibrium in the large time limit when $N=2$ or $N=3$ and $\kappa=0,$ respectively.
As already mentioned in the above, in the case $N = 3$, it
is not only unknown whether the incompressible Navier-Stokes equations possess
global smooth solutions for arbitrarily large smooth initial data (see e.g. Wiegner \cite{Wiegnerdd79} and Sohr \cite{Sohr}).
Therefore, when $\kappa\not=0$ and $N=3$, we can not use  the idea of \cite{Lidfff00} and \cite{EspejojjEspejojainidd793} to discuss the
large time behavior to problem \dref{334451.1fghyuisda}, since,
%
 the globally {\bf bounded} for the solutions are needed in \cite{Lidfff00} and \cite{EspejojjEspejojainidd793}.

In order to derive these theorems, in Section 2, we introduce the regularized system of \dref{334451.1fghyuisda}, establish some basic estimates of the solutions and recall a local existence result. In Section 3,  a key step of the proof of our main results is to establish a bound for  $n_\varepsilon(\cdot, t)$ in $L^{{p}} (\Omega)$ for any $p>1$. The approach is based on
the weighted estimate of
$\int_{\Omega}n_\varepsilon^{{p}}g(c_\varepsilon)$ with some weight function $g(c_\varepsilon)$ which is uniformly bounded both from above and below by positive
constants. Here $n_\varepsilon$ and $c_\varepsilon$ are components of the solutions to \dref{1.1fghyuisda} below.
On the basis of the previously established estimates and the compactness properties thereby implied,  we shall
pass to the limit along an adequate sequence of numbers $\varepsilon = \varepsilon_j\searrow0$
and thereby verify Theorem \ref{theorem3}.
%
Using the basic
relaxation properties expressed in \dref{ddczhjjjj2.5ghxxccju48cfg9ssdd24} and \dref{ddczhjjjj2.5ghxxssddccju48cfg9ssdd24}, Section 4  is devoted to showing the large time behavior of global solutions to \dref{334451.1fghyuisda} obtained
in the above section. To this end,   thanks to the decay property of
$m_\varepsilon(\cdot,t)-\hat{m}+n_\varepsilon(\cdot,t)
-\hat{n}$ formulated in Lemmas \ref{ssdddlemmddddaddffffdfffgg4sssdddd5630} and \ref{11aaalemdfghkkmaddffffdfffgg4sssdddd5630}, this actually entails a certain eventual regularity and
decay of $u_\varepsilon$ also in the present situation, where $\hat{n}=\frac{1}{|\Omega|}\{\int_{\Omega}n_0-\int_{\Omega}m_0\}_{+}$ and $\hat{m}=\frac{1}{|\Omega|}\{\int_{\Omega}m_0 -\int_{\Omega}n_0\}_{+}$.
Using these bounds (see Lemmas \ref{ssdddlemmddddaddffffdfffgg4sssdddd5630}--\ref{aaalemmaddffffsddddfffgg4sssdddd5630}), based on maximal
Sobolev regularity in the Stokes evolution system as well as  inhomogeneous linear heat
equations  and the  standard Schauder theory, we then prove eventual H\"{o}lder regularity and smoothness of solution $(n_\varepsilon,c_\varepsilon,m_\varepsilon,u_\varepsilon)$ (see Lemmas \ref{lemma45630hhuujjuuyy}--\ref{lemma45630hhuujjsdfffggguuyy}).
 For convergence as $t\rightarrow\infty$, we  draw upon uniform H\"{o}lder bounds and smoothness for solution $(n,c,m,u)$ (see Lemmas \ref{lemma45630223}--\ref{sssslemma45ssddddff630hhuujjsdfffggguuyy}). Finally, applying an Ehrling-type
lemma, we can prove   any such solution approaches
the spatially homogeneous equilibrium by using  the above convergence properties (see Lemma \ref{lemma4dd5630hhuujjuuyy}).

\section{Preliminaries}
As mentioned in the introduction, the chemotactic sensitivity $S$ in the first equation in \dref{334451.1fghyuisda} and the nonlinear convective term $\kappa(u \cdot \nabla)u$ in the
Navier-Stokes subsystem of \dref{334451.1fghyuisda} bring about a great challenge to the study of system
\dref{334451.1fghyuisda}. To deal with these difficulties, according to the ideas in  \cite{Winkler51215} (see also \cite{Winklerssdd51215,Kegssddsddfff00,Winklerpejoevsssdd793}), we first consider the approximate problems given by
\begin{equation}
 \left\{\begin{array}{ll}
   n_{\varepsilon t}+u_{\varepsilon}\cdot\nabla n_{\varepsilon}=\Delta n_{\varepsilon}-\nabla\cdot(n_{\varepsilon}S_\varepsilon(x, n_{\varepsilon}, c_{\varepsilon})\nabla c_{\varepsilon})-n_{\varepsilon}m_{\varepsilon},\quad
x\in \Omega, t>0,\\
       c_{\varepsilon t}+u_{\varepsilon}\cdot\nabla c_{\varepsilon}=\Delta c_{\varepsilon}-c_{\varepsilon}+m_{\varepsilon},\quad
x\in \Omega, t>0,\\
m_{\varepsilon t}+u_{\varepsilon}\cdot\nabla m_{\varepsilon}=\Delta m_{\varepsilon}-n_{\varepsilon}m_{\varepsilon},\quad
x\in \Omega, t>0,\\
u_{\varepsilon t}+\nabla P_{\varepsilon}=\Delta u_{\varepsilon}-\kappa (Y_{\varepsilon}u_{\varepsilon} \cdot \nabla)u_{\varepsilon}+(n_{\varepsilon}+m_{\varepsilon})\nabla \phi,\quad
x\in \Omega, t>0,\\
\nabla\cdot u_{\varepsilon}=0,\quad
x\in \Omega, t>0,\\
 \disp{\nabla n_{\varepsilon}\cdot\nu=\nabla c_{\varepsilon}\cdot\nu=\nabla m_{\varepsilon}\cdot\nu=0,u_{\varepsilon}=0,\quad
x\in \partial\Omega, t>0,}\\
\disp{n_{\varepsilon}(x,0)=n_0(x),c_{\varepsilon}(x,0)=c_0(x),m_{\varepsilon}(x,0)=m_0(x),u_{\varepsilon}(x,0)=u_0(x)},\quad
x\in \Omega,\\
 \end{array}\right.\label{1.1fghyuisda}
\end{equation}
where
\begin{equation}
\begin{array}{ll}
S_\varepsilon (x, n, c) = \rho_\varepsilon(x)\chi_\varepsilon (n)S(x, n, c),~~ x\in\bar{\Omega},~~n\geq0,~~c\geq0
 \end{array}\label{3.10gghhjuuloollyuigghhhyy}
\end{equation}
and
\begin{equation}
 \begin{array}{ll}
 Y_\varepsilon w := (1 + \varepsilon A)^{-1}w ~~~~\mbox{for all}~~ w\in L^2_{\sigma}(\Omega)
 \end{array}\label{aasddffgg1.1fghyuisda}
\end{equation}
is the standard Yosida approximation.
Here $(\rho_{\varepsilon} )_{\varepsilon\in(0,1)} \in C^\infty_0 (\Omega)$ and $(\chi_{\varepsilon}  )_{\varepsilon\in(0,1)} \in C^\infty_0 ([0,\infty))$
  are a family of  functions which satisfy
   $$0\leq\rho_\varepsilon \leq 1~~~\mbox{in}~~\Omega,~~\rho_\varepsilon \nearrow1~~~~\mbox{in}~~\Omega~~\mbox{as}~~\varepsilon\searrow0$$
   and
   $$0\leq\chi_\varepsilon \leq 1~~~\mbox{in}~~[0,\infty),~~\chi_\varepsilon \nearrow1~~\mbox{in}~~[0,\infty)~~\mbox{as}~~\varepsilon\searrow0.$$

Without essential difficulty,
the local existence of approximate solutions to \dref{1.1fghyuisda} can be easily proved according to the corresponding
procedure in Lemma 2.1 of \cite{Winkler51215}  (see also \cite{Winkler11215} and Lemma 2.1 of \cite{Painter55677}). Therefore, we give the following lemma without proof.
\begin{lemma}\label{lemma70}
Assume
that $\varepsilon\in(0,1).$
%
Then there exist $T_{max,\varepsilon}\in  (0,\infty]$ and
a classical solution $(n_\varepsilon, c_\varepsilon, m_\varepsilon,u_\varepsilon, P_\varepsilon)$ of \dref{1.1fghyuisda} in
$\Omega\times(0, T_{max,\varepsilon})$ such that
\begin{equation}
 \left\{\begin{array}{ll}
 n_\varepsilon\in C^0(\bar{\Omega}\times[0,T_{max,\varepsilon}))\cap C^{2,1}(\bar{\Omega}\times(0,T_{max,\varepsilon})),\\
  c_\varepsilon\in  C^0(\bar{\Omega}\times[0,T_{max,\varepsilon}))\cap C^{2,1}(\bar{\Omega}\times(0,T_{max,\varepsilon})),\\
   m_\varepsilon\in  C^0(\bar{\Omega}\times[0,T_{max,\varepsilon}))\cap C^{2,1}(\bar{\Omega}\times(0,T_{max,\varepsilon})),\\
  u_\varepsilon\in  C^0(\bar{\Omega}\times[0,T_{max,\varepsilon}))\cap C^{2,1}(\bar{\Omega}\times(0,T_{max,\varepsilon})),\\
  P_\varepsilon\in  C^{1,0}(\bar{\Omega}\times(0,T_{max,\varepsilon})),\\
   \end{array}\right.\label{1.1ddfghyuisda}
\end{equation}
 classically solving \dref{1.1fghyuisda} in $\Omega\times[0,T_{max,\varepsilon})$.
%
Moreover,  $n_\varepsilon,c_\varepsilon$ and $m_\varepsilon$ are nonnegative in
$\Omega\times(0, T_{max,\varepsilon})$, and
\begin{equation}
\|n_\varepsilon(\cdot, t)\|_{L^\infty(\Omega)}+\|c_\varepsilon(\cdot, t)\|_{W^{1,\infty}(\Omega)}+\|m_\varepsilon(\cdot, t)\|_{W^{1,\infty}(\Omega)}+\|A^\gamma u_\varepsilon(\cdot, t)\|_{L^{2}(\Omega)}\rightarrow\infty~~ \mbox{as}~~ t\nearrow T_{max,\varepsilon},
\label{1.163072x}
\end{equation}
where $\gamma$ is given by \dref{ccvvx1.731426677gg}.
\end{lemma}

\begin{lemma}(\cite{Horstmann791,Winkler792,Zhengddfggghjjkk1})\label{llssdrffmmggnnccvvccvvkkkkgghhkkllvvlemma45630}
The Stokes operator $A$  denotes the realization of the Stokes operator under homogeneous
Dirichlet boundary conditions in the solenoidal subspace $L^2_{\sigma}(\Omega)$ of $L^2(\Omega)$.
%
 Let $\mathcal{P}: L^p (\Omega)\rightarrow L^p_{\sigma}(\Omega)$ stand  for the Helmholtz projection in $L^p (\Omega).$ Then there exist positive constants
 $\kappa_i(i=1,\ldots,3)$ such that
\begin{equation}
\| e^{-tA}\mathcal{P}\varphi\|_{L^p(\Omega)} \leq \kappa_1(\Omega)t^{-\gamma}\|\varphi\|_{L^q(\Omega)} ~~~\mbox{for all}~~~ t > 0~~~\mbox{and any}~~~\varphi\in L^{q}(\Omega)
\label{1.1ddfghssdddyddffssddduisda}
\end{equation}
as well as
\begin{equation}
\| e^{-tA}\mathcal{P}\nabla\cdot\varphi\|_{L^p(\Omega)} \leq \kappa_2(\Omega)t^{-\frac{1}{2}-\frac{3}{2}(\frac{1}{q}-\frac{1}{p})}\|\varphi\|_{L^{q}(\Omega)} ~~~\mbox{for all}~~~ t > 0~~~\mbox{and any}~~~\varphi\in L^{q}(\Omega)
\label{1.1ddfghssdddyssddduisda}
\end{equation}
and
\begin{equation}
\| e^{-tA}\mathcal{P}\varphi\|_{L^p(\Omega)} \leq \kappa_3(\Omega)\|\varphi\|_{L^p(\Omega)} ~~~\mbox{for all}~~~ t > 0~~~\mbox{and any}~~~\varphi\in L^{p}(\Omega).
\label{1.1ddddfghhfghssdddyssddduisda}
\end{equation}

%
\end{lemma}

%
%

Invoking the divergence free of the fluid and the homogeneous Neumann boundary conditions
on $n_\varepsilon, m_\varepsilon$ and  $c_\varepsilon$, we can establish
the following basic estimates by using the maximum principle to the second and third equations. The proof of this lemma is very similar to that of Lemmas 2.2 and 2.6  of  \cite{Tao41215} (see also Lemma 3.2 of \cite{Wangssddss21215}), so we omit its proof here

\begin{lemma}\label{fvfgsdfggfflemma45}
There exists 
$\lambda > 0$ 
such that the solution of \dref{1.1fghyuisda} satisfies
%
%
\begin{equation}
\frac{d}{dt}\int_{\Omega}c_{\varepsilon}(\cdot,t)\leq0,~~~\mbox{and}~~~\frac{d}{dt}\int_{\Omega}m_{\varepsilon}(\cdot,t)\leq0~~\mbox{for all}~~ t\in(0, T_{max,\varepsilon}),
\label{ddfgczhhhh2.5ghjjjssddju48cfg924ghyuji}
\end{equation}

\begin{equation}
\|c_{\varepsilon}(\cdot,t)\|_{L^\infty(\Omega)}+\|m_{\varepsilon} (\cdot,t)\|_{L^\infty(\Omega)}\leq \lambda~~\mbox{for all}~~ t\in(0, T_{max,\varepsilon}),
\label{ddfgczhhhh2.5ghju48cfg924ghyuji}
\end{equation}
\begin{equation}
\int_{\Omega}{n_{\varepsilon} }-\int_{\Omega}{m_{\varepsilon}}= \int_{\Omega}{n_0}-\int_{\Omega}{m_0}~~\mbox{for all}~~ t\in(0, T_{max,\varepsilon}),
\label{sssddfgczhhhh2.5ghju48cfg924ghyuji}
\end{equation}
\begin{equation}
\|c_{\varepsilon}(\cdot,t)\|_{L^2(\Omega)}^2+2\int_0^{t}\int_{\Omega}{|\nabla c_{\varepsilon}|^{2}}\leq \lambda~~\mbox{for all}~~ t\in(0, T_{max,\varepsilon}),
\label{ddczhjjjj2.5ghju48cfg9ssdd24}
\end{equation}
\begin{equation}
\int_0^{t}\int_{\Omega}{n_{\varepsilon}m_{\varepsilon}}dxds\leq \lambda~~\mbox{for all}~~ t\in(0, T_{max,\varepsilon})
\label{ddczhjjjj2.5ghxxccju48cfg9ssdd24}
\end{equation}
as well as
\begin{equation}
\frac{1}{2}\int_0^{t}\int_{\Omega}{|\nabla m_{\varepsilon}|^2}dxds\leq \frac{1}{2}\int_{\Omega}{m_0^2}~~\mbox{for all}~~ t\in(0, T_{max,\varepsilon})
\label{ddczhjjjj2.5ghxxssddccju48cfg9ssdd24}
\end{equation}
and
\begin{equation}
\int_0^{t}\int_{\Omega}{|\nabla c_{\varepsilon}|^{2}}\leq \lambda~~\mbox{for all}~~ t\in(0, T_{max,\varepsilon}).
\label{ddczhjjjj2.5ghju48cfgffff924}
\end{equation}
%
%
\end{lemma}

 For simplicity, here and hereafter, we take the notations
   \begin{equation}
   C_S:=\sup_{0\leq s\leq \|c_{0}\|_{L^\infty(\Omega)}}S(s)
  \label{hnjmssddaqwswddaassffssff3.10deerfgghhjdddfgggjjuuloollgghhhyhh}
\end{equation}
by using  \dref{ddfgczhhhh2.5ghju48cfg924ghyuji} and \dref{x1.73142vghf48gg}.

\section{ A-priori estimates }
In this section we want to ensure that the time-local solutions obtained in Lemma \ref{lemma70} are in fact global
solutions. To this end, for any $p>1,$ we firstly  obtain boundedness of $n_\varepsilon$ in $L^p (\Omega)$ under the assumption that $\alpha>0.$
Inspired by the weighted estimate argument developed in \cite{Winkler2233444ssdff51215} (see also \cite{Winkler61215,Winklerpejoevsssdd793}), 
we shall invoke
a weight function
$g(c_\varepsilon)$ which is uniformly bounded from above and below by positive constants.
Before deriving the uniform of $L^p$ norm of $n_\varepsilon$, let us first recalling the well-known facts for $g(c_\varepsilon)$.

\begin{lemma}\label{fvfgfflemma45}
Let 
$$g(s) = e^{\beta s^2},~~~\mbox{for any}~~s\in(0,\|c_{0}\|_{L^\infty(\Omega)}],$$
where \begin{equation}
\beta=\frac{1}{8\|c_0\|_{L^\infty(\Omega)}^2}
\label{czfvgbdfgg2.5ddffghddffsddffhjuyddfffdddddddggfffuccvviihjj}
\end{equation}
and $C_S$ is given by \dref{hnjmssddaqwswddaassffssff3.10deerfgghhjdddfgggjjuuloollgghhhyhh}.
Then for any $s\in(0,\|c_{0}\|_{L^\infty(\Omega)}]$,
\begin{equation}
\begin{array}{rl}
1\leq g(s)\leq\mu_0:=e^{\frac{1}{8}}
\end{array}
\label{czfvgb2.5ddffghsddffhjuyddfffdddddddggfffuccvviihjj}
\end{equation}
and
\begin{equation}g'(s)\leq \mu_1:= \frac{1}{4\|c_0\|_{L^\infty}}e^{\frac{1}{8}}.
\label{1111czfvgb2.ghhjkl5ghsddffhjuyddfffudddfccvviihjj}
\end{equation}

\end{lemma}
\begin{proof}
Obviously, \dref{czfvgb2.5ddffghsddffhjuyddfffdddddddggfffuccvviihjj} holds.
On the other hand, a direct computation shows
\begin{equation}
\begin{array}{rl}g'(s)=&2\beta se^{\beta s^2}.\\
\end{array}
\label{czfvgb2.5ghsddffhjuyddfffddfffuccvviihjj}
\end{equation}
This combined with the fact that $\beta=\frac{1}{8\|c_0\|_{L^\infty(\Omega)}^2}$ implies 
\dref{1111czfvgb2.ghhjkl5ghsddffhjuyddfffudddfccvviihjj}.
\end{proof}

\begin{lemma}\label{lemmaghjffggssddgghhmk4563025xxhjklojjkkk}
Let $\alpha>0$. 
Then  for any $p>1$, there exists $C>0$ 
such that the solution of \dref{1.1fghyuisda} satisfies
\begin{equation}
\begin{array}{rl}
&\disp{\int_{\Omega} n_{\varepsilon}  ^p\leq C~~~\mbox{for all}~~ t\in (0, T_{max,\varepsilon}).}\\
\end{array}
\label{czfvgb2.5ghhjuyuccvviihjj}
\end{equation}
\end{lemma}
\begin{proof}
Firstly,  we define a functional
$$L(n_{\varepsilon} ,c_{\varepsilon})=\frac{1}{p}\int_{\Omega}n_{\varepsilon}^pg(c_{\varepsilon}),$$
where $p>\max\{1,2\alpha\}$, $g(c_{\varepsilon})=e^{\beta c ^2}$
and  $\beta$ is the same as \dref{czfvgbdfgg2.5ddffghddffsddffhjuyddfffdddddddggfffuccvviihjj}.
Using the first two equations in \dref{1.1fghyuisda}, we find:
\begin{equation}
\begin{array}{rl}
&\disp{\frac{d}{dt}L(n_{\varepsilon} ,c_{\varepsilon})}\\
=&\disp{\int_{\Omega}n_{\varepsilon}  ^{p-1}n_{\varepsilon t} g(c_{\varepsilon})+\frac{1}{p}\int_{\Omega}n_{\varepsilon}  ^pg'(c_{\varepsilon})c_{\varepsilon t}}\\
=&\disp{\int_{\Omega}n_{\varepsilon}  ^{p-1}g(c_{\varepsilon})(\Delta  n_{\varepsilon} -\nabla\cdot(n_{\varepsilon} S_{\varepsilon} (x, n_{\varepsilon} , c_{\varepsilon})\nabla c_{\varepsilon}) -u_{\varepsilon} \cdot\nabla n_{\varepsilon} -n_{\varepsilon} m _{\varepsilon})}\\
&\disp{+\frac{1}{p}\int_{\Omega}n_{\varepsilon}  ^pg'(c_{\varepsilon})(\Delta c_{\varepsilon}-c_{\varepsilon} +m_{\varepsilon} -u \cdot\nabla c_{\varepsilon})}\\
=&\disp{-\frac{1}{p}
\int_{\Omega}n_{\varepsilon}  ^pg'(c_{\varepsilon})u_{\varepsilon} \cdot\nabla c_{\varepsilon}-\int_{\Omega}n_{\varepsilon}  ^{p-1}g(c_{\varepsilon})u_{\varepsilon} \cdot\nabla n_{\varepsilon} }\\
&\disp{-\frac{1}{p}\int_{\Omega}n_{\varepsilon}  ^pg'(c_{\varepsilon})c_{\varepsilon}-\int_{\Omega}n_{\varepsilon}  ^pg(c_{\varepsilon})m_{\varepsilon} +\frac{1}{p}\int_{\Omega}n_{\varepsilon}  ^pg'(c_{\varepsilon})m_{\varepsilon} }\\
&+\disp{\int_{\Omega}n_{\varepsilon}  ^{p-1}g(c_{\varepsilon})\Delta n_{\varepsilon} -\int_{\Omega}n_{\varepsilon}  ^{p-1}g(c_{\varepsilon})\nabla\cdot(n_{\varepsilon} S_{\varepsilon}(x, n_{\varepsilon}, c_{\varepsilon})\nabla c_{\varepsilon})
}\\
&\disp{+\frac{1}{p}\int_{\Omega}n_{\varepsilon}  ^pg'(c_{\varepsilon})\Delta c_{\varepsilon} }\\
=:&\disp{\sum_{i=1}^8I_i~~~\mbox{for all}~~ t\in (0, T_{max,\varepsilon}).}
\end{array}
\label{55hhjjcffghhhjkkllz2dddd.5}
\end{equation}
In the following, we will estimate the right-hand sides of \dref{55hhjjcffghhhjkkllz2dddd.5} one by one. To this end, firstly, applying the elementary calculus identity
$$\frac{1}{p}n_{\varepsilon}  ^pg'(c_{\varepsilon})\nabla c_{\varepsilon} +n_{\varepsilon}  ^{p-1}g(c_{\varepsilon})\nabla n_{\varepsilon} =\frac{1}{p}\nabla(n_{\varepsilon}  ^pg(c_{\varepsilon}))$$
and the  fact that
$$\nabla\cdot u_{\varepsilon}=0,\quad
x\in \Omega, t>0,\\$$
we once more integrate by parts to find that
\begin{equation}I_1+I_2=-
\frac{1}{p}\int_{\Omega}n_{\varepsilon}  ^pg'(c_{\varepsilon})u_{\varepsilon} \cdot\nabla c_{\varepsilon}-\int_{\Omega}n_{\varepsilon}  ^{p-1}g(c_{\varepsilon})u_{\varepsilon} \cdot\nabla n_{\varepsilon} =0
\label{55hhjjcfddfffddghhhjkkllz2dddd.5}
\end{equation}
by using $u_{\varepsilon} =0,
x\in \partial\Omega, t>0$.
Next,  
we derive from  the non-negativity of $g',g$ (see \dref{czfvgb2.5ghsddffhjuyddfffddfffuccvviihjj} and \dref{czfvgb2.5ddffghsddffhjuyddfffdddddddggfffuccvviihjj}), $c_{\varepsilon} ,m_{\varepsilon} $ and $n_{\varepsilon} $ that
\begin{equation}I_3+I_4=-\frac{1}{p}\int_{\Omega}n_{\varepsilon}  ^pg'(c_{\varepsilon})c_{\varepsilon}-\int_{\Omega}n_{\varepsilon}  ^pg(c_{\varepsilon})m_{\varepsilon} \leq0,
\label{55hhjjcfddfffddghhhjkkssdddsssllz2dddd.5}
\end{equation}
which combined with \dref{55hhjjcffghhhjkkllz2dddd.5} and \dref{55hhjjcfddfffddghhhjkkllz2dddd.5} yields
\begin{equation}
\begin{array}{rl}
&\disp{\frac{d}{dt}L(n_{\varepsilon} ,c_{\varepsilon})}\\
\leq&\disp{\int_{\Omega}n_{\varepsilon}  ^{p-1}g(c_{\varepsilon})\Delta n_{\varepsilon} -\int_{\Omega}n_{\varepsilon}  ^{p-1}g(c_{\varepsilon})\nabla\cdot(n_{\varepsilon} S_{\varepsilon}(x, n_{\varepsilon}, c_{\varepsilon})\nabla c_{\varepsilon})+\frac{1}{p}\int_{\Omega}n_{\varepsilon}  ^pg'(c_{\varepsilon})\Delta c_{\varepsilon} }\\
&+\disp{\frac{1}{p}\int_{\Omega}n_{\varepsilon}  ^pg'(c_{\varepsilon})m_{\varepsilon} .}\\
\end{array}
\label{55hhjjcffghhhjkddffkllz2dddd.5}
\end{equation}
Now we proceed to estimate the fourth term on the right-hand side herein by using
\dref{ddfgczhhhh2.5ghju48cfg924ghyuji} and \dref{1111czfvgb2.ghhjkl5ghsddffhjuyddfffudddfccvviihjj} to find that
\begin{equation}
\begin{array}{rl}
\disp{\frac{1}{p}\int_{\Omega}n_{\varepsilon}  ^pg'(c_{\varepsilon})m_{\varepsilon} }\leq&\disp{
\frac{1}{p}\mu_1\lambda\int_{\Omega}n_{\varepsilon}  ^p,}\\
\end{array}
\label{55hhjjcffghhhjkdddddssdddffkllz2dddd.5}
\end{equation}
where $\mu_1$ is the same as \dref{1111czfvgb2.ghhjkl5ghsddffhjuyddfffudddfccvviihjj}.

Now we estimate the term $\int_{\Omega}n_{\varepsilon}  ^{p-1}g(c_{\varepsilon})\Delta  n_{\varepsilon} $ and $\frac{1}{p}\int_{\Omega}n_{\varepsilon}  ^pg'(c_{\varepsilon})\Delta c_{\varepsilon} $
in the right hand side of \dref{55hhjjcffghhhjkddffkllz2dddd.5}.  In fact, we once more integrate by parts to see  that
\begin{equation}
\begin{array}{rl}
&\disp{\int_{\Omega}n_{\varepsilon}  ^{p-1}g(c_{\varepsilon})\Delta  n_{\varepsilon} +\frac{1}{p}\int_{\Omega}n_{\varepsilon}  ^pg'(c_{\varepsilon})\Delta c_{\varepsilon} }\\
=&\disp{-(p-1)\int_{\Omega}n_{\varepsilon} ^{p-2}g(c_{\varepsilon}) |\nabla n_{\varepsilon} |^2 - \int_{\Omega}n_{\varepsilon}  ^{p-1}g'(c_{\varepsilon}) \nabla n_{\varepsilon} \cdot \nabla c_{\varepsilon} }\\
&\disp{-\frac{1}{p}\int_{\Omega}n_{\varepsilon}  ^pg''(c_{\varepsilon})|\nabla c_{\varepsilon} |^2-\int_{\Omega}n_{\varepsilon}  ^{p-1}g'(c_{\varepsilon})\nabla c_{\varepsilon} \cdot \nabla n_{\varepsilon} }\\
\leq&\disp{-(p-1)\int_{\Omega}n_{\varepsilon} ^{p-2}g(c_{\varepsilon}) |\nabla n_{\varepsilon} |^2 -\frac{1}{p}\int_{\Omega}n_{\varepsilon}  ^pg''(c_{\varepsilon})|\nabla c_{\varepsilon} |^2 }\\
&\disp{+2\int_{\Omega}n_{\varepsilon}  ^{p-1}g'(c_{\varepsilon})|\nabla n_{\varepsilon} || \nabla c_{\varepsilon} |}\\
\leq&\disp{-\frac{3(p-1)}{4}\int_{\Omega}n_{\varepsilon} ^{p-2}g(c_{\varepsilon}) |\nabla n_{\varepsilon} |^2 -\frac{1}{p}\int_{\Omega}n_{\varepsilon}  ^pg''(c_{\varepsilon})|\nabla c_{\varepsilon} |^2 }\\
&\disp{+\frac{4}{p-1}\int_{\Omega}n_{\varepsilon}  ^p\frac{g'(c_{\varepsilon})^2}{g(c_{\varepsilon})}| \nabla c_{\varepsilon} |^2}\\
\end{array}
\label{55hhjjcffghdrgddffffhjkhhjkddffkllz2dssdsdddsdddddd.5}
\end{equation}
by using the Young inequality.
Next,  recall \dref{x1.73142vghf48gg}, 
 we can estimate second term on the right-hand side of \dref{55hhjjcffghhhjkddffkllz2dddd.5} as follows:
 \begin{equation}
\begin{array}{rl}
&\disp{-\int_{\Omega}n_{\varepsilon}  ^{p-1}g(c_{\varepsilon})\nabla\cdot(n_{\varepsilon} S_{\varepsilon}(x, n_{\varepsilon}, c_{\varepsilon})\nabla c_{\varepsilon})}\\
=&\disp{(p-1)\int_{\Omega}n_{\varepsilon}  ^{p-1}g(c_{\varepsilon})S (x, n_{\varepsilon} , c_{\varepsilon})\nabla c_{\varepsilon} \cdot\nabla n_{\varepsilon} +\int_{\Omega}n_{\varepsilon}  ^pg'(c_{\varepsilon})S (x, n_{\varepsilon} , c_{\varepsilon})\nabla c_{\varepsilon} \cdot\nabla c_{\varepsilon} }\\
\leq&\disp{(p-1)C_S\int_{\Omega}n_{\varepsilon} ^{p-1-\alpha}g(c_{\varepsilon})|\nabla c_{\varepsilon} ||\nabla n_\varepsilon |+C_S\int_{\Omega}n_\varepsilon ^{p-\alpha}g'(c_{\varepsilon})|\nabla c_{\varepsilon} |^2}\\
\leq&\disp{\frac{p-1}{4}\int_{\Omega}n_{\varepsilon} ^{p-2} g(c_{\varepsilon})|\nabla n_{\varepsilon} |^2+ C_S^2(p-1)\int_{\Omega}n_\varepsilon ^{p-2\alpha} g(c_{\varepsilon})|\nabla c_{\varepsilon} |^2}\\
&\disp{+C_S\int_{\Omega}n_\varepsilon ^{p-\alpha}g'(c_{\varepsilon})|\nabla c_{\varepsilon} |^2,}\\
\end{array}
\label{55hhjjcffghhhjkdssdddffkllz2dddd.5}
\end{equation}
where in the last inequality, we have used the Young inequality.
Now, collecting \dref{55hhjjcffghhhjkddffkllz2dddd.5}--\dref{55hhjjcffghhhjkdssdddffkllz2dddd.5}, we may have
\begin{equation}
\begin{array}{rl}
&\disp{\frac{d}{dt}L(n_{\varepsilon} ,c_{\varepsilon})+\frac{p-1}{2}\int_{\Omega}n_{\varepsilon} ^{p-2}g(c_{\varepsilon}) |\nabla n_{\varepsilon} |^2 }\\
&+\disp{\int_{\Omega}n_{\varepsilon}  ^p\left[\frac{1}{p}g''(c_{\varepsilon})-\frac{4}{p-1}\frac{g'(c_{\varepsilon})^2}{g(c_{\varepsilon})}-
 C_S^2(p-1)n_{\varepsilon} ^{-2\alpha}g(c_{\varepsilon})-C_Sg'(c_{\varepsilon})n_{\varepsilon} ^{-\alpha}\right]|\nabla c_{\varepsilon} |^2 }\\
\leq&\disp{\frac{1}{p}\mu_1\lambda\int_{\Omega}n_{\varepsilon}  ^p,}\\
\end{array}
\label{55hhjjcffghhhjkddffkllz2ddddsddll.5}
\end{equation}
whence returning to the definition of $g(c_{\varepsilon})$ we conclude that
\begin{equation}
\begin{array}{rl}
&\disp{\frac{d}{dt}L(n_{\varepsilon} ,c_{\varepsilon})+\frac{p-1}{2}\int_{\Omega}n_{\varepsilon} ^{p-2}g(c_{\varepsilon}) |\nabla n_{\varepsilon} |^2 }\\
\leq&\disp{-\int_{\Omega}n_{\varepsilon}  ^pe^{\beta c_{\varepsilon} ^2}\left[\frac{1}{p}(4\beta^2 c_{\varepsilon} ^2+2\beta)-\frac{16}{p-1}\beta^2 c_{\varepsilon} ^2-
 C_S^2(p-1)n_{\varepsilon} ^{-2\alpha}-2C_S\beta  c_{\varepsilon}  n_{\varepsilon} ^{-\alpha}\right]|\nabla c_{\varepsilon} |^2}\\
 &\disp{ +\frac{1}{p}\mu_1\lambda\int_{\Omega}n_{\varepsilon}  ^p}\\
 \leq&\disp{-\int_{\Omega}n_{\varepsilon}  ^pe^{\beta c_{\varepsilon} ^2}\left[\frac{2\beta}{p}-\frac{16}{p-1}\beta^2 c_{\varepsilon} ^2-
 C_S^2(p-1)n_{\varepsilon} ^{-2\alpha}-2C_S\beta  c_{\varepsilon}  n_{\varepsilon} ^{-\alpha}\right]|\nabla c_{\varepsilon} |^2}\\
 &\disp{  +\frac{1}{p}\mu_1\lambda\int_{\Omega}n_{\varepsilon}  ^p}\\
 \leq&\disp{-\int_{\Omega}n_{\varepsilon}  ^pe^{\beta c_{\varepsilon} ^2}\left[\frac{2\beta}{p}-\frac{16}{p-1}\beta^2 \|c_0\|_{L^\infty(\Omega)}^2-
 C_S^2(p-1)n_{\varepsilon} ^{-2\alpha}-2C_S\beta  \|c_0\|_{L^\infty(\Omega)} n_{\varepsilon} ^{-\alpha}\right]|\nabla c_{\varepsilon} |^2}\\
 &\disp{ +\frac{1}{p}\mu_1\lambda\int_{\Omega}n_{\varepsilon}  ^p}\\
 \leq&\disp{-\int_{\Omega}n_{\varepsilon}  ^pe^{\beta c_{\varepsilon} ^2}\left[\frac{2\beta}{p}-\frac{8}{p}\beta^2 \|c_0\|_{L^\infty(\Omega)}^2-
 C_S^2pn_{\varepsilon} ^{-2\alpha}-2C_S\beta  \|c_0\|_{L^\infty(\Omega)} n_{\varepsilon} ^{-\alpha}\right]|\nabla c_{\varepsilon} |^2}\\
 &\disp{ +\frac{1}{p}\mu_1\lambda\int_{\Omega}n_{\varepsilon}  ^p}\\
\end{array}
\label{55hhjjcffghhhjkddffkllz2ddddsddll.ssddd5}
\end{equation}
by using $p>\max\{1,2\alpha\}.$
In view of $\beta=\frac{1}{8\|c_0\|_{L^\infty(\Omega)}^2}$, thus, \dref{55hhjjcffghhhjkddffkllz2ddddsddll.ssddd5} implies  that
\begin{equation}
\begin{array}{rl}
&\disp{\frac{d}{dt}L(n_{\varepsilon} ,c_{\varepsilon})+\frac{p-1}{2}\int_{\Omega}n_{\varepsilon} ^{p-2}g(c_{\varepsilon}) |\nabla n_{\varepsilon} |^2 }\\
 \leq&\disp{-\int_{\Omega}n_{\varepsilon}  ^pe^{\beta c ^2}\left[\frac{\beta}{p}-
 C_S^2pn_{\varepsilon} ^{-2\alpha}-2C_S\beta  \|c_0\|_{L^\infty(\Omega)} n_{\varepsilon} ^{-\alpha}\right]|\nabla c_{\varepsilon} |^2 +\frac{1}{p}\mu_1\lambda\int_{\Omega}n_{\varepsilon}  ^p.}\\
\end{array}
\label{55hhjjcffghhhjkddffkllz2ddddsddll.ssdddssd5}
\end{equation}
On the other hand,  due to $\alpha>0,$ we may have
$$\lim_{s\rightarrow+\infty}[
 C_S^2ps^{-2\alpha}+2C_S\beta  \|c_0\|_{L^\infty(\Omega)} s^{-\alpha}]=0,$$
 so that, there exists $\eta_0>0$, such that for any $s>\eta_0$,
 $$[C_S^2ps^{-2\alpha}+2C_S\beta  \|c_0\|_{L^\infty(\Omega)} s^{-\alpha}]<\frac{\beta}{2p}.$$
Therefore, by  some basic
calculation, we derive  from  \dref{ddfgczhhhh2.5ghju48cfg924ghyuji} that
 \begin{equation}
\begin{array}{rl}
&\disp{\int_{\Omega}n_{\varepsilon}  ^pe^{\beta c_{\varepsilon} ^2}\left[C_S^2pn_{\varepsilon} ^{-2\alpha}+2C_S\beta  \|c_0\|_{L^\infty(\Omega)} n_{\varepsilon} ^{-\alpha}\right]|\nabla c_{\varepsilon} |^2  }\\
 \leq&\disp{\int_{n_{\varepsilon} >\eta_0}n_{\varepsilon}  ^pe^{\beta c_{\varepsilon} ^2}\left[C_S^2pn_{\varepsilon} ^{-2\alpha}+2C_S\beta  \|c_0\|_{L^\infty(\Omega)} n_{\varepsilon} ^{-\alpha}\right]|\nabla c_{\varepsilon} |^2}\\
 &\disp{+\int_{n_{\varepsilon} \leq\eta_0}n_{\varepsilon}  ^pe^{\beta c_{\varepsilon} ^2}\left[C_S^2pn_{\varepsilon} ^{-2\alpha}+2C_S\beta  \|c_0\|_{L^\infty(\Omega)} n ^{-\alpha}\right]|\nabla c_{\varepsilon} |^2 }\\
 \leq&\disp{\int_{n_{\varepsilon} >\eta_0}n_{\varepsilon}  ^pe^{\beta c_{\varepsilon} ^2}\frac{\beta}{2p}|\nabla c_{\varepsilon} |^2+\int_{n_{\varepsilon} \leq\eta_0}n_{\varepsilon}  ^pe^{\beta c_{\varepsilon} ^2}\left[C_S^2pn_{\varepsilon} ^{-2\alpha}+2C_S\beta  \|c_0\|_{L^\infty(\Omega)} n ^{-\alpha}\right]|\nabla c_{\varepsilon} |^2 }\\
 \leq&\disp{\int_{\Omega}n_{\varepsilon}  ^pe^{\beta c_{\varepsilon} ^2}\frac{\beta}{2p}|\nabla c_{\varepsilon} |^2+\int_{n_{\varepsilon} \leq\eta_0}e^{\beta c_{\varepsilon} ^2}\left[C_S^2pn_\varepsilon ^{p-2\alpha}+2C_S\beta  \|c_0\|_{L^\infty(\Omega)} n_{\varepsilon} ^{p-\alpha}\right]|\nabla c_{\varepsilon} |^2 }\\
 \leq&\disp{\int_{\Omega}n_{\varepsilon}  ^pe^{\beta c_{\varepsilon} ^2}\frac{\beta}{2p}|\nabla c_{\varepsilon} |^2+\gamma_0\int_{n_{\varepsilon} \leq\eta_0}|\nabla c_{\varepsilon} |^2 }\\
 \leq&\disp{\int_{\Omega}n_{\varepsilon}  ^pe^{\beta c_{\varepsilon} ^2}\frac{\beta}{2p}|\nabla c_{\varepsilon} |^2+\gamma_0\int_{\Omega}|\nabla c_{\varepsilon} |^2 }\\
\end{array}
\label{55hhjjcffghhhjkddfssddfkllz2ddddsddll.sssddsdddssd5}
\end{equation}
with
$$\gamma_0=e^{\beta \|c_0\|_{L^\infty(\Omega)}^2}\left[C_S^2p\eta_0^{p-2\alpha}+2C_S\beta  \|c_0\|_{L^\infty(\Omega)} \eta_0^{p-\alpha}\right]$$
by using \dref{ddfgczhhhh2.5ghju48cfg924ghyuji} and $p>\max\{1,2\alpha\}$.
Substituting  \dref{55hhjjcffghhhjkddfssddfkllz2ddddsddll.sssddsdddssd5} into \dref{55hhjjcffghhhjkddffkllz2ddddsddll.ssdddssd5}, we have
%
 \begin{equation}
\begin{array}{rl}
&\disp{\frac{d}{dt}L(n_{\varepsilon} ,c_{\varepsilon})+\frac{p-1}{2}\int_{\Omega}n_{\varepsilon} ^{p-2}g(c_{\varepsilon}) |\nabla n_{\varepsilon} |^2 +\int_{\Omega}n_{\varepsilon}  ^pe^{\beta c_{\varepsilon} ^2}\frac{\beta}{2p}|\nabla c_{\varepsilon} |^2}\\
 \leq&\disp{\gamma_0\int_{\Omega}|\nabla c_{\varepsilon} |^2 +\frac{1}{p}\mu_1\lambda\int_{\Omega}n_{\varepsilon}  ^p.}\\
\end{array}
\label{55hhjjcffghhhjkddffkllz2ddddsddll.sssddsdddssd5}
\end{equation}

Now, according to \dref{ddfgczhhhh2.5ghju48cfg924ghyuji},  we therefore obtain on using the  Gagliardo-Nirenberg inequality that
\begin{equation}
\begin{array}{rl}
\disp \int_{\Omega}n_{\varepsilon} ^{p+\frac{1}{3}}=&\disp{ \|n_{\varepsilon} ^{\frac{p}{2}}\|_{L^{\frac{2(3p+1)}{3p}}(\Omega)}^{\frac{2(3p+1)}{3p}}}\\
\leq&\disp{ C_1[\|\nabla n_{\varepsilon} ^{\frac{p}{2}}\|_{L^{\frac{2}{p}}(\Omega)}^{\frac{2(3p-2)}{3p-1}}\| n_{\varepsilon} ^{\frac{p}{2}}\|_{L^{\frac{2}{p}}(\Omega)}^{\frac{2(3p+1)}{3p}-\frac{2(3p-2)}{3p-1}}+
\|n_{\varepsilon} ^{\frac{p}{2}}\|_{L^{\frac{2}{p}}(\Omega)}^{\frac{2(3p+1)}{3p}}]}\\
\leq&\disp{ \frac{(p-1)}{4}\frac{4}{p^2}\|n_{\varepsilon} ^{\frac{p}{2}}\|_{L^{2}(\Omega)}^2+C_2}\\
=&\disp{ \frac{(p-1)}{4}\int_{\Omega}n_{\varepsilon} ^{p-2} |\nabla n_{\varepsilon} |^2 +C_2}\\
\leq&\disp{ \frac{(p-1)}{4}\int_{\Omega}n_{\varepsilon} ^{p-2}g(c_{\varepsilon}) |\nabla n_{\varepsilon}|^2 +C_2}\\
\end{array}
\label{55hhjjcffghhhjssddkdssddddfsdddfkllz2dddd.5}
\end{equation}
for some positive constants $C_1$ and $C_2$, where in the last inequality, we have used \dref{czfvgb2.5ddffghsddffhjuyddfffdddddddggfffuccvviihjj}.
Collecting \dref{55hhjjcffghhhjkddffkllz2ddddsddll.sssddsdddssd5} and \dref{55hhjjcffghhhjssddkdssddddfsdddfkllz2dddd.5}, we have
\begin{equation}
\begin{array}{rl}
&\disp{\frac{d}{dt}L(n_{\varepsilon} ,c_{\varepsilon})+\frac{p-1}{4}\int_{\Omega}n_{\varepsilon} ^{p-2}g(c_{\varepsilon}) |\nabla n_{\varepsilon} |^2 +\int_{\Omega}n_{\varepsilon}  ^pe^{\beta c_{\varepsilon} ^2}\frac{\beta}{2p}|\nabla c_{\varepsilon} |^2+\int_{\Omega}n_\varepsilon ^{p+\frac{1}{3}}}\\
 \leq&\disp{\gamma_0\int_{\Omega}|\nabla c_{\varepsilon} |^2 +\frac{1}{p}\mu_1\lambda\int_{\Omega}n_{\varepsilon}  ^p+C_2}\\
  \leq&\disp{\gamma_0\int_{\Omega}|\nabla c_{\varepsilon} |^2 +\frac{1}{2}\int_{\Omega}n_{\varepsilon} ^{p+\frac{1}{3}}+C_3}\\
\end{array}
\label{55hhjjcffghhhjkddffkllz2ddddsddll.ssdddssddsdddssdddsd5}
\end{equation}
by using the Young inequality.
Therefore,
\begin{equation}
\begin{array}{rl}
&\disp{\frac{d}{dt}L(n_{\varepsilon} ,c_{\varepsilon})+\frac{p-1}{4}\int_{\Omega}n_{\varepsilon} ^{p-2}g(c_{\varepsilon}) |\nabla n_{\varepsilon} |^2 +\int_{\Omega}n_{\varepsilon}  ^pe^{\beta c_{\varepsilon} ^2}\frac{\beta}{2p}|\nabla c_{\varepsilon} |^2+\frac{1}{2}\int_{\Omega}n_{\varepsilon} ^{p+\frac{1}{3}}}\\
  \leq&\disp{\gamma_0\int_{\Omega}|\nabla c_{\varepsilon} |^2 +C_3.}\\
\end{array}
\label{55hhjjcffghhhjkddffddfsdddfkllz2dddd.5}
\end{equation}
To track the time evolution of $c_\varepsilon $, testing  the second equation in \dref{1.1fghyuisda} by $c_\varepsilon$ and using $\nabla\cdot u_\varepsilon =0$ and \dref{ddfgczhhhh2.5ghju48cfg924ghyuji} yields that for some positive constant $C_4$ such that
\begin{equation}
\begin{array}{rl}
\disp\frac{1}{{2}}\disp\frac{d}{dt}\|{c_{\varepsilon} }\|^{{{2}}}_{L^{{2}}(\Omega)}+
\int_{\Omega} |\nabla c_{\varepsilon} |^2=&\disp{-\int_{\Omega} c_{\varepsilon} ^2+\int_{\Omega}m_{\varepsilon} c_{\varepsilon} }\\
\leq&\disp{-\frac{1}{2}\int_{\Omega} c_{\varepsilon} ^2+C_4~~\mbox{for all}~~ t\in(0, T_{max,\varepsilon}),}\\
\end{array}
\label{hhxxcdfvvjjcz2.5}
\end{equation}
wereafter integrating the above inequality  in time yields
\begin{equation}
\begin{array}{rl}
&\disp{\int_{t}^{t+1}\int_{\Omega}  |\nabla {c_{\varepsilon} }|^2\leq C_5~~\mbox{for all}~~ t>0}\\
\end{array}
\label{bnmbncz2.5ghhjuddfghhdddddffggyhjkklluivvbnnihjj}
\end{equation}
for some $C_5> 0$ by an integration. This yields to
\begin{equation}
\begin{array}{rl}
&\disp{\int_{t}^{t+\tau}\left[\gamma_0\int_{\Omega}| \nabla c_{\varepsilon} |^2+C_3\right]\leq C_6~~\mbox{for all}~~
t\in(0, T_{max,\varepsilon}-\tau)}\\
\end{array}
\label{bnmbncz2.5ghhjuddfghhdddddffggyhjkkllddffguivvbnnihjj}
\end{equation}
with \begin{equation}
\tau:=\min\{1,\frac{1}{6}T_{max,\varepsilon}\}.
\label{cz2.5ghju48cfg924vbhu}
\end{equation}
Finally,  \dref{bnmbncz2.5ghhjuddfghhdddddffggyhjkkllddffguivvbnnihjj} in conjunction with Lemma 2.3 of \cite{Wddffang11215} (see also \cite{Zhenddsdddddgssddsddfff00})  and \dref{55hhjjcffghhhjkddffddfsdddfkllz2dddd.5}
 establish  \dref{czfvgb2.5ghhjuyuccvviihjj}.

\end{proof}
In a straightforward manner, the estimates gained above can be seen to imply the following  $\varepsilon$-independent estimates, which plays an important role  in proving Theorem \ref{theorem3}.

%

\begin{lemma}\label{ssdddlemmaghjffggssddgghhmk4563025xxhjklojjkkk}
Let $\alpha>0$.
Then there exists $C>0$ 
such that the solution of \dref{1.1fghyuisda} satisfies
\begin{equation}
\begin{array}{rl}
&\disp{\int_{\Omega} \left(|\nabla c_{\varepsilon}| ^{2}+|\nabla m_{\varepsilon}| ^{2}\right)\leq C~~~\mbox{for all}~~ t\in (0, T_{max,\varepsilon}).}\\
\end{array}
\label{czfvgb2.5ghhjuyuccvviihjj}
\end{equation}
%
Moreover, for $t\in(0, T_{max,\varepsilon}-\tau)$, it holds that
\begin{equation}
\begin{array}{rl}
&\disp{\int_{t}^{t+\tau}\int_{\Omega} \left[   |\nabla {c_{\varepsilon} }|^4+ |\nabla {m_{\varepsilon}}|^4+|\nabla {u_{\varepsilon} }|^2+| {u_{\varepsilon} }|^{\frac{10}{3}}\right]\leq C,}\\
\end{array}
\label{bnmbncz2.5ghhjuyuivvbnnihjj}
\end{equation}
where $\tau$ is the same as \dref{cz2.5ghju48cfg924vbhu}.
\end{lemma}
\begin{proof}
We multiply the second equation in \dref{1.1fghyuisda} by $-\Delta c_{\varepsilon} $ and integrate by parts to see that
\begin{equation}
\begin{array}{rl}
&\disp\frac{1}{{2}}\disp\frac{d}{dt}\|\nabla{c_{\varepsilon} }\|^{{{2}}}_{L^{{2}}(\Omega)}+
\int_{\Omega} |\Delta c_{\varepsilon} |^2+ \int_{\Omega} | \nabla c_{\varepsilon} |^2
\\
=&\disp{-\int_{\Omega} m_{\varepsilon} \Delta c_{\varepsilon} +\int_{\Omega} (u_{\varepsilon} \cdot\nabla c_{\varepsilon})\Delta c_{\varepsilon} }
\\
=&\disp{-\int_{\Omega} m_{\varepsilon} \Delta c_{\varepsilon}-\int_{\Omega}\nabla c_{\varepsilon} \nabla (u_{\varepsilon} \cdot\nabla c_{\varepsilon})}
\\
=&\disp{-\int_{\Omega} m_{\varepsilon} \Delta c_{\varepsilon}-\int_{\Omega}\nabla c_{\varepsilon} \nabla (\nabla u_{\varepsilon} \cdot\nabla c_{\varepsilon}),}
\end{array}
\label{hhxxcsssdfvvjjczddfdddfff2.5}
\end{equation}
where we have used the fact that
$$
\disp{\int_{\Omega}\nabla c_{\varepsilon} \cdot(D^2 c_{\varepsilon} \cdot u_{\varepsilon} )
=\frac{1}{2}\int_{\Omega}  u_{\varepsilon} \cdot\nabla|\nabla c_{\varepsilon} |^2=0
~~\mbox{for all}~~ t\in(0,T_{max,\varepsilon}).}
$$
On the other hand,  by the Young inequality and \dref{ddczhjjjj2.5ghxxssddccju48cfg9ssdd24},
\begin{equation}
\begin{array}{rl}
\disp-\int_{\Omega} m_{\varepsilon}\Delta c_{\varepsilon}\leq&\disp{\int_{\Omega} m^2_{\varepsilon}+\frac{1}{4}\int_{\Omega}|\Delta c_{\varepsilon}|^2  }\\
\leq&\disp{|\Omega|\| m_0\|^2_{L^\infty(\Omega)}+\frac{1}{4}\int_{\Omega}|\Delta c_{\varepsilon}|^2  ~~\mbox{for all}~~ t\in(0, T_{max,\varepsilon}).}\\
\end{array}
\label{ssdddaassshhxxcdfvvjjcz2.5}
\end{equation}
In the last summand in \dref{hhxxcsssdfvvjjczddfdddfff2.5}, we use the Cauchy-Schwarz inequality to obtain
\begin{equation}
\begin{array}{rl}
\disp-\int_{\Omega}\nabla c_{\varepsilon} \nabla (\nabla u_{\varepsilon} \cdot\nabla c_{\varepsilon})
\leq&\disp{\|\nabla u_{\varepsilon} \|_{L^{2}(\Omega)}\|\nabla c_{\varepsilon} \|_{L^{4}(\Omega)}^2~~\mbox{for all}~~ t\in(0,T_{max,\varepsilon}).}
\end{array}
\label{hhxxcsssdfvvjjcddssdddffzddfdddfff2.5}
\end{equation}
Now thanks to \dref{ddfgczhhhh2.5ghju48cfg924ghyuji} and in view of the Gagliardo-Nirenberg inequality, we can find $C_1> 0$ and  $C_2> 0$ fulfilling
integrate by parts to find that
%
\begin{equation}
\begin{array}{rl}
\disp \|\nabla c_{\varepsilon} \|_{L^{4}(\Omega)}^2\leq&\disp{C_{1}\|\Delta c_{\varepsilon} \|_{L^{2}(\Omega)}\|c_{\varepsilon} \|_{L^{\infty}(\Omega)}+C_{1}\|c_{\varepsilon} \|_{L^{\infty}(\Omega)}^4}\\
\leq&\disp{C_{2}\|\Delta c_{\varepsilon} \|_{L^{2}(\Omega)}+C_{2}
~~\mbox{for all}~~ t\in(0,T_{max,\varepsilon}).}\\
\end{array}
\label{hhxxcsssdfvvjjcddfffddffzddfdddfff2.5}
\end{equation}
This, together with the  Young  inequality, yields
\begin{equation}
\begin{array}{rl}
&\disp-\int_{\Omega}\nabla c_{\varepsilon} \nabla (\nabla u_{\varepsilon} \cdot\nabla c_{\varepsilon})
\\
\leq&\disp{\|\nabla u_{\varepsilon} \|_{L^{2}(\Omega)}[C_{2}\|\Delta c_{\varepsilon} \|_{L^{2}(\Omega)}+C_{2}]}
\\
\leq&\disp{C_{2}^2\|\nabla u_{\varepsilon} \|_{L^{2}(\Omega)}^2
+\frac{1}{4}\|\Delta c_{\varepsilon} \|_{L^{2}(\Omega)}^2+C_3~~\mbox{for all}~~ t\in(0,T_{max,\varepsilon}).}
\end{array}
\label{hhxxcsssdfvvjjcddffzddfdddfff2.5}
\end{equation}
Inserting \dref{ssdddaassshhxxcdfvvjjcz2.5} and \dref{hhxxcsssdfvvjjcddffzddfdddfff2.5} into  \dref{hhxxcsssdfvvjjczddfdddfff2.5}, we have
\begin{equation}
\begin{array}{rl}
&\disp\frac{1}{{2}}\disp\frac{d}{dt}\|\nabla{c_{\varepsilon} }\|^{{{2}}}_{L^{{2}}(\Omega)}+\frac{1}{2}
\int_{\Omega} |\Delta c_{\varepsilon} |^2+ \int_{\Omega} | \nabla c_{\varepsilon} |^2
\\
\leq&\disp{C_{2}^2\|\nabla u_{\varepsilon} \|_{L^{2}(\Omega)}^2+C_4.}
\end{array}
\label{hhxxcsssdfvvjjczddfssdddddfff2.5}
\end{equation}
Now, multiplying the
third equation of \dref{1.1fghyuisda} by $u_\varepsilon$, integrating by parts and using $\nabla\cdot u_{\varepsilon}=0$
\begin{equation}
\begin{array}{rl}
\disp{\frac{1}{2}\frac{d}{dt}\int_{\Omega}{|u_{\varepsilon}|^2}+\int_{\Omega}{|\nabla u_{\varepsilon}|^2}}= &\disp{ \int_{\Omega}(n_{\varepsilon}+m_{\varepsilon})u_{\varepsilon}\cdot\nabla \phi~~\mbox{for all}~~ t\in(0, T_{max,\varepsilon}).}\\
\end{array}
\label{ddddfgcz2.5ghju48cfg924ghyuji}
\end{equation}
Here we use the H\"{o}lder inequality, the Young inequality and the continuity of the embedding $W^{1,2}(\Omega)\hookrightarrow L^6(\Omega)$ and  to
find $C_{5} $ and $C_{6}> 0$ such that
\begin{equation}
\begin{array}{rl}
\disp\int_{\Omega}(n_{\varepsilon}+m_{\varepsilon})u_{\varepsilon}\cdot\nabla \phi\leq&\disp{\|\nabla \phi\|_{L^\infty(\Omega)}\| n_{\varepsilon} \|_{L^{\frac{6}{5}}(\Omega)}\| u_{\varepsilon}\|_{L^{6}(\Omega)}+\|\nabla \phi\|_{L^\infty(\Omega)}\| m_{\varepsilon} \|_{L^{\frac{6}{5}}(\Omega)}\| u_{\varepsilon}\|_{L^{6}(\Omega)}}\\
\leq&\disp{C_{5}\|\nabla \phi\|_{L^\infty(\Omega)}(\| n_{\varepsilon} \|_{L^{\frac{6}{5}}(\Omega)}+\|m_{\varepsilon} \|_{L^{\frac{6}{5}}(\Omega)})\|\nabla u_{\varepsilon}\|_{L^{2}(\Omega)}}\\
\leq&\disp{\frac{1}{2}\|\nabla u_{\varepsilon}\|_{L^{2}(\Omega)}^2+C_{6}~~\mbox{for all}~~ t\in(0, T_{max,\varepsilon})}\\
\end{array}
\label{dddddddddfgcz2.5ghju48cfg924ghyuji}
\end{equation}
by using \dref{czfvgb2.5ghhjuyuccvviihjj} and \dref{ddfgczhhhh2.5ghju48cfg924ghyuji}.
Inserting \dref{dddddddddfgcz2.5ghju48cfg924ghyuji} into \dref{ddddfgcz2.5ghju48cfg924ghyuji} and integrating  in time to see that
\begin{equation}
\begin{array}{rl}
&\disp{\int_{t}^{t+\tau}\int_{\Omega}  |\nabla {u_{\varepsilon}}|^2\leq C_7~~\mbox{for all}~~ t\in(0,T_{max,\varepsilon}-\tau)}\\
\end{array}
\label{bnmbncz2.5ghhjuddfghssddhdddddffggyhjkklluivvbnnihjj}
\end{equation}
and
\begin{equation}
\begin{array}{rl}
&\disp{\int_{\Omega} u_{\varepsilon}^{2}\leq  C_7~~~\mbox{for all}~~ t\in (0, T_{max,\varepsilon}),}\\
\end{array}
\label{czfvgb2.5ghhjussddyuccvviihjj}
\end{equation}
where we use that once more employing the Gagliardo-Nirenberg inequality, the H\"{o}lder
inequality and the Young inequality we can find $C_8> 0$ and $C_9> 0$ satisfying
\begin{equation}
\begin{array}{rl}
\disp\int_{t}^{t+\tau}\disp\int_{\Omega} |u_{\varepsilon}|^{\frac{10}{3}} =&\disp{\int_{t}^{t+\tau}\| {u_{\varepsilon}}\|^{{\frac{10}{3}}}_{L^{\frac{10}{3}}(\Omega)}}\\
\leq&\disp{C_8\int_{t}^{t+\tau}\left(\| \nabla{u_{\varepsilon}}\|^{2}_{L^{2}(\Omega)}\|{u_{\varepsilon}}\|^{{\frac{4}{3}}}_{L^{2}(\Omega)}+
\|{u_{\varepsilon}}\|^{{\frac{10}{3}}}_{L^{2}(\Omega)}\right)}\\
\leq&\disp{C_9~~\mbox{for all}~~ t\in(0,T_{max,\varepsilon}-\tau).}\\
\end{array}
\label{5555bnmbncz2ddfssdddvgffghhbhh.htt678hyuiihjj}
\end{equation}
Next, combining \dref{hhxxcsssdfvvjjczddfssdddddfff2.5}, \dref{bnmbncz2.5ghhjuddfghssddhdddddffggyhjkklluivvbnnihjj} and rearranging shows that
\begin{equation}
\begin{array}{rl}
&\disp{\int_{t}^{t+\tau}\int_{\Omega}  \left(|\Delta {c_{\varepsilon}}|^2+|\nabla {c_{\varepsilon}}|^4\right)\leq C_{10}~~\mbox{for all}~~ t\in(0,T_{max,\varepsilon}-\tau)}\\
\end{array}
\label{bnmbncz2.5ghhjuddfghssddhddddddddffggyhjkklluivvbnnihjj}
\end{equation}
and
\begin{equation}
\begin{array}{rl}
&\disp{\int_{\Omega} |\nabla {c_{\varepsilon}}|^{2}\leq  C_{10}~~~\mbox{for all}~~ t\in (0, T_{max,\varepsilon})}\\
\end{array}
\label{czfvgb2.5ghhjussddyuccvvissddihjj}
\end{equation}
by using \dref{hhxxcsssdfvvjjcddfffddffzddfdddfff2.5}.

Testing  the third equation in \dref{1.1fghyuisda} by $-\Delta m_{\varepsilon} $ and integrating by parts and using \dref{czfvgb2.5ghhjuyuccvviihjj} and \dref{bnmbncz2.5ghhjuddfghssddhdddddffggyhjkklluivvbnnihjj}, one can finally derive
\begin{equation}
\begin{array}{rl}
&\disp{\int_{t}^{t+\tau}\int_{\Omega}  \left(|\Delta {m_{\varepsilon}}|^2+|\nabla {m_{\varepsilon}}|^4\right)\leq C_{11}~~\mbox{for all}~~ t\in(0,T_{max,\varepsilon}-\tau)}\\
\end{array}
\label{bnmbncz2.5ghhjuddfghsddfgggsddhddddddddffggyhjkklluivvbnnihjj}
\end{equation}
and
\begin{equation}
\begin{array}{rl}
&\disp{\int_{\Omega} |\nabla {m_{\varepsilon}}|^{2}\leq  C_{11}~~~\mbox{for all}~~ t\in (0, T_{max,\varepsilon}).}\\
\end{array}
\label{czfvgb2.5ghhjussddyuccvsddddvissddihjj}
\end{equation}
 \end{proof}
 With Lemmas \ref{fvfgsdfggfflemma45} and \ref{lemmaghjffggssddgghhmk4563025xxhjklojjkkk}--\ref{ssdddlemmaghjffggssddgghhmk4563025xxhjklojjkkk} at hand, we can proceed to show that our approximate solutions
are actually global in time.

\begin{lemma}\label{lemma45630hhuujj}
For any $\varepsilon > 0,$
then one can
find $C > 0$ 
 such that the solutions of \dref{1.1fghyuisda} fulfill

\end{lemma}
\begin{proof}
Firstly, under the assumption that $T_{max,\varepsilon}< \infty$, for any $\varepsilon > 0$, Lemmas \ref{fvfgsdfggfflemma45}--\ref{ssdddlemmaghjffggssddgghhmk4563025xxhjklojjkkk} would provide us with $C_1 > 0$
such that
\begin{equation}
\begin{array}{rl}
&\disp{\int_{\Omega} n_{\varepsilon}  ^p\leq C_1~~~\mbox{for all}~~ t\in (0, T_{max,\varepsilon})~~~\mbox{for all}~~p>4,}\\
\end{array}
\label{czfvgb2sss.5ghhjuyuccvviihjj}
\end{equation}
\begin{equation}
\begin{array}{rl}
&\disp{\|c_{\varepsilon} (\cdot,t)\|_{L^\infty(\Omega)}\leq C_1~~~\mbox{for all}~~ t\in (0, T_{max,\varepsilon})}\\
\end{array}
\label{ssssczfvgb2sss.aass5ghhjuyuccvviihjj}
\end{equation}
as well as
\begin{equation}
\begin{array}{rl}
&\disp{\|m_{\varepsilon} (\cdot,t)\|_{L^\infty(\Omega)}\leq C_1~~~\mbox{for all}~~ t\in (0, T_{max,\varepsilon})}\\
\end{array}
\label{czfvgb2sss.aass5ghhjuyuccvviihjj}
\end{equation}
and
\begin{equation}
\begin{array}{rl}
&\disp{\int_{t}^{t+\tau}\int_{\Omega} \left[   |\nabla {c_{\varepsilon} }|^4+ |\nabla {m}|^4\right]\leq C_1.}\\
\end{array}
\label{bnmbncz2.5ghssffhjuyuivvbnnihjj}
\end{equation}
Then, aided by the $L^2$-estimate for $\nabla u_{\varepsilon}$ (from a testing argument),
we can obtain that
\begin{equation}
\begin{array}{rl}
\|A^\gamma u_{\varepsilon} (\cdot, t)\|_{L^2(\Omega)}\leq&\disp{C_{2}~~ \mbox{for all}~~ t\in(0,T_{max,\varepsilon})}\\
\end{array}
\label{cz2.571hhhhh51csdddcvvhddfccvvhjjjkkhhggjjllll}
\end{equation}
with some $C_2= C_2(\varepsilon) > 0$.
Thus, the continuous embedding
$D(A^\gamma)\hookrightarrow L^\infty(\Omega)$ implies the $L^\infty$ -estimate for $u_{\varepsilon}.$
Therefore, employing the same arguments as in the proof of Lemma 3.2 in \cite{Zhenssssssdffssdddddddgssddsddfff00}
(see also \cite{Kegssddsddfff00,Zhengsddfffsdddssddddkkllssssssssdefr23,Winkler11215,Winkler51215}), and taking advantage of \dref{czfvgb2sss.5ghhjuyuccvviihjj}--\dref{cz2.571hhhhh51csdddcvvhddfccvvhjjjkkhhggjjllll}, we conclude the estimates
\begin{equation}
\|n_{\varepsilon} (\cdot,t)\|_{L^\infty(\Omega)}  \leq C_3 ~~\mbox{for all}~~ t\in(0,\infty)
\label{zjscz2.5297x9630111kk}
\end{equation}
as well as
\begin{equation}
\|c_{\varepsilon} (\cdot,t)\|_{W^{1,\infty}(\Omega)}  \leq C_3 ~~~~\mbox{for all}~~ t\in(0,\infty)
\label{zjscz2.5297x9630111kkhh}
\end{equation}
and
\begin{equation}
\|m_{\varepsilon} (\cdot,t)\|_{W^{1,\infty}(\Omega)}  \leq C_3~~~~\mbox{for all}~~ t\in(0,\infty)
\label{sssszjscz2.5297x9630111kkhh}
\end{equation}
and some positive constant $C_3.$
In view of  \dref{cz2.571hhhhh51csdddcvvhddfccvvhjjjkkhhggjjllll}--\dref{sssszjscz2.5297x9630111kkhh}, we apply Lemma \ref{lemma70} to reach a contradiction.

\end{proof}


\subsection{Further a-priori estimates}

With the help of Lemma \ref{lemmaghjffggssddgghhmk4563025xxhjklojjkkk} and  the Gagliardo--Nirenberg inequality,
one can derive the following Lemma:
\begin{lemma}\label{lemmddaghjsffggggsddgghhmk4563025xxhjklojjkkk}
Let $\alpha>0$. Then for each $T\in(0, T_{max,\varepsilon})$,
 there exists $C>0$ independent of $\varepsilon$ such that the solution of \dref{1.1fghyuisda} satisfies
\begin{equation}
\begin{array}{rl}
&\disp{\int_{0}^T\int_{\Omega}|\nabla n_{\varepsilon}|^{2}\leq C(T+1).}\\
\end{array}
\label{bnmbncz2.ffghh5ghhjuyuivvbnnihjj}
\end{equation}
\end{lemma}
\begin{proof}
Multiply the first equation in $\dref{1.1fghyuisda}$ by $ n_{\varepsilon}$
 and  using $\nabla\cdot u_\varepsilon=0$, we derive
\begin{equation}
\begin{array}{rl}
&\disp{\frac{1}{{2}}\frac{d}{dt}\|{ n_{\varepsilon} }\|^{{{{2}}}}_{L^{{{2}}}(\Omega)}+
\int_{\Omega}  |\nabla n_{\varepsilon}|^2}\\
=&\disp{-
\int_{\Omega}  n_{\varepsilon}\nabla\cdot(n_{\varepsilon}S_\varepsilon(x, n_{\varepsilon}, c_{\varepsilon})\cdot\nabla c_{\varepsilon})-\int_{\Omega}n_{\varepsilon}^{2}m_{\varepsilon}}\\
\leq&\disp{
\int_{\Omega}  n_{\varepsilon}|S_\varepsilon(x, n_{\varepsilon}, c_{\varepsilon})||\nabla n_{\varepsilon}||\nabla c_{\varepsilon}|~~\mbox{for all}~~ t\in(0, T_{max,\varepsilon}).}
\end{array}
\label{55hhjjcffgjjjkkkkhhhjkklddfgggffgglffghhhz2.5}
\end{equation}
Recalling \dref{x1.73142vghf48gg} and using $\alpha\geq0$, from Young inequality again, we derive from Lemma \ref{lemmaghjffggssddgghhmk4563025xxhjklojjkkk} that
\begin{equation}
\begin{array}{rl}
&\disp\int_{\Omega} n_{\varepsilon} |S_\varepsilon(x, n_{\varepsilon}, c_{\varepsilon})||\nabla n_{\varepsilon}||\nabla c_{\varepsilon}|\\
\leq&\disp{C_S\int_{\Omega}n_{\varepsilon} |\nabla n_{\varepsilon}||\nabla c_{\varepsilon}|}\\
\leq&\disp{\frac{1}{2}\int_{\Omega} |\nabla n_{\varepsilon}|^2+\frac{C_S^2}{2}\int_{\Omega}n_{\varepsilon}^{2}
|\nabla c_{\varepsilon}|^2}\\
\leq&\disp{\frac{1}{2}\int_{\Omega} |\nabla n_{\varepsilon}|^2+\frac{1}{4}
\int_{\Omega}n_{\varepsilon}^{4}+\frac{C_S^4}{4} \int_{\Omega}|\nabla c_{\varepsilon}|^4}\\
\leq&\disp{\frac{1}{2}\int_{\Omega} |\nabla n_{\varepsilon}|^2+\frac{C_S^4}{4} \int_{\Omega}|\nabla c_{\varepsilon}|^4+C_1~~\mbox{for all}~~ t\in(0, T_{max,\varepsilon}),}
\end{array}
\label{55hhjjcffghhhjkkllfffghggggghgghjjjhfghhhz2.5}
\end{equation}
which combined with \dref{55hhjjcffgjjjkkkkhhhjkklddfgggffgglffghhhz2.5}  implies that
\begin{equation}
\begin{array}{rl}
&\disp{\frac{1}{{2}}\frac{d}{dt}\|{ n_{\varepsilon} }\|^{{{{2}}}}_{L^{{{2}}}(\Omega)}+
\frac{1}{2}\int_{\Omega}  |\nabla n_{\varepsilon}|^2\leq\frac{C_S^4}{4} \int_{\Omega}|\nabla c_{\varepsilon}|^4+C_1~~\mbox{for all}~~ t\in(0, T_{max,\varepsilon}),}
\end{array}
\label{55hhjjcffghhhjkklddfgggffgglffghhhz2.5}
\end{equation}
so that, gather \dref{bnmbncz2.5ghhjuyuivvbnnihjj} and \dref{55hhjjcffghhhjkklddfgggffgglffghhhz2.5}, one can get \dref{bnmbncz2.ffghh5ghhjuyuivvbnnihjj}.
\end{proof}

\section{Passing to the limit: The proof of Theorem  \ref{theorem3}}

With the help of a priori estimates, in
this subsection,  by means of a standard extraction procedure we can
now derive the following lemma which actually contains our main existence result (Theorem  \ref{theorem3}) already.



\begin{lemma}\label{lemma45hyuuuj630223}
Assume that   $\alpha>0$.
 Then
 for any $\kappa\in\mathbb{R}$,
 there exists $(\varepsilon_j)_{j\in \mathbb{N}}\subset (0, 1)$ such that $\varepsilon_j\rightarrow 0$ as $j\rightarrow\infty$ and that
\begin{equation}
n_\varepsilon\rightarrow n ~~\mbox{a.e.}~~\mbox{in}~~\Omega\times(0,\infty)~~\mbox{and in}~~ L_{loc}^{2}(\bar{\Omega}\times[0,\infty))\label{zjscz2.5297x963ddfgh0ddfggg6662222tt3}
\end{equation}
\begin{equation}
n_\varepsilon\rightharpoonup n ~~~\mbox{weak star in}~~ L^{\infty}_{loc}([0,\infty),L^p(\Omega))~~~\mbox{for any}~~p>1,\label{zjsczssdd2.5297x963ddfgh0ddfggg6662222tt3}
\end{equation}
\begin{equation}
\nabla n_\varepsilon\rightharpoonup \nabla n ~~\mbox{in}~~ L_{loc}^{2}(\bar{\Omega}\times[0,\infty)),\label{zjscz2.5297x963ddfgh0ddgghjjfggg6662222tt3}
\end{equation}
\begin{equation}
c_\varepsilon\rightarrow c ~~\mbox{in}~~ L^{2}_{loc}(\bar{\Omega}\times[0,\infty))~~\mbox{and}~~\mbox{a.e.}~~\mbox{in}~~\Omega\times(0,\infty),
 \label{zjscz2.fgghh5297x963ddfgh0ddfggg6662222tt3}
\end{equation}
\begin{equation}
m_\varepsilon\rightarrow m ~~\mbox{in}~~ L^{2}_{loc}(\bar{\Omega}\times[0,\infty))~~\mbox{and}~~\mbox{a.e.}~~\mbox{in}~~\Omega\times(0,\infty),
 \label{zjscz2.fgghh5297x96ddddd3ddfgh0ddfggg6662222tt3}
\end{equation}
\begin{equation}
\nabla c_\varepsilon\rightharpoonup \nabla c ~~\mbox{in}~~ L^{4}_{loc}(\bar{\Omega}\times[0,\infty)),
 \label{zjscz2.fgghh5297x963ddfgh0dddddfggg6662222tt3}
\end{equation}
\begin{equation}
\nabla m_\varepsilon\rightharpoonup \nabla m ~~\mbox{in}~~ L^{4}_{loc}(\bar{\Omega}\times[0,\infty)),
 \label{zjscz2.fgghh5297x963ddfgh0dddddfggg6662222tt3}
\end{equation}
\begin{equation}
u_\varepsilon\rightarrow u~~\mbox{in}~~ L_{loc}^2(\bar{\Omega}\times[0,\infty))~~\mbox{and}~~\mbox{a.e.}~~\mbox{in}~~\Omega\times(0,\infty),
 \label{zjscz2.5297x96302222t666t4}
\end{equation}
\begin{equation}
\nabla c_\varepsilon\rightharpoonup \nabla c~~\begin{array}{ll}
 \mbox{in}~~ L_{loc}^{2}(\bar{\Omega}\times[0,\infty))
   \end{array}\label{1.1ddfgghhhge666ccdf2345ddvbnmklllhyuisda}
\end{equation}
as well as
\begin{equation}
\nabla m_\varepsilon\rightharpoonup \nabla m~~\begin{array}{ll}
 \mbox{in}~~ L_{loc}^{2}(\bar{\Omega}\times[0,\infty))
   \end{array}\label{1.1dddddfgghhhge666ccdf2345ddvbnmklllhyuisda}
\end{equation}
and
\begin{equation}
 \nabla u_\varepsilon\rightharpoonup \nabla u ~~\mbox{ in}~~L^{2}_{loc}(\bar{\Omega}\times[0,\infty);\mathbb{R}^{3})
 \label{zjscz2.5297x96366602222tt4455}
\end{equation}
 some quadruple  $(n, c,m, u)$ which is a global weak solution of \dref{334451.1fghyuisda} in the natural sense as specified
in \cite{Zhenssssssdffssdddddddgssddsddfff00}.
\end{lemma}
\begin{proof}
Firstly, applying   the
discussion in Section 3, under the assumptions of Theorem \ref{theorem3}, for each $T > 0$,  we can find $\varepsilon$-independent constant $C(T)$ such that
\begin{equation}
\|n_\varepsilon(\cdot,t)\|_{L^p(\Omega)}+ \|c_\varepsilon(\cdot,t)\|_{W^{1,2}(\Omega)}+ \|m_\varepsilon(\cdot,t)\|_{W^{1,2}(\Omega)}+\|u_\varepsilon(\cdot,t)\|_{L^2(\Omega)} \leq C(T) ~~\mbox{for all}~~ t\in(0,T)~~\mbox{and}~~p>1
\label{zjscz2.ssddd5297x9630111kk}
\end{equation}
as well as
\begin{equation}
\int_{0}^T\int_{\Omega}\left(|\nabla n_\varepsilon|^2+|\nabla c_\varepsilon|^4+|\nabla m_\varepsilon|^4+ |\Delta c_\varepsilon|^2+ |\Delta m_\varepsilon|^2\right)\leq C(T) ~~\mbox{for all}~~ t\in(0, T)
\label{fvgbhzjscz2.5sss297x96302222tt4455hyuhii}
 \end{equation}
 and
 \begin{equation}
\int_{0}^T\int_{\Omega}|\nabla u_\varepsilon|^2\leq C(T) ~~\mbox{for all}~~ t\in(0, T) .
\label{fvgbhzjscz2.5297x96302222tt4455hyuhii}
 \end{equation}
  Now,   choosing $\varphi\in W^{1,2} (\Omega)$ as a test function in the first equation in \dref{1.1fghyuisda} and
using \dref{zjscz2.ssddd5297x9630111kk},
 we have
 \begin{equation}
\begin{array}{rl}
&\disp\left|\int_{\Omega}(n_{\varepsilon,t})\varphi\right|\\
 =&\disp{\left|\int_{\Omega}\left[\Delta n_{\varepsilon}-\nabla\cdot(n_{\varepsilon} S_\varepsilon(x, n_{\varepsilon}, c_{\varepsilon})\nabla c_{\varepsilon})-u_{\varepsilon}\cdot\nabla n_{\varepsilon}-n_{\varepsilon}m_{\varepsilon}\right]\varphi\right|}
\\
=&\disp{\left|\int_\Omega \left[-\nabla n_{\varepsilon}\cdot\nabla\varphi+n_{\varepsilon} S_\varepsilon(x, n_{\varepsilon}, c_{\varepsilon})\nabla c_{\varepsilon}\cdot\nabla\varphi+ n_{\varepsilon}u_{\varepsilon}\cdot\nabla  \varphi- n_{\varepsilon}m_{\varepsilon}  \varphi\right]\right|}\\
\leq&\disp{\left\{\|\nabla n_{\varepsilon}\|_{L^{2}(\Omega)}+ \|n_{\varepsilon}S_\varepsilon(x, n_{\varepsilon}, c_{\varepsilon})\nabla c_{\varepsilon}\|_{L^{2}(\Omega)}+ \|n_{\varepsilon}u_{\varepsilon}\|_{L^{2}(\Omega)}+\|n_{\varepsilon}m_{\varepsilon}\|_{L^{2}(\Omega)}
\right\}}\\
&\times\disp{\|\varphi\|_{W^{1,2}(\Omega)}}\\
\end{array}
\label{gbhncvbmdcfvgcz2.5ghju48}
\end{equation}
for all $t>0$.
Along with \dref{zjscz2.ssddd5297x9630111kk} and \dref{fvgbhzjscz2.5sss297x96302222tt4455hyuhii}, further implies that
\begin{equation}
\begin{array}{rl}
&\disp\int_0^T\|\partial_tn_\varepsilon(\cdot,t)\|_{({W^{1,2}(\Omega)})^*}^{2}dt \\
\leq&\disp{\int_0^T\left\{\|\nabla n_{\varepsilon}\|_{L^{2}(\Omega)}+ \|n_{\varepsilon} S_\varepsilon(x, n_{\varepsilon}, c_{\varepsilon})\nabla c_{\varepsilon}\|_{L^{2}(\Omega)}+ \|n_{\varepsilon}u_{\varepsilon}\|_{L^{2}(\Omega)}
\right\}}^{2}dt
\\
\leq&\disp{C_1\int_0^T\left\{\|\nabla n_{\varepsilon}\|_{L^{2}(\Omega)}^{2}+ \|n_{\varepsilon}\|_{L^4(\Omega)}^4 \|\nabla c_{\varepsilon}\|_{L^{4}(\Omega)}^{4}\right\}dt}\\
&\disp{+C_1\int_0^T\left\{\|n_{\varepsilon}u_{\varepsilon}\|_{L^{2}(\Omega)}^{2}+
\|m_{\varepsilon}\|^{2}_{L^{\infty}(\Omega)}\|n_{\varepsilon}\|_{L^{2}(\Omega)}^{2}
\right\}}dt\\
\leq&\disp{C_2\int_0^T\left\{\|\nabla n_{\varepsilon}\|_{L^{2}(\Omega)}^{2}+  \|\nabla c_{\varepsilon}\|_{L^{4}(\Omega)}^{4}\right\}dt}\\
&\disp{+C_2\int_0^T\left\{\|n_{\varepsilon}u_{\varepsilon}\|_{L^{2}(\Omega)}^{2}+
1
\right\}}dt,\\
\end{array}
\label{gbhncvbmdcfvgczffghhh2.5ghju48}
\end{equation}
where $C_1$ and $C_2$ are  positive constants independent of $\varepsilon$.
Now,
due to  the H\"{o}lder inequality, we have
\begin{equation}
\begin{array}{rl}
&\disp\int_{0}^T\int_{\Omega}|n_{\varepsilon}u_\varepsilon|^{2} \\
\leq&\disp{\int_{0}^T\|n_{\varepsilon}\|^{2}_{L^3(\Omega)}
\|u_\varepsilon\|^{2}_{L^6(\Omega)}}\\
\leq&\disp{C_3\int_{0}^T
\|\nabla u_\varepsilon\|^{2}_{L^2(\Omega)}}\\
\leq&\disp{C_4(T+1)~~\mbox{for all}~ T > 0}\\
\end{array}
\label{gbhncvbmdcfvgsddddczffghhh2.5ghju48}
\end{equation}
by using \dref{zjscz2.ssddd5297x9630111kk} and \dref{fvgbhzjscz2.5sss297x96302222tt4455hyuhii}.

Inserting \dref{gbhncvbmdcfvgsddddczffghhh2.5ghju48} into \dref{gbhncvbmdcfvgczffghhh2.5ghju48} and applying \dref{zjscz2.ssddd5297x9630111kk} and \dref{fvgbhzjscz2.5sss297x96302222tt4455hyuhii}, we can obtain for some positive constant $C_1(T)$ such that
 \begin{equation}
\|n_{\varepsilon t}\|_{L^2(0,T;(W^{1,2}(\Omega))^*)}\leq C_1(T)
\label{fvgbhzjsczsssd2.5297x9630222ssdd2tt4455hyuhii}
 \end{equation}
by using \dref{fvgbhzjscz2.5sss297x96302222tt4455hyuhii} and Lemma \ref{lemma45630hhuujj}.

In a similar  way, one can derive
 \begin{equation}
\|c_{\varepsilon t}\|_{L^2(0,T;(W^{1,2}(\Omega))^*)}\leq C_2(T)
\label{111fvgbhzjsczsssd2.5297x9630ssddd2222tt4455hyuhii}
 \end{equation}
 and
 \begin{equation}
\|m_{\varepsilon t}\|_{L^2(0,T;(W^{1,2}(\Omega))^*)}\leq C_2(T)
\label{fvgbhzjsczsssd2.5297x9630ssddd2222tt4455hyuhii}
 \end{equation}
with some $C_2(T) > 0.$

Next, for any given  $\varphi\in C^{\infty}_{0,\sigma} (\Omega;\mathbb{R}^3)$, we infer from the fourth  equation in \dref{1.1fghyuisda} that
$$
\begin{array}{rl}
\disp\left|\int_{\Omega}\partial_{t}u_{\varepsilon}(\cdot,t)\varphi\right|=&\disp{\left|-\int_\Omega \nabla u_{\varepsilon}\cdot\nabla\varphi-\kappa\int_\Omega (Y_{\varepsilon}u_{\varepsilon}\otimes u_{\varepsilon})\cdot\nabla\varphi+\int_\Omega (n_{\varepsilon}+m_{\varepsilon})\nabla \phi\cdot\varphi\right|
~~\mbox{for all}~~ t>0.}
\end{array}
$$
Now,   by virtue  of \dref{fvgbhzjscz2.5297x96302222tt4455hyuhii}, Lemma \ref{lemmaghjffggssddgghhmk4563025xxhjklojjkkk}  and Lemma \ref{fvfgsdfggfflemma45},
we thus infer that there exist  positive constants
 $C_{3},C_{4}$ and $C_{5}$ such that
 \begin{equation}
\begin{array}{rl}
&\disp\int_0^T\|\partial_{t}u_{\varepsilon}(\cdot,t)\|_{(W^{1,2}_{0,\sigma}(\Omega))^*}^{2}dt\\
\leq&\disp{C_{3}\left(\int_0^T\int_\Omega|\nabla u_{\varepsilon}|^{2}+\int_0^T\int_\Omega |Y_{\varepsilon}u_{\varepsilon}\otimes u_{\varepsilon}|^{2}+\int_0^T\int_\Omega n_\varepsilon^{2}+\int_0^T\int_\Omega m_\varepsilon^{2}\right)}\\
\leq&\disp{C_{4}\left(\int_0^T\int_\Omega|\nabla u_{\varepsilon}|^{2}+\int_0^T\int_\Omega |u_\varepsilon|^{2}+\int_0^T\int_\Omega n_{\varepsilon}^{2}+T\right)}\\
\leq&\disp{C_{5}(T+1)
~~\mbox{for all}~~ T>0.}
\end{array}
\label{fvgbhzjsczsssssddd2.5297x9630sssssdddssss2222tt4455hyuhii}
 \end{equation}
 Here we have used the fact that
 $$\|Y_{\varepsilon}v\|_{L^2(\Omega)}\leq \|v\|_{L^2(\Omega)}~~~\mbox{for all}~~ v\in L^2_{\sigma}(\Omega).$$
%
 Combining estimates \dref{zjscz2.ssddd5297x9630111kk}--\dref{fvgbhzjsczsssssddd2.5297x9630sssssdddssss2222tt4455hyuhii}, we conclude from
Aubin-Lions lemma (see e.g. \cite{Simon}) that $(u_\varepsilon)_{\varepsilon\in (0,1)}$ is relatively compact in $L^2_{loc} (\Omega\times[0,\infty);\mathbb{R}^3)$
and
$(n_\varepsilon)_{\varepsilon\in (0,1)}$,  $(c_\varepsilon)_{\varepsilon\in (0,1)}$, $(m_\varepsilon)_{\varepsilon\in (0,1)}$ are relatively compact in $L^2_{loc} (\Omega\times[0,\infty)).$
Therefore,
in conjunction with \dref{zjscz2.ssddd5297x9630111kk}--\dref{fvgbhzjscz2.5297x96302222tt4455hyuhii} and standard compactness arguments, we can thus
find a sequence $(\varepsilon_j)_{j\in N} \subset (0,1)$ such that $\varepsilon_j\searrow0$ as $j \rightarrow\infty$, and such that
\begin{equation}
 n_\varepsilon\rightarrow n ~~\mbox{in}~~ L^2_{loc}(\bar{\Omega}\times[0,\infty))~~~~\mbox{and}~~
 n_\varepsilon\rightarrow n~~~a.e. ~~~\mbox{in}~~~\Omega\times(0,\infty),
 \label{zjscz2.5297sssssssx963sss0222222ee}
\end{equation}
\begin{equation}
 n_\varepsilon \rightharpoonup n ~~\mbox{weak star in}~~ L^{\infty}_{loc}([0,\infty),L^p(\Omega))~~~~\mbox{for any}~~
 p>1,
 \label{zjscz2.5297ssdsssssssx963sss0222222ee}
\end{equation}
\begin{equation}
 \nabla n_\varepsilon\rightharpoonup \nabla n ~~\mbox{in}~~ L^2_{loc}(\bar{\Omega}\times[0,\infty)),
 \label{zjscz2.5297ssssssssssx963sss0222222ee}
\end{equation}
\begin{equation}
 c_\varepsilon\rightarrow c ~~\mbox{in}~~ L^2_{loc}(\bar{\Omega}\times[0,\infty))~~~~\mbox{and}~~
 c_\varepsilon\rightarrow c~~~a.e. ~~~\mbox{in}~~~\Omega\times(0,\infty),
 \label{zjscz2.5297sssssssssx9630222222ee}
\end{equation}
\begin{equation}
 \nabla c_\varepsilon\rightharpoonup \nabla c ~~\mbox{in}~~ L^4_{loc}(\bar{\Omega}\times[0,\infty))
 \label{zjscz2.5297ssssssdddsssssssx963sss0222222ee}
\end{equation}
\begin{equation}
\Delta c_\varepsilon\rightharpoonup \Delta c ~~\mbox{in}~~ L^2_{loc}(\bar{\Omega}\times[0,\infty))
 \label{zjscz2.5297sddffssddsssssdddsssssssx963sss0222222ee}
\end{equation}
\begin{equation}
 m_\varepsilon\rightarrow m ~~\mbox{in}~~ L^2_{loc}(\bar{\Omega}\times[0,\infty))~~~~\mbox{and}~~
 m_\varepsilon\rightarrow m~~~a.e. ~~~\mbox{in}~~~\Omega\times(0,\infty),
 \label{zjscz2.5297sssssssssx9630222222ee}
\end{equation}
\begin{equation}
 \nabla m_\varepsilon\rightharpoonup \nabla m ~~\mbox{in}~~ L^4_{loc}(\bar{\Omega}\times[0,\infty))
 \label{zjscz2.5297sssssssddssssssssx963sss0222222ee}
\end{equation}
\begin{equation}
\Delta m_\varepsilon\rightharpoonup \Delta m ~~\mbox{in}~~ L^2_{loc}(\bar{\Omega}\times[0,\infty))
 \label{zjscz2.5297sssssdddsssddssssssssx963sss0222222ee}
\end{equation}
as well as
\begin{equation}
 u_\varepsilon\rightarrow u ~~\mbox{in}~~ L^2_{loc}(\bar{\Omega}\times[0,\infty))~~~~\mbox{and}~~
u_\varepsilon\rightarrow u~~~a.e. ~~~\mbox{in}~~~\Omega\times(0,\infty)
 \label{zjscz2.5297sssfffssssssx9630222222ee}
\end{equation}
and
\begin{equation}
 \nabla u_\varepsilon\rightharpoonup \nabla u ~~\mbox{in}~~ L^2_{loc}(\bar{\Omega}\times[0,\infty)).
 \label{zjscz2.5297ssssssssssssssssx963sss0222222ee}
\end{equation}
for some limit function $(n,c,m,u).$
On the other hand, according to the bounds provided by Lemma  \ref{fvfgsdfggfflemma45} and Lemmas \ref{lemmaghjffggssddgghhmk4563025xxhjklojjkkk}--\ref{ssdddlemmaghjffggssddgghhmk4563025xxhjklojjkkk}, this readily yields  that, for any  $\varepsilon\in(0,1)$,
\begin{equation}m_{\varepsilon}-c_{\varepsilon}-u_{\varepsilon}\cdot\nabla c_{\varepsilon}~~~\mbox{is bounded in}~~~L^{\frac{5}{3}} (\Omega\times(0, T)), \label{zjsdffggghhcz2.5297sssssdddsssddssssssssx963sss0222222ee}
\end{equation}
where we have used the fact that
$$
 \begin{array}{ll}
  &\disp\int_0^T\int_{\Omega}\left[|u_{\varepsilon}\cdot\nabla c_{\varepsilon}|^{\frac{5}{3}} \right]\\
  \leq&\disp C_6\left[\int_{0}^T\disp\int_{\Omega} |u_{\varepsilon}|^{\frac{10}{3}}\right]^{\frac{1}{2}}\left[\int_{0}^T\disp\int_{\Omega} |\nabla c_\varepsilon|^{\frac{10}{3}}\right]^{\frac{1}{2}}\\
\leq&\disp C_7\left[\int_{0}^T\disp\int_{\Omega} |u_{\varepsilon}|^{\frac{10}{3}}\right]^{\frac{1}{2}}
\left[\int_{0}^T\disp\int_{\Omega} |\nabla c_\varepsilon|^{4}\right]^{\frac{5}{12}}\\
  \leq &C_8(T+1)
   \end{array}
$$
by using Lemma \ref{ssdddlemmaghjffggssddgghhmk4563025xxhjklojjkkk}.
Therefore, in light of \dref{zjsdffggghhcz2.5297sssssdddsssddssssssssx963sss0222222ee}, regularity estimates 
 for the second equation of \dref{1.1fghyuisda}
(see e.g. \cite{Ladyzenskajaggk7101}) ensure that 
$(c_{\varepsilon})_{\varepsilon\in(0,1)}$ is bounded in
$L^{\frac{5}{3}} ((0, T); W^{2,\frac{5}{3}}(\Omega))$.
Hence,  by virtue of \dref{fvgbhzjsczsssd2.5297x9630ssddd2222tt4455hyuhii}, we derive  form the Aubin--Lions lemma that $(c_{\varepsilon})_{\varepsilon\in(0,1)}$
 relatively  compacts  in
$L^{\frac{5}{3}} ((0, T); W^{1,\frac{5}{3}}(\Omega))$. Thus, we  can choose  an appropriate subsequence that is
still written as $(\varepsilon_j )_{j\in \mathbb{N}}$ such that $\nabla c_{\varepsilon_j} \rightarrow z_1$
     in $L^{\frac{5}{3}} (\Omega\times(0, T))$ for all $T\in(0, \infty)$ and some
$z_1\in L^{\frac{5}{3}} (\Omega\times(0, T))$ as $j\rightarrow\infty$. Therefore,  by \dref{zjscz2.5297ssssssdddsssssssx963sss0222222ee}, we can also derive that $\nabla c_{\varepsilon_j} \rightarrow z_1$ a.e. in $\Omega\times(0, \infty)$
 as $j \rightarrow\infty$.
In view  of \dref{zjscz2.5297ssssssdddsssssssx963sss0222222ee} and   the Egorov theorem, we conclude  that
$z_1=\nabla c$ and hence
\begin{equation}
 \nabla c_\varepsilon\rightarrow \nabla c ~~~a.e. ~~~\mbox{in}~~~\Omega\times(0,\infty).
 \label{zjscz2.5297sssssssssssssx9630222222ee}
\end{equation}
This combined  with \dref{zjscz2.5297sssssssx963sss0222222ee}, \dref{zjscz2.5297ssdsssssssx963sss0222222ee} as well as  \dref{zjscz2.5297ssssssdddsssssssx963sss0222222ee} and \dref{x1.73142vghf48rtgyhu} implies that
\begin{equation}n_\varepsilon S_\varepsilon(x, n_{\varepsilon}, c_{\varepsilon})\nabla c_\varepsilon\rightharpoonup nS(x, n, c)\nabla c
~\mbox{in}~ L^{2}(\Omega\times(0,T))~\mbox{as}~\varepsilon = \varepsilon_j\searrow 0~\mbox{for each}~ T\in(0,\infty)
\label{1.1ddddfddffttyygghhyujiikkkffghhhifgghhhgffddgge6bhhjh66ccdf2345ddvbnmklddfggllhyuisda}
\end{equation}
by using the Egorov theorem.
 Next we shall prove that $(n,c,m,u)$ is a weak solution of problem \dref{1.1fghyuisda}.  To this end, testing the
first equation in \dref{1.1fghyuisda} by $\varphi\in C^\infty_0(\Omega\times [0,\infty))$, we obtain
\begin{equation}
\begin{array}{rl}\label{eqx4ss5xx12112ccgghh}
\disp{-\int_0^{\infty}\int_{\Omega}n_\varepsilon\varphi_t-\int_{\Omega}n_0\varphi(\cdot,0)  }=&\disp{-
\int_0^\infty\int_{\Omega}\nabla n_\varepsilon\cdot\nabla\varphi+\int_0^\infty\int_{\Omega}n_\varepsilon
S_\varepsilon(x,n_\varepsilon,c_\varepsilon)\nabla c_\varepsilon\cdot\nabla\varphi}\\
&+\disp{\int_0^\infty\int_{\Omega}n_\varepsilon u_\varepsilon\cdot\nabla\varphi-\int_0^{\infty}\int_{\Omega}n_\varepsilon m_\varepsilon\varphi}\\
\end{array}
\end{equation}
 for all $\varepsilon\in (0,1)$.
 Then \dref{zjscz2.5297ssssssssssx963sss0222222ee}--\dref{1.1ddddfddffttyygghhyujiikkkffghhhifgghhhgffddgge6bhhjh66ccdf2345ddvbnmklddfggllhyuisda} and the dominated convergence theorem enables
us to conclude
\begin{equation}
\begin{array}{rl}\label{eqx45xx12112ccgghh}
\disp{-\int_0^{\infty}\int_{\Omega}n\varphi_t-\int_{\Omega}n_0\varphi(\cdot,0)  }=&\disp{-
\int_0^\infty\int_{\Omega}\nabla n\cdot\nabla\varphi+\int_0^\infty\int_{\Omega}n
S(x,n,c)\nabla c\cdot\nabla\varphi}\\
&+\disp{\int_0^\infty\int_{\Omega}nu\cdot\nabla\varphi-\int_0^\infty\int_{\Omega}nm\varphi}\\
\end{array}
\end{equation}
by a limit procedure.
Next,
multiplying the second equation and the third equation  in \dref{1.1fghyuisda} by
$\varphi\in  C^\infty_0(\Omega\times [0,\infty))$, we derive from a limit procedure that
 \begin{equation}
\begin{array}{rl}\label{222eqx45xx12112ccgghhjj}
\disp{-\int_0^{\infty}\int_{\Omega}c\varphi_t-\int_{\Omega}c_0\varphi(\cdot,0)  }=&\disp{-
\int_0^\infty\int_{\Omega}\nabla c\cdot\nabla\varphi-\int_0^\infty\int_{\Omega}c\varphi+\int_0^\infty\int_{\Omega}m\varphi+
\int_0^\infty\int_{\Omega}cu\cdot\nabla\varphi}\\
\end{array}
\end{equation}
and
\begin{equation}
\begin{array}{rl}\label{111eqx45fffffxx12112ccgghhjj}
\disp{-\int_0^{\infty}\int_{\Omega}m\varphi_t-\int_{\Omega}m_0\varphi(\cdot,0)  }=&\disp{-
\int_0^\infty\int_{\Omega}\nabla m\cdot\nabla\varphi-\int_0^\infty\int_{\Omega}nm\varphi+
\int_0^\infty\int_{\Omega}mu\cdot\nabla\varphi}\\
\end{array}
\end{equation}
in a completed similar manner (see \cite{Zhenssssssdffssdddddddgssddsddfff00} for details).
Then testing the fourth equation of \dref{1.1fghyuisda} by $\varphi\in C_0^{\infty} (\bar{\Omega}\times[0, T);\mathbb{R}^3)$, we obtain
\begin{equation}
\begin{array}{rl}\label{eqx45xx12ddd112ccgghhjjgghh}
\disp{-\int_0^{\infty}\int_{\Omega}u\varphi_t-\int_{\Omega}u_0\varphi(\cdot,0)+ \kappa
\int_0^T\int_{\Omega} u\otimes u\cdot\nabla\varphi}=&\disp{-
\int_0^\infty\int_{\Omega}\nabla u\cdot\nabla\varphi-
\int_0^\infty\int_{\Omega}(n+m)\nabla\phi\cdot\varphi}\\
\end{array}
\end{equation}
by using Lemma \ref{lemma45630hhuujj} and a limit procedure (see \cite{Zhenssssssdffssdddddddgssddsddfff00} for details).
This means that $(n,c,m,u)$ is a weak solution of \dref{1.1fghyuisda}, in the natural sense as specified
in \cite{Zhenssssssdffssdddddddgssddsddfff00}.

\end{proof}

Moreover, if in addition we assume that $\kappa\in\mathbb{R}$, then our solutions will actually be bounded and smooth
and hence classical. In  fact, by applying the standard
parabolic regularity and the classical Schauder estimates for the Stokes evolution, we will
show that it is sufficiently regular so as to be a classical solution.
\begin{lemma}\label{lemmassddddff45630223}
Let  $(n,c,m,u)$ be a weak solution of \dref{334451.1fghyuisda}.
Assume that   $\alpha>0$ and $\kappa=0$. Then  $(n,c,m,u)$ solves \dref{334451.1fghyuisda} in the classical sense in $\Omega\times (0,\infty).$
 Moreover, this solution is bounded in
$\Omega\times(0,\infty)$ in the sense that
\begin{equation}
\|n(\cdot, t)\|_{L^\infty(\Omega)}+\|c(\cdot, t)\|_{W^{1,\infty}(\Omega)}+\|m(\cdot, t)\|_{W^{1,\infty}(\Omega)}+\|A^\gamma u(\cdot, t)\|_{L^2(\Omega)}\leq C~~ \mbox{for all}~~ t>0.
\label{1.163072xggttsdddyyu}
\end{equation}

\end{lemma}
\begin{proof}
In what follows, let $C, C_i$ denote some different constants, 
and if no special explanation, they depend at most on $\Omega, \phi, m, n_0, c_0$ and
$u_0$.

{\bf Step 1. The boundedness of $\|A^\gamma u (\cdot, t)\|_{L^2(\Omega)}$ and $\| u (\cdot, t)\|_{L^{\infty}(\Omega)}$ for all $t\in (0, T_{max,\varepsilon})$}

On the basis of the variation-of-constants formula for the projected version of the third
equation in \dref{334451.1fghyuisda}, we derive that
$$u (\cdot, t) = e^{-tA}u_0 +\int_0^te^{-(t-\tau)A}
\mathcal{P}((n (\cdot,\tau)+m (\cdot,\tau))\nabla\phi)d\tau~~ \mbox{for all}~~ t\in(0,T_{max,\varepsilon}).$$
On  the other hand, in view of Lemma \ref{lemmaghjffggssddgghhmk4563025xxhjklojjkkk} as well as  \dref{x1.73142vghf481} and \dref{ddfgczhhhh2.5ghju48cfg924ghyuji},
$$
\|h (\cdot,t)\|_{L^{{4+2\alpha}}(\Omega)}\leq C ~~~\mbox{for all}~~ t\in(0,T_{max,\varepsilon})
$$
with $h :=\mathcal{P}((n (\cdot,\tau)+m (\cdot,\tau)\nabla\phi)$.
Therefore, according to standard smoothing
properties of the Stokes semigroup we see that there exist $C_1 , C_2> 0$ and $\lambda_1 > 0$ such that
\begin{equation}
\begin{array}{rl}
\|A^\gamma u (\cdot, t)\|_{L^2(\Omega)}\leq&\disp{\|A^\gamma
e^{-tA}u_0\|_{L^2(\Omega)} +\int_0^t\|A^\gamma e^{-(t-\tau)A}h (\cdot,\tau)d\tau\|_{L^2(\Omega)}d\tau}\\
\leq&\disp{\|A^\gamma u_0\|_{L^2(\Omega)} +C_{1}\int_0^t(t-\tau)^{-\gamma-\frac{3}{2}(\frac{1}{{4+2\alpha}}-\frac{1}{2})}e^{-\lambda_1(t-\tau)}\|h (\cdot,\tau)\|_{L^{{4+2\alpha}}(\Omega)}d\tau}\\
\leq&\disp{C_{2}~~ \mbox{for all}~~ t\in(0,T_{max,\varepsilon})}\\
\end{array}
\label{cz2.571hhhhh51ccvvhddfccvvhjjjkkhhggjjllll}
\end{equation}
with $\gamma\in ( \frac{3}{4}, 1),$
where  in
the last inequality, 
 we have used the fact that
$$\begin{array}{rl}\disp\int_{0}^t(t-\tau)^{-\gamma-\frac{3}{2}(\frac{1}{{4+2\alpha}}-\frac{1}{2})}e^{-\lambda_1(t-\tau)}ds
\leq&\disp{\int_{0}^{\infty}\sigma^{-\gamma-\frac{3}{2}(\frac{1}{{4+2\alpha}}-\frac{1}{2})} e^{-\lambda_1\sigma}d\sigma<+\infty}\\
\end{array}
$$
by using  $-\gamma-\frac{3}{2}(\frac{1}{{4+2\alpha}}-\frac{1}{2})>-1.$
\dref{cz2.571hhhhh51ccvvhddfccvvhjjjkkhhggjjllll} implies to
 \begin{equation}
\begin{array}{rl}
\|u (\cdot, t)\|_{L^\infty(\Omega)}\leq  \sigma_{0}~~ \mbox{for all}~~ t\in(0,T_{max,\varepsilon})\\
\end{array}
\label{cz2ddfgjjj.5jkkcvvvhjkfffffkhhgll}
\end{equation}
by using the fact that $D(A^\gamma)$ is continuously embedded into $L^\infty(\Omega)$ (by $\gamma>\frac{3}{4}$).

{\bf Step 2. The boundedness of $\|\nabla c (\cdot, t)\|_{L^4(\Omega)}$ for all $t\in (0, T_{max,\varepsilon})$}

Now, multiply the second  equation in $\dref{334451.1fghyuisda}$ by $-\Delta c $, in view of  \dref{czfvgb2.5ghhjuyuccvviihjj}, \dref{ddfgczhhhh2.5ghju48cfg924ghyuji} and \dref{cz2ddfgjjj.5jkkcvvvhjkfffffkhhgll}, we derive from \dref{bnmbncz2.5ghhjuddfghhdddddffggyhjkklluivvbnnihjj} and the Young inequality that
 \begin{equation}
\begin{array}{rl}
\|\nabla c (\cdot, t)\|_{L^2(\Omega)}\leq  \sigma_{1}~~ \mbox{for all}~~ t\in(0,T_{max,\varepsilon}).\\
\end{array}
\label{cz2ddfgjjj.5jddfghkkcvvvhjkfffffkhhgll}
\end{equation}

Considering the fact that $\nabla c \cdot\nabla\Delta c   = \frac{1}{2}\Delta |\nabla c |^2-|D^2c |^2$,
by a straightforward computation using the second equation in \dref{334451.1fghyuisda} and several integrations by parts, we find that
\begin{equation}
\begin{array}{rl}
&\disp{\frac{1}{{4}}\frac{d}{dt} \|\nabla c \|^{{{4}}}_{L^{{4}}(\Omega)}}
\\
= &\disp{\int_{\Omega} |\nabla c |^{2}\nabla c \cdot\nabla(\Delta c
-c +m -u \cdot\nabla  c )}
\\
=&\disp{\frac{1}{{2}}\int_{\Omega} |\nabla c |^{2}\Delta |\nabla c |^2
-\int_{\Omega} |\nabla c |^{2}|D^2 c |^2-\int_{\Omega} |\nabla c |^{4}}
\\
&+\disp{\int_\Omega m \nabla\cdot( |\nabla c |^{2}\nabla c )
+\int_\Omega (u \cdot\nabla  c )\nabla\cdot( |\nabla c |^{2}\nabla c )}
\\
=&\disp{-\frac{1}{{2}}\int_{\Omega} \left|\nabla |\nabla c |^{2}\right|^2-\int_{\Omega} |\nabla c |^{4}
+\frac{1}{{2}}\int_{\partial\Omega} |\nabla c |^{2}\frac{\partial  |\nabla c |^{2}}{\partial\nu}}\\
&-\disp{\int_{\Omega} |\nabla c |^{2}|D^2 c |^2
+\int_\Omega m  |\nabla c |^{2}\Delta c +\int_\Omega m \nabla c \cdot\nabla( |\nabla c |^{2})}
\\
&+\disp{\int_\Omega (u \cdot\nabla  c ) |\nabla c |^{2}\Delta c
+\int_\Omega (u \cdot\nabla  c )\nabla c \cdot\nabla( |\nabla c |^{2})}
\\
\end{array}
\label{cz2.5gdfhjjkhju48156}
\end{equation}
for all $t\in(0,T_{max,\varepsilon})$.
On the other hand, since Lemma 2.2 of \cite{Zhengssssssdefr23}, we derive from \dref{ddfgczhhhh2.5ghju48cfg924ghyuji} and the Young inequality  that
\begin{equation}
\begin{array}{rl}
 \|\nabla {c }\|_{L^{4}(\Omega)}^4\leq&\displaystyle{\kappa_0\||\nabla{c }|D^2{c }
 \|_{L^2(\Omega)}^{\frac{2}{3}}
 \|{c }\|_{L^\infty(\Omega)}^{\frac{8}{3}}+\kappa_0\|{c }\|_{L^\infty(\Omega)}^4}
\\
\leq&\displaystyle{\frac{1}{4(1+16\sigma_0^2)}\int_{\Omega} |\nabla c |^{2}|D^2 c |^2+\kappa_1,}\\
\end{array}
\label{cz2.563022222ikossdddpl255}
\end{equation}
where $\sigma_0$ is the same as \dref{cz2ddfgjjj.5jkkcvvvhjkfffffkhhgll} and %
 $\kappa_0 $ and $\kappa_{1}$ are
some positive constants.
Thanks to the pointwise inequality $|\Delta c | \leq\sqrt{3}|D^2c |$, along with \dref{x1.73142vghf481}  as well as \dref{ddfgczhhhh2.5ghju48cfg924ghyuji} and \dref{ddfgczhhhh2.5ghju48cfg924ghyuji} this implies that
\begin{equation}
\begin{array}{rl}
&\disp\int_\Omega m  |\nabla c |^{2}\Delta c
\\
\leq&\disp{\sqrt{3} \int_\Omega m  |\nabla c |^{2}|D^2c |}
\\
\leq&\disp{\frac{1}{8}\int_\Omega  |\nabla c |^{2}|D^2c |^2+6\|m \|_{L^\infty(\Omega)}^2\int_\Omega |\nabla c |^{2}}
\\
\leq&\disp{\frac{1}{8}\int_\Omega  |\nabla c |^{2}|D^2c |^2+6\lambda^2\sigma_{1}^2}
\\
\leq&\disp{\frac{1}{8}\int_\Omega  |\nabla c |^{2}|D^2c |^2+C_3~~ \mbox{for all}~~ t\in(0,T_{max,\varepsilon})}
\\
\end{array}
\label{cz2.5ghju48hjuikl1}
\end{equation}
and
\begin{equation}
\begin{array}{rl}
&\disp\int_\Omega (u \cdot\nabla  c ) |\nabla c |^{2}\Delta c
\\
\leq&\disp{\sqrt{3}\int_\Omega |u \cdot\nabla  c | |\nabla c |^{2}|D^2c |}
\\
\leq&\disp{\frac{1}{16}\int_\Omega  |\nabla c |^{2}|D^2c |^2
+12\int_\Omega |u \cdot\nabla  c |^2 |\nabla c |^{2}}
\\
\leq&\disp{\frac{1}{16}\int_\Omega  |\nabla c |^{2}|D^2c |^2
+12\|u \|^2_{L^\infty(\Omega)}\int_\Omega  |\nabla c |^{4}~~ \mbox{for all}~~ t\in(0,T_{max,\varepsilon})}
\end{array}
\label{cz2.5ghju48hjuikl451}
\end{equation}
by using \dref{czfvgb2.5ghhjuyuccvviihjj} and the Young inequality. Now, inserting \dref{cz2.563022222ikossdddpl255} into  \dref{cz2.5ghju48hjuikl451},  this shows that
\begin{equation}
\begin{array}{rl}
&\disp\int_\Omega (u \cdot\nabla  c ) |\nabla c |^{2}\Delta c
\\
\leq&\disp{\frac{1}{16}\int_\Omega  |\nabla c |^{2}|D^2c |^2
+12\|u \|^2_{L^\infty(\Omega)}\times[\frac{1}{4(1+16\sigma_0^2)}\int_{\Omega} |\nabla c |^{2}|D^2 c |^2+\kappa_1]}\\
\leq&\disp{\frac{1}{4}\int_\Omega  |\nabla c |^{2}|D^2c |^2+C_4~~ \mbox{for all}~~ t\in(0,T_{max,\varepsilon})}
\end{array}
\label{3344cz2.5ghju48hssddjuikl451}
\end{equation}
by using \dref{cz2ddfgjjj.5jkkcvvvhjkfffffkhhgll}.
Again, from the Young inequality, \dref{x1.73142vghf481} as well as \dref{ddfgczhhhh2.5ghju48cfg924ghyuji}   and \dref{czfvgb2.5ghhjuyuccvviihjj}, we have
\begin{equation}
\begin{array}{rl}
&\disp\int_\Omega m \nabla c \cdot\nabla( |\nabla c |^{2})
\\
\leq &\disp{\frac{1}{8}\int_{\Omega} \left|\nabla |\nabla c |^{2}\right|^2+2
\lambda^2\sigma_{1}^2}
\\
\leq &\disp{\frac{1}{8}\int_{\Omega}\left|\nabla |\nabla c |^{2}\right|^2+C_5}
\end{array}
\label{cz2.5ghju4ghvvvbbbjuk81}
\end{equation}
and
\begin{equation}
\begin{array}{rl}
&\disp\int_\Omega (u \cdot\nabla  c )\nabla c \cdot\nabla( |\nabla c |^{2})
\\
\leq &\disp{\frac{1}{16}\int_{\Omega}\left|\nabla |\nabla c |^{2}\right|^2+4\int_\Omega |u \cdot\nabla  c |^2 |\nabla c |^{2}}
\\
\leq &\disp{\frac{1}{16}\int_{\Omega}\left|\nabla |\nabla c |^{2}\right|^2
+4 \|u \|^2_{L^\infty(\Omega)}\times[\frac{1}{4(1+16\sigma_0^2)}\int_{\Omega} |\nabla c |^{2}|D^2 c |^2+\kappa_1]}\\
\leq &\disp{\frac{1}{16}\int_{\Omega}\left|\nabla |\nabla c |^{2}\right|^2
+\frac{1}{16}\int_{\Omega} |\nabla c |^{2}|D^2 c |^2+C_6.}
\end{array}
\label{cz2.5ghju4ccvvvghjuk81}
\end{equation}
Given the the boundedness of $\|\nabla  c  \|_{L^2(\Omega)}^2$ (see \dref{cz2ddfgjjj.5jddfghkkcvvvhjkfffffkhhgll}), it is well-known that
(cf. \cite{Ishida, Tao41215,Zhengsdsd6}) the boundary trace embedding  implies that
\begin{equation} \label{com-est-4}
\begin{split}
\int_{\partial\Omega}  |\nabla c |^{2}\frac{\partial}{\partial \nu} |\nabla c |^2&\leq \frac{1}{16}\int_\Omega |\nabla |\nabla c |^2|^2+C_7 \Bigr(\int_{\Omega} |\nabla c |^2\Bigr)^2\\
&\leq \frac{1}{16}\int_\Omega |\nabla |\nabla c |^2|^2+C_{8}.
\end{split}
\end{equation}
Now, together with \dref{cz2.5gdfhjjkhju48156}, \dref{cz2.5ghju48hjuikl1}--\dref{com-est-4}, we can derive that, for some positive constant $C_9$,
\begin{equation}
\begin{array}{rl}
\disp{\frac{1}{{4}}\frac{d}{dt} \|\nabla c \|^{{{4}}}_{L^{{4}}(\Omega)}+\frac{3}{{4}}\int_{\Omega} \left|\nabla |\nabla c |^{2}\right|^2
+\frac{1}{{2}}\int_{\Omega} |\nabla c |^{2}|D^2 c |^2}\leq&\disp{C_{9}~~ \mbox{for all}~~ t\in(0,T_{max,\varepsilon}),}\\
\end{array}
\label{cz2.sss5gdfhjjkhjsdfggu48156}
\end{equation}
which combined with \dref{cz2.563022222ikossdddpl255}  yields to
\begin{equation}
\begin{array}{rl}
\disp{\frac{1}{{4}}\frac{d}{dt} \|\nabla c \|^{{{4}}}_{L^{{4}}(\Omega)}+ C_{10}\|\nabla c \|^{{{4}}}_{L^{{4}}(\Omega)}}\leq&\disp{C_{11}~~ \mbox{for all}~~ t\in(0,T_{max,\varepsilon}).}\\
\end{array}
\label{cz2.5gdfhjjkhjsdfggu48156}
\end{equation}
This implies
\begin{equation}
\begin{array}{rl}
\|\nabla c (\cdot, t)\|_{L^4(\Omega)}\leq  C_{12}~~ \mbox{for all}~~ t\in(0,T_{max,\varepsilon})\\
\end{array}
\label{cz2ddfgjjj.5jddfghkkcvkkkllvvhjkfffffkhhgll}
\end{equation}
by  integration.

{\bf Step 3. The boundedness of $\|\nabla m (\cdot, t)\|_{L^{4}(\Omega)}$  for all $t\in (0, T_{max,\varepsilon})$}

An application of the variation of
constants formula for $c $ leads to
\begin{equation}
\begin{array}{rl}
&\disp{\|\nabla m (\cdot, t)\|_{L^{4}(\Omega)}}\\
\leq&\disp{\|\nabla e^{t(\Delta-1)} m_0\|_{L^{4}(\Omega)}+
\int_{0}^t\|\nabla e^{(t-s)(\Delta-1)}(m (s)-n (s)m (s))\|_{L^{4}(\Omega)}ds}\\
&\disp{+\int_{0}^t\|\nabla e^{(t-s)(\Delta-1)}\nabla \cdot(u (s) m (s))\|_{L^{4}(\Omega)}ds.}\\
\end{array}
\label{44444zjccfgghhhfgbhjcvvvbscz2.5297x96301ku}
\end{equation}
To estimate the terms on the right of  \dref{44444zjccfgghhhfgbhjcvvvbscz2.5297x96301ku}, in light of \dref{ddfgczhhhh2.5ghju48cfg924ghyuji} and \dref{czfvgb2.5ghhjuyuccvviihjj},
applying  the $L^p$-$L^q$ estimates associated heat semigroup, for some positive constant $\lambda_1$ such that
\begin{equation}
\begin{array}{rl}
\|\nabla e^{t(\Delta-1)} m_0\|_{L^{4}(\Omega)}\leq &\disp{C_{13}~~ \mbox{for all}~~ t\in(0,T_{max,\varepsilon})}\\
\end{array}
\label{zjsssddddccffgbhjcghhhjjjvvvbscz2.5297x96301ku}
\end{equation}
as well as
\begin{equation}
\begin{array}{rl}
&\disp{\int_{0}^t\|\nabla e^{(t-s)(\Delta-1)}(m (s)-n (s)m (s))\|_{L^{4}(\Omega)}ds}\\
\leq&\disp{C_{14}\int_{0}^t[1+(t-s)^{-\frac{1}{2}-\frac{3}{2}(\frac{1}{4}-\frac{1}{4})}] e^{-\lambda_1(t-s)}(\|n (s)\|_{L^{4}(\Omega)}+\|m (s)\|_{L^{\infty}(\Omega)})ds}\\
\leq&\disp{C_{15}~~ \mbox{for all}~~ t\in(0,T_{max,\varepsilon})}\\
\end{array}
\label{zjccffgbhjcvvvbscz2.5297x96301ku}
\end{equation}
and
\begin{equation}
\begin{array}{rl}
&\disp{\int_{0}^t\|\nabla e^{(t-s)(\Delta-1)}\nabla \cdot(u (s) c (s))\|_{L^{4}(\Omega)}ds}\\
\leq&\disp{C_{16}\int_{0}^t\|(-\Delta+1)^\iota e^{(t-s)(\Delta-1)}\nabla \cdot(u (s) c (s))\|_{L^{4}(\Omega)}ds}\\
\leq&\disp{C_{17}\int_{0}^t(t-s)^{-\iota-\frac{1}{2}-\tilde{\kappa}} e^{-\lambda_1(t-s)}\|u (s) c (s)\|_{L^{4}(\Omega)}ds}\\
\leq&\disp{C_{18}\int_{0}^t(t-s)^{-\iota-\frac{1}{2}-\tilde{\kappa}} e^{-\lambda_1(t-s)}\|u (s)\|_{L^{\infty}(\Omega)}\| c (s)\|_{L^{\infty}(\Omega)}ds}\\
\leq&\disp{C_{19}~~ \mbox{for all}~~ t\in(0,T_{max,\varepsilon}),}\\
\end{array}
\label{zjccffgbhjcvdgghhhhdfgghhvvbscz2.5297x96301ku}
\end{equation}
where $\iota=\frac{13}{28},\tilde{\kappa}=\frac{1}{56}$.
Inserting \dref{zjsssddddccffgbhjcghhhjjjvvvbscz2.5297x96301ku}--\dref{zjccffgbhjcvdgghhhhdfgghhvvbscz2.5297x96301ku} into \dref{44444zjccfgghhhfgbhjcvvvbscz2.5297x96301ku}, one has
\begin{equation}
\begin{array}{rl}
\|\nabla m (\cdot, t)\|_{L^4(\Omega)}\leq  \sigma_{2}~ \mbox{for all}~~ t\in(0,T_{max,\varepsilon}).\\
\end{array}
\label{cz2ddfgjjj.5jssddddddfghkkcvkkkllvvhjkfffffkhhgll}
\end{equation}

{\bf Step 4. The boundedness of $\|c (\cdot, t)\|_{W^{1,\infty}(\Omega)}$ and $\|m (\cdot, t)\|_{W^{1,\infty}(\Omega)}$
 for all  $t\in (\tau, T_{max,\varepsilon})$ with $\tau\in(0,T_{max,\varepsilon})$}

Choosing $\theta\in(\frac{1}{2}+\frac{3}{8},1),$ 
 then the domain of the fractional power $D((-\Delta + 1)^\theta)\hookrightarrow W^{1,\infty}(\Omega)$ (see e.g.  \cite{Horstmann791,Winkler792}). Thus,  in light of $\alpha>0$, using the H\"{o}lder inequality and the $L^p$-$L^q$ estimates associated heat semigroup, 
\begin{equation}
\begin{array}{rl}
&\| c (\cdot, t)\|_{W^{1,\infty}(\Omega)}\\
\leq&\disp{C_{20}\|(-\Delta+1)^\theta c (\cdot, t)\|_{L^{4}(\Omega)}}\\
\leq&\disp{C_{21}t^{-\theta}e^{-\lambda_1 t}\|c_0\|_{L^{4}(\Omega)}+C_{21}\int_{0}^t(t-s)^{-\theta}e^{-\lambda_1(t-s)}
\|(m -u  \cdot \nabla c )(s)\|_{L^{4}(\Omega)}ds}\\
\leq&\disp{C_{22}+C_{22}\int_{0}^t(t-s)^{-\theta}e^{-\mu(t-s)}[\|m (s)\|_{L^{\infty}(\Omega)}+\|c (s)\|_{L^{\infty}(\Omega)}+\|u (s)\|_{L^\infty(\Omega)}
\|\nabla c (s)\|_{L^{4}(\Omega)}]ds}\\
\leq&\disp{C_{23}~~ \mbox{for all}~~ t\in(\tau,T_{max,\varepsilon})}\\
\end{array}
\label{zjccffgbhjcvvvbscz2.5297x96301ku}
\end{equation}
and
\begin{equation}
\begin{array}{rl}
&\| m (\cdot, t)\|_{W^{1,\infty}(\Omega)}\\
\leq&\disp{C_{24}\|(-\Delta+1)^\theta m (\cdot, t)\|_{L^{4}(\Omega)}}\\
\leq&\disp{C_{25}t^{-\theta}e^{-\lambda_1 t}\|m_0\|_{L^{4}(\Omega)}+C_{25}\int_{0}^t(t-s)^{-\theta}e^{-\lambda_1(t-s)}
\|(m -m n -u  \cdot \nabla m )(s)\|_{L^{4}(\Omega)}ds}\\
\leq&\disp{C_{26}+C_{26}\int_{0}^t(t-s)^{-\theta}e^{-\mu(t-s)}[\|m (s)\|_{L^{\infty}(\Omega)}+\|n (s)\|_{L^{4}(\Omega)}+\|u (s)\|_{L^\infty(\Omega)}
\|\nabla m (s)\|_{L^{4}(\Omega)}]ds}\\
\leq&\disp{C_{27}~~ \mbox{for all}~~ t\in(\tau,T_{max,\varepsilon})}\\
\end{array}
\label{zjccffgbhjcvbscz97x96u}
\end{equation}
with $\tau\in(0,T_{max,\varepsilon})$,
where  we have used \dref{ddfgczhhhh2.5ghju48cfg924ghyuji}, \dref{cz2ddfgjjj.5jddfghkkcvkkkllvvhjkfffffkhhgll},
 \dref{cz2ddfgjjj.5jkkcvvvhjkfffffkhhgll}, \dref{cz2ddfgjjj.5jssddddddfghkkcvkkkllvvhjkfffffkhhgll} as well as  the H\"{o}lder inequality and
$$\int_{0}^t(t-s)^{-\theta}e^{-\lambda_1(t-s)}\leq \int_{0}^{\infty}\sigma^{-\theta}e^{-\lambda_1\sigma}d\sigma<+\infty.$$

{\bf Step 5. The boundedness of $\|c (\cdot, t)\|_{W^{1,\infty}(\Omega)}$ and $\|n (\cdot, t)\|_{L^{\infty}(\Omega)}$   for all  $t\in (0, T_{max,\varepsilon})$}

 Recalling Lemma \ref{lemma70}, \dref{zjccffgbhjcvvvbscz2.5297x96301ku} and \dref{zjccffgbhjcvbscz97x96u}, we infer that
%
\begin{equation}
\begin{array}{rl}
\|\nabla c (\cdot, t)\|_{L^{{\infty}}(\Omega)}\leq \kappa_{1} ~~ \mbox{for all}~~~  t\in(0,T_{max,\varepsilon}) \\
\end{array}
\label{cz2.5g5ddfgggggghh56789hhjui78jj90099}
\end{equation}
and
\begin{equation}
\begin{array}{rl}
\|\nabla m (\cdot, t)\|_{L^{{\infty}}(\Omega)}\leq \kappa_{2} ~~ \mbox{for all}~~~  t\in(0,T_{max,\varepsilon}). \\
\end{array}
\label{cz2.5g5ddfggggggdffgghh56789hhjui78jj90099}
\end{equation}

{\bf Step 6. The boundedness of $\|n (\cdot, t)\|_{L^{\infty}(\Omega)}$   for all  $t\in (\tau, T_{max,\varepsilon})$ with $\tau\in(0,T_{max,\varepsilon})$}

Fix $T\in (0, T_{max,\varepsilon})$.  Let $M(T):=\sup_{t\in(0,T)}\|n (\cdot,t)\|_{L^\infty(\Omega)}$ and $\tilde{h} :=n S (x, n , c )\nabla c +u $. Then by  \dref{czfvgb2.5ghhjuyuccvviihjj}, \dref{x1.73142vghf48gg}  and \dref{cz2ddfgjjj.5jddfghkkcvkkkllvvhjkfffffkhhgll},
there exists $C_{28} > 0$ such that
\begin{equation}
\begin{array}{rl}
\|\tilde{h} (\cdot, t)\|_{L^{4}(\Omega)}\leq&\disp{C_{28}~~ t\in(0,T_{max,\varepsilon}),}\\
\end{array}
\label{cz2ddff.57151ccvhhjjjkkkuuifghhhivhccvvhjjjkkhhggjjllll}
\end{equation}
where we have used \dref{1.163072x} and  the boundedness of $\|c (\cdot, t)\|_{W^{1,\infty}(\Omega)}$  for all  $t\in (\tau, T_{max,\varepsilon})$ with $\tau\in(0,T_{max,\varepsilon})$.
Hence, due to the fact that $\nabla\cdot u =0$,  again,  by means of an
associate variation-of-constants formula for $n $, we can derive
\begin{equation}
n (t)=e^{(t-t_0)\Delta}n (\cdot,t_0)-\int_{t_0}^{t}e^{(t-s)\Delta}\nabla\cdot(n (\cdot,s)\tilde{h} (\cdot,s)) ds-\int_{t_0}^{t}e^{(t-s)\Delta}(n (\cdot,s)m (\cdot,s)) ds,~~ t\in(t_0, T),
\label{sss5555fghbnmcz2.5ghjjjkkklu48cfg924ghyuji}
\end{equation}
where $t_0 := (t-1)_{+}$.
As the last summand in \dref{5555fghbnmcz2.5ghjjjkkklu48cfg924ghyuji} is non-positive by the maximum principle, we can thus estimate
\begin{equation}
\|n (t)\|_{L^\infty(\Omega)}\leq \|e^{(t-t_0)\Delta}n (\cdot,t_0)\|_{L^\infty(\Omega)}+\int_{t_0}^{t}\|
e^{(t-s)\Delta}\nabla\cdot(n (\cdot,s)\tilde{h} (\cdot,s))\|_{L^\infty(\Omega)} ds,~~ t\in(t_0, T).
\label{5555fghbnmcz2.5ghjjjkkklu48cfg924ghyuji}
\end{equation}
If $t\in(0,1]$,
by virtue of the maximum principle, we derive that
\begin{equation}
\begin{array}{rl}
\|e^{(t-t_0)\Delta}n (\cdot,t_0)\|_{L^{\infty}(\Omega)}\leq &\disp{\|n_0\|_{L^{\infty}(\Omega)},}\\
\end{array}
\label{zjccffgbhjffghhjcghhhjjjvvvbscz2.5297x96301ku}
\end{equation}
while if $t > 1$ then with the help of the  $L^p$-$L^q$ estimates for the Neumann heat semigroup and \dref{ddfgczhhhh2.5ghju48cfg924ghyuji}, we conclude that
\begin{equation}
\begin{array}{rl}
\|e^{(t-t_0)\Delta}n (\cdot,t_0)\|_{L^{\infty}(\Omega)}\leq &\disp{C_{29}(t-t_0)^{-\frac{3}{2}}\|n (\cdot,t_0)\|_{L^{1}(\Omega)}\leq C_{30}.}\\
\end{array}
\label{zjccffgbhjffghhjcghghjkjjhhjjjvvvbscz2.5297x96301ku}
\end{equation}
Finally, we fix an arbitrary $\frac{7}{2}\in(3,4)$ and then once more invoke known smoothing
properties of the
Stokes semigroup  and the H\"{o}lder inequality to find $C_4 > 0$ such that
\begin{equation}
\begin{array}{rl}
&\disp\int_{t_0}^t\| e^{(t-s)\Delta}\nabla\cdot(n (\cdot,s)\tilde{h} (\cdot,s)\|_{L^\infty(\Omega)}ds\\
\leq&\disp C_{31}\int_{t_0}^t(t-s)^{-\frac{1}{2}-\frac{3}{7}}\|n (\cdot,s)\tilde{h} (\cdot,s)\|_{L^p(\Omega)}ds\\
\leq&\disp C_{32}\int_{t_0}^t(t-s)^{-\frac{1}{2}-\frac{3}{7}}\| n (\cdot,s)\|_{L^{28}(\Omega)}\|\tilde{h} (\cdot,s)\|_{L^{4}(\Omega)}ds\\
\leq&\disp C_{33}\int_{t_0}^t(t-s)^{-\frac{1}{2}-\frac{3}{7}}\| u (\cdot,s)\|_{L^{\infty}(\Omega)}^\frac{27}{28}\| u (\cdot,s)\||_{L^1(\Omega)}^{\frac{1}{28}}\|\tilde{h} (\cdot,s)\|_{L^{4}(\Omega)}ds\\
\leq&\disp C_{34}M^b(T)~~\mbox{for all}~~ t\in(0, T),\\
\end{array}
\label{ccvbccvvbbnnndffghhjjvcvvbccfbbnfgbghjjccmmllffvvggcvvvvbbjjkkdffzjscz2.5297x9630xxy}
\end{equation}
In combination with \dref{5555fghbnmcz2.5ghjjjkkklu48cfg924ghyuji}--\dref{ccvbccvvbbnnndffghhjjvcvvbccfbbnfgbghjjccmmllffvvggcvvvvbbjjkkdffzjscz2.5297x9630xxy} and using the definition of $M(T)$
we obtain
$C_{35}> 0$ such that
\begin{equation}
\begin{array}{rl}
&\disp  M(T)\leq C_{35}+C_{35}M^{\frac{27}{28}}(T)~~\mbox{for all}~~ T\in(0, T_{max,\varepsilon}).\\
\end{array}
\label{ccvbccvvbbnnndffghhjjvcvvfghhhbccfbbnfgbghjjccmmllffvvggcvvvvbbjjkkdffzjscz2.5297x9630xxy}
\end{equation}
Hence,  with  some basic calculation, in light of  $T\in (0, T_{max,\varepsilon})$ was arbitrary,
one can get
\begin{equation}
\begin{array}{rl}
\|n (\cdot, t)\|_{L^{\infty}(\Omega)}\leq&\disp{C_{36}~~ \mbox{for all}~~ t\in(0,T_{max,\varepsilon}).}\\
\end{array}
\label{cz2.57ghhhh151ccvhhjjjkkkffgghhuuiivhccvvhjjjkkhhggjjllll}
\end{equation}

Finally,
 by virtue of Lemma \ref{lemma70} and
\dref{cz2.571hhhhh51ccvvhddfccvvhjjjkkhhggjjllll}, \dref{cz2.5g5ddfgggggghh56789hhjui78jj90099}--\dref{cz2.5g5ddfggggggdffgghh56789hhjui78jj90099}, \dref{cz2.57ghhhh151ccvhhjjjkkkffgghhuuiivhccvvhjjjkkhhggjjllll},  the local solution can be extend to the global-in-time solutions.

Employing almost exactly the same arguments as in the proof of Lemma 3.1 in  \cite{LiLiLiLisssdffssdddddddgssddsddfff00} (see also \cite{Zhenddsdddddgssddsddfff00}), and taking advantage of \dref{1.163072xggttsdddyyu}, we conclude the regularity theories for the Stokes semigroup and the   H\"{o}lder estimate for local solutions of parabolic equations, we can
obtain weak solution $(n,c,m,u)$   is a classical solution.
\end{proof}

The most important consequence of Lemmas \ref{lemma45hyuuuj630223}--\ref{lemmassddddff45630223} is the following:

{\bf Proof of Theorem  \ref{theorem3}}: The theorem \ref{theorem3} is part of the statement proven by Lemmas \ref{lemma45hyuuuj630223}--\ref{lemmassddddff45630223}.


\subsection{Eventual smoothness and asymptotics}


Given the preliminary lemma collected in the above, in this subsection, we now establish the claimed asymptotic behavior
of the solutions to \dref{1.1fghyuisda} under  $\alpha>0$.
Before going further, we list the following lemma, which will be used 
to derive the convergence properties of solution with respect
to the norm in $L^2 (\Omega)$.



\begin{lemma}\label{fhhghfbglemma4563025xxhjklojjkkkgyhuissddff} (Lemma 4.6 of \cite{EspejojjEspejojainidd793})
Let $\lambda > 0, C > 0$, and suppose that $y\in C^1 ([0,\infty))$ and
$h\in  C^0 ([0,\infty))$ are nonnegative functions satisfying $y'(t)+\lambda y(t)\leq h(t)$ for some $\lambda > 0$ and all
$t > 0$. Then if
$\int_0^\infty h(s)ds \leq C,$ we have $\lim_{t\rightarrow\infty}y(t)=0$.
\end{lemma}
To begin with, let us collect some basic solution properties which essentially have already been used
in \cite{EspejojjEspejojainidd793}.
\begin{lemma}\label{lemmadsssddffffdfffgg4dddd5630}
The  global solution of \dref{1.1fghyuisda} satisfies
\begin{equation}
\begin{array}{rl}
\disp\disp\int_0^\infty\int_\Omega \left(n_{\varepsilon}m_{\varepsilon}+|\nabla m_{\varepsilon}|^2\right)<&\disp{+\infty.}\\
\end{array}
\label{hhxxcdfvvsssjjdfffssddcz2.5}
\end{equation}
\end{lemma}
\begin{proof}
These properties are immediate consequences of \dref{ddczhjjjj2.5ghju48cfgffff924} and \dref{ddczhjjjj2.5ghju48cfg9ssdd24}.
\end{proof}
As an immediate consequence, we obtain the following which will firstly serve as a
fundament for our proof of stabilization in the first  and third solution components.
\begin{lemma}\label{lemmaddffffdfffgg4sssdddd5630}
Under the assumptions of Lemma \ref{lemma45hyuuuj630223}, for any $\eta> 0$, there are $T > 0$ and $\varepsilon_0>0$ such that for any $t > T$ and
such that for any $\varepsilon\in(0,\varepsilon_0 )$
\begin{equation}0\leq\frac{1}{|\Omega|}\int_{\Omega}m_\varepsilon(\cdot,t)-\hat{m}<\eta~~~\mbox{for  any}~~~t>T
\label{11111hhxxcdfvhhhvssssssfftggggsssjjghjjsdggggdddfffddffssddcssdz2.5}
\end{equation}
and
\begin{equation}0\leq\frac{1}{|\Omega|}\int_{\Omega}n_\varepsilon(\cdot,t)-\hat{n}<\eta~~~\mbox{for  any}~~~t>T,
\label{11111hhxxcdfvhhhvsssssssssfggjjsdggggdddfffddffddfggssddcssdz2.5}
\end{equation}
where \begin{equation}
\hat{m}=\left\{\frac{1}{|\Omega|}\int_{\Omega}m_{0}-\frac{1}{|\Omega|}\int_{\Omega}n_{0}\right\}_{+}
\label{1111hhxxcdfvhhhvsddfffgssjjdfffsfffsddcsssz2.5}
\end{equation}
and
\begin{equation}
\hat{n}=\left\{\frac{1}{|\Omega|}\int_{\Omega}n_{0}-\frac{1}{|\Omega|}\int_{\Omega}m_{0}\right\}_{+}.
\label{1111hhxxcddffdfvhhhvsddfffgssdfffjjdfffssddcsssz2.5}
\end{equation}
\end{lemma}
\begin{proof}
Pursuing a strategy demonstrated in lemma 4.2 of \cite{Winkler61215}, we start by noting that as a first
consequence of Lemma \ref{lemmadsssddffffdfffgg4dddd5630} we know that
\begin{equation}
\begin{array}{rl}
\disp\disp\int_{t-1}^t\int_\Omega \left(n_\varepsilon m_\varepsilon+|\nabla m_\varepsilon|^2\right)\rightarrow&\disp{0~~~\mbox{as}~~t\rightarrow\infty.}\\
\end{array}
\label{1111hhxxcdfvhhhvsssjjdfffssddcsssz2.5}
\end{equation}
Next, in view of \dref{czfvgb2.5ghhjuyuccvviihjj}, by using the H\"{o}lder inequality and  the Poincar\'{e} inequality, for some positive constant $K$,
\begin{equation}
\begin{array}{rl}
\disp\disp\int_{t-1}^t\int_\Omega n_\varepsilon m_\varepsilon=&\disp{\int_{t-1}^t\int_\Omega n_\varepsilon(m_{\varepsilon}-\bar{m})+\int_{t-1}^t\bar{m}\int_\Omega n_{\varepsilon}}\\
\geq&\disp{-\int_{t-1}^t\|n_{\varepsilon}\|_{L^2(\Omega)}\|m_{\varepsilon}-\bar{m}\|_{L^2(\Omega)}+\int_{t-1}^t\bar{m}\int_{\Omega}n_\varepsilon(x,s)dxds}\\
\geq&\disp{-K\int_{t-1}^t\|\nabla m_{\varepsilon}\|_{L^2(\Omega)}+\frac{1}{|\Omega|}\int_{t-1}^t\left[\int_{\Omega}m_\varepsilon(x,s)dx\int_{\Omega}n_{\varepsilon}(x,s)dx\right]ds}\\
\geq&\disp{-K\left(\int_{t-1}^t\|\nabla m_{\varepsilon}\|_{L^2(\Omega)}^2\right)^{\frac{1}{2}}+\frac{1}{|\Omega|}\int_{t-1}^t\left[\int_{\Omega}m_\varepsilon(x,s)dx\int_{\Omega}n_{\varepsilon}(x,s)dx\right]ds.}\\
\end{array}
\label{11112222hhxxcdfvssdhhhvsssjjdfffssddcsssz2.5}
\end{equation}
Inserting \dref{1111hhxxcdfvhhhvsssjjdfffssddcsssz2.5} into \dref{11112222hhxxcdfvssdhhhvsssjjdfffssddcsssz2.5}, we obtain
\begin{equation}
\begin{array}{rl}
\disp\int_{t-1}^t\left[\int_{\Omega}m_\varepsilon(x,s)dx\int_{\Omega}n_{\varepsilon}(x,s)dx\right]ds\rightarrow0~~~\mbox{as}~~t\rightarrow\infty.
\end{array}
\label{11111111111hhxxcdfvssdhhhvsssjjdfffssddcsssz2.5}
\end{equation}
Now if $\int_{\Omega}n_ 0-\int_{\Omega}m_ 0 \geq 0,$ \dref{sssddfgczhhhh2.5ghju48cfg924ghyuji} warrants that $\int_{\Omega}n_{\varepsilon}-\int_{\Omega}m_{\varepsilon} \geq 0$, which along with \dref{11111111111hhxxcdfvssdhhhvsssjjdfffssddcsssz2.5} implies that
\begin{equation}
\begin{array}{rl}
\disp\int_{t-1}^t\left(\int_{\Omega}m_\varepsilon(x,s)dx\right)^2ds\rightarrow0~~~\mbox{as}~~t\rightarrow\infty.
\end{array}
\label{hhxxcdfvssdhhhvssssssjjdfffssddcsssz2.5}
\end{equation}
Noticing that $\int_{\Omega}m_\varepsilon(s) \geq \int_{\Omega}m_\varepsilon(t)$ for all $t\geq s,$ we have
$$0\leq \left(\int_{\Omega}m_\varepsilon(x,t)dx\right)^2
\leq \int_{t-1}^t\left(\int_{\Omega}m_\varepsilon(x,s)dx\right)^2ds\rightarrow0~~~\mbox{as}~~t\rightarrow\infty,
$$
where we invoke \dref{sssddfgczhhhh2.5ghju48cfg924ghyuji} to obtain
$$\int_{\Omega}n_{\varepsilon}(\cdot,t)\rightarrow  \int_{\Omega}n_{0}-\int_{\Omega}m_{0}~~\mbox{as}~~t\rightarrow\infty.$$
By very similar argument, one can see that
$\int_{\Omega}n_{\varepsilon}\rightarrow 0$  and $\int_{\Omega}m_{\varepsilon}\rightarrow \int_{\Omega}m_{0}-\int_{\Omega}n_{0}$ as $t\rightarrow\infty$ in the case of $\int_{\Omega}n_0 - \int_{\Omega}m_0 < 0$.
This readily establishes \dref{11111hhxxcdfvhhhvssssssfftggggsssjjghjjsdggggdddfffddffssddcssdz2.5} and \dref{11111hhxxcdfvhhhvsssssssssfggjjsdggggdddfffddffddfggssddcssdz2.5}.
\end{proof}
\begin{lemma}\label{ssdddlemmddddaddffffdfffgg4sssdddd5630}
Under the assumptions of Lemma \ref{lemma45hyuuuj630223},  for any $\eta > 0$, there are $T > 0$ and $\varepsilon_0>0$ such that for any $t > T$ and
such that for any $\varepsilon\in(0,\varepsilon_0 )$
\begin{equation}\int_{t}^{t+1}\|m_\varepsilon(\cdot,t)-\hat{m}\|_{L^\infty(\Omega)}<\eta
\label{11111hhxxcdfvhhhvsssssssssjjghjjsdggggdddfffddffssddcssdz2.5}
\end{equation}
and
\begin{equation}\|m_\varepsilon(\cdot,t)-\hat{m}\|_{L^p(\Omega)}< \eta,
\label{11111hhxxcdfvhhhvsssdddjjkddffkksssssssscccjjghjjsdggggdddfffddffssddcssdz2.5}
\end{equation}
where  $\hat{m}$ is give by \dref{1111hhxxcdfvhhhvsddfffgssjjdfffsfffsddcsssz2.5}.
\end{lemma}
\begin{proof}
Firstly,  since Lemma \ref{lemma45630hhuujj} asserts the existence of $\kappa_1$  such that
\begin{equation}
\int_t^{t+1}\|\nabla m_\varepsilon (\cdot,t)\|_{L^{4}(\Omega)}^4  \leq \kappa_1
\label{1111hhxxcdfvhhhvsssssssssjjdfffddffssddcssdz2.ssdd5}
\end{equation}
and since  \dref{sssddfgczhhhh2.5ghju48cfg924ghyuji} implies that
\begin{equation}\frac{1}{|\Omega|}\int_{\Omega}m_\varepsilon(\cdot,t)\geq \left\{\frac{1}{|\Omega|}\int_{\Omega}m_{0}-\frac{1}{|\Omega|}\int_{\Omega}n_{0}\right\}_{+},
\label{11111hhxxcdfvhssdddhhvsssssssssjjghjjsdggggdddfffddffssddcssdz2.5}
\end{equation}
thus, by \dref{11111hhxxcdfvhhhvssssssfftggggsssjjghjjsdggggdddfffddffssddcssdz2.5} we  infer 
from the interpolation inequality and the H\"{o}lder
inequality that
\begin{equation}\label{fvgbccvvhnjmkfhhhhhgbdffrhnkkkn6291}
\begin{array}{rl}
\disp
&\int_t^{t+1}\|m_\varepsilon-\hat{m}\|_{L^\infty(\Omega)}ds\\
\leq&{ C(\disp\int_t^{t+1}(\|\nabla m_\varepsilon\|_{L^{4}(\Omega)}^{\frac{12}{13}}\|m_\varepsilon-\hat{m}\|_{L^1(\Omega)}^{\frac{1}{13}}+\|m_\varepsilon-\hat{m}\|_{L^1(\Omega)})ds}\\
\leq&{ C\disp\int_t^{t+1}\|\nabla
m_\varepsilon\|_{L^{4}(\Omega)}^{4}ds)^{\frac{3}{13}}(\disp\int_t^{t+1}\|m_\varepsilon-\hat{m}\|_{L^1(\Omega)}ds)^{\frac{1}{13}}+C\disp\int_t^{t+1}\|m_\varepsilon-\hat{m}\|_{L^1(\Omega)})ds}\\
\leq&{ C\disp\int_t^{t+1}\|\nabla
m_\varepsilon\|_{L^{4}(\Omega)}^{4}ds)^{\frac{3}{13}}(\sup_{t>0}\|m_\varepsilon(\cdot,t)-\hat{m}\|_{L^1(\Omega)})^{\frac{1}{13}}+C\disp\sup_{t>0}\|m_\varepsilon(\cdot,t)-\hat{m}\|_{L^1(\Omega)})}\\
\rightarrow&{0~~~\mbox{as}~~t\rightarrow+\infty,}\\
\end{array}
\end{equation}
which immediately implies \dref{11111hhxxcdfvhhhvsssssssssjjghjjsdggggdddfffddffssddcssdz2.5}.
Here we have used the fact that
$$\|m_\varepsilon-\hat{m}\|_{L^1(\Omega)}=\int_{\Omega}[m_\varepsilon(\cdot,t)-\hat{m}]=
|\Omega|\left[\frac{1}{|\Omega|}\int_{\Omega}m_\varepsilon(\cdot,t)-\hat{m}\right]\rightarrow0~~~\mbox{as}~~t\rightarrow+\infty$$
by using \dref{11111hhxxcdfvhhhvssssssfftggggsssjjghjjsdggggdddfffddffssddcssdz2.5}.
Next,  
for any $p>1$, in view of
Lemma \ref{fvfgsdfggfflemma45},  we derive from the
the interpolation and the H\"{o}lder
inequality that
\begin{equation}\label{fvgbccvsddfgddffvhnsdfffjmkfhhhhhgbdffrhnkkkn6291}
\begin{array}{rl}
\disp
&\|m_{\varepsilon}-\hat{m}\|_{L^p(\Omega)}\\
\leq&{ \disp\|m_{\varepsilon}-\hat{m}\|_{L^{\infty}(\Omega)}^{\frac{p-1}{p}}\|m_{\varepsilon}-\hat{m}\|_{L^1(\Omega)}^{\frac{1}{p}}}\\
\rightarrow&{0~~~\mbox{as}~~t\rightarrow+\infty,}\\
\end{array}
\end{equation}
which yields \dref{11111hhxxcdfvhhhvsssdddjjkddffkksssssssscccjjghjjsdggggdddfffddffssddcssdz2.5} directly.
\end{proof}
\begin{lemma}\label{sedddlemmaddffffdfffgg4sssdddd5630}
Under the assumptions of Lemma \ref{lemma45hyuuuj630223}, for any $\eta > 0$, there are $T > 0$ and $\varepsilon_0>0$ such that for any $t > T$ and
such that for any $\varepsilon\in(0,\varepsilon_0 )$
\begin{equation}\|c_\varepsilon(\cdot,t)-\hat{m}\|_{L^2(\Omega)}<\eta
\label{11111hhxxcdfvhhhvsssssssssjjsdgkkkgggdddfffddffssddcssdz2.5}
\end{equation}
and
\begin{equation}
\begin{array}{rl}
\disp\disp\int_t^{t+1}\int_\Omega |\nabla c_\varepsilon|^2&<\disp{\eta.}\\
\end{array}
\label{hhxxcdfvhhhvsssssddffsssjjssdfffssddcssdz2.5}
\end{equation}
where  $\hat{m}$ is give by \dref{1111hhxxcdfvhhhvsddfffgssjjdfffsfffsddcsssz2.5}.
\end{lemma}
\begin{proof}
Firstly, by means of the testing procedure, 
 we may derive from  the Young inequality that 
\begin{equation}
\begin{array}{rl}
&\disp{\frac{1}{2}\frac{d}{dt}\|c_\varepsilon-\hat{m}\|^{{2}}_{L^{{2}}(\Omega)}}
\\
=&\disp{
\int_{\Omega}(c_\varepsilon-\hat{m})[\Delta c_\varepsilon-u_\varepsilon\cdot\nabla c_\varepsilon-(c_\varepsilon-\hat{m})+(m_\varepsilon-\bar{m})]}\\
=&\disp{
\int_{\Omega}(c_\varepsilon-\hat{m})(\Delta c_\varepsilon-u_\varepsilon\cdot\nabla c_\varepsilon)-\int_{\Omega}(c_\varepsilon-\hat{m})^2+\int_{\Omega}(c_\varepsilon-\hat{m})(m_\varepsilon-\hat{m})}\\
\leq&\disp{-
\int_{\Omega}|\nabla c_\varepsilon|^2-\int_{\Omega}(c_\varepsilon-\hat{m})^2+\frac{1}{2}\int_{\Omega}(m_\varepsilon-\hat{m})^2}\\
\leq&\disp{-\int_{\Omega}(c_\varepsilon-\hat{m})^2+\frac{1}{2}\int_{\Omega}(m_\varepsilon-\hat{m})^2~~\mbox{for all}~~ t>0, }\\
\end{array}
\label{wwwwwcz2.511ssssdfggsddffggg4ddfggg114}
\end{equation}
where we have used the fact that $\nabla\cdot u_\varepsilon = 0$ and $u_\varepsilon |_{\partial\Omega} = 0$.
On the other hand, the bounds from \ref{ssdddlemmddddaddffffdfffgg4sssdddd5630} entails 
\begin{equation}
\begin{array}{rl}
\disp\int_{t}^{t+1}\int_{\Omega}(m_\varepsilon-\hat{m})^2ds\rightarrow0~~~\mbox{as}~~t\rightarrow\infty.
\end{array}
\label{hhxxcdfvssdhhhvssssssjjdfffssddcsssddssz2.5}
\end{equation}
This together with \dref{wwwwwcz2.511ssssdfggsddffggg4ddfggg114} and Lemma \ref{fhhghfbglemma4563025xxhjklojjkkkgyhuissddff} imply  \dref{11111hhxxcdfvhhhvsssssssssjjsdgkkkgggdddfffddffssddcssdz2.5} and \dref{hhxxcdfvhhhvsssssddffsssjjssdfffssddcssdz2.5}.
\end{proof}
\begin{lemma}\label{11aaalemdfghkkmaddffffdfffgg4sssdddd5630}
Under the assumptions of Lemma \ref{lemma45hyuuuj630223}, for any $p > 1$ and $\eta>0$, there are $T > 0$ and $\varepsilon_0>0$ such that for any $t > T$ and
such that for any $\varepsilon\in(0,\varepsilon_0 )$
\begin{equation}\|n_\varepsilon(\cdot,t)-\hat{n}\|_{L^p(\Omega)}< \eta,
\label{11111hhxxcdfvhhhvssssssssscccjjghjjsdggggdddfffddffssddcssdz2.5}
\end{equation}
where $\hat{n}$ is given by \dref{1111hhxxcddffdfvhhhvsddfffgssdfffjjdfffssddcsssz2.5}.
\end{lemma}
\begin{proof}
Firstly, for any 
$p>1$, by Lemma \ref{lemmaghjffggssddgghhmk4563025xxhjklojjkkk}, there exist  positive constants $\alpha_1$ and $q>p$ such that
\begin{equation}
\begin{array}{rl}
&\disp{\int_{\Omega} n_\varepsilon ^{q}(x,t)\leq \alpha_1~~~\mbox{for all}~~ t>0.}\\
\end{array}
\label{czfvgb2.5ghhjussdyuccvviihjj}
\end{equation}
By the interpolation and the H\"{o}lder
inequality, we have
\begin{equation}\label{fvgbccvsddfgddffvhnjmkfhhhhhgbdffrhnkkkn6291}
\begin{array}{rl}
\disp
\|n_{\varepsilon}-\hat{n}\|_{L^p(\Omega)}
\leq&{ \disp\|n_{\varepsilon}-\hat{n}\|_{L^{q}(\Omega)}^{\frac{q(p-1)}{p(q-1)}}\|n_{\varepsilon}-\hat{n}\|_{L^1(\Omega)}^{\frac{q-p}{p(q-1)}}}\\
\rightarrow&{0~~~\mbox{as}~~t\rightarrow+\infty}\\
\end{array}
\end{equation}
by using \dref{11111hhxxcdfvhhhvsssssssssfggjjsdggggdddfffddffddfggssddcssdz2.5}.
From \dref{fvgbccvsddfgddffvhnjmkfhhhhhgbdffrhnkkkn6291} we readily derive \dref{11111hhxxcdfvhhhvssssssssscccjjghjjsdggggdddfffddffssddcssdz2.5} and thereby completes the proof.
\end{proof}


The stabilization property implied by Lemmas \ref{ssdddlemmddddaddffffdfffgg4sssdddd5630} and \ref{11aaalemdfghkkmaddffffdfffgg4sssdddd5630} can now be turned into a preliminary statement on decay
of $u_\varepsilon$ by making use of Lemma \ref{fhhghfbglemma4563025xxhjklojjkkkgyhuissddff} and
the standard testing procedures.
\begin{lemma}\label{11aaalemmaddffffdsddfffffgg4sssdddd5630}
Under the assumptions of Lemma \ref{lemma45hyuuuj630223},  for any $\eta > 0$,  there are $T > 0$ and $\varepsilon_0>0$ such that for any $t > T$ and
such that for any $\varepsilon\in(0,\varepsilon_0 )$ 
\begin{equation}\|u_\varepsilon(\cdot,t)\|_{L^2(\Omega)}<\eta,
\label{11111hhxxcdfvhhhvssssssssscccjjghjjsdgggddddgdddfffddffssddcssdz2.5}
\end{equation}
\begin{equation}\int_{t}^{t+1}\|\nabla u_\varepsilon\|_{L^2(\Omega)}^2dx<\eta
\label{11111hhxxcdfvhhhvsssssddfffssddffsscccjjghjjsdggggdddfffddffssddcssdz2.5}
\end{equation}
as well as
\begin{equation}\int_{t}^{t+1}\|u_\varepsilon\|_{L^q(\Omega)}^2dx<\eta
\label{11111hhxxcdfvhhhvsssddffssddfffssddffsscccjjghjjsdggggdddfffddffssddcssdz2.5}
\end{equation}
and
\begin{equation}\int_{t}^{t+1}\|u_\varepsilon\|_{L^q(\Omega)}dx<\eta
\label{11111hhxxcdfvhhhvddfffsssddffssddfffssddffsscccjjghjjsdggggdddfffddffssddcssdz2.5}
\end{equation}
for any $q\in[1,6)$.
\end{lemma}
\begin{proof}
From the fourth equation in \dref{1.1fghyuisda} we obtain the associated Navier-Stokes energy inequality
in the form
\begin{equation}
\begin{array}{rl}
&\disp{\frac{1}{2}\frac{d}{dt}\|u_{\varepsilon}\|^{{2}}_{L^{{2}}(\Omega)}}
\\
=&\disp{-
\int_{\Omega}|\nabla u_{\varepsilon}|^2+\int_{\Omega}(n_{\varepsilon}+m_{\varepsilon})\nabla\phi\cdot u_{\varepsilon}-\int_{\Omega}\nabla P_{\varepsilon}\cdot u_{\varepsilon}}\\
=&\disp{-
\int_{\Omega}|\nabla u_{\varepsilon}|^2+\int_{\Omega}(n_{\varepsilon}-\hat{n}+m_{\varepsilon}-\hat{m})\nabla\phi\cdot u_{\varepsilon}}\\
\leq&\disp{-
\int_{\Omega}|\nabla u_{\varepsilon}|^2+K_1\left(\int_{\Omega}(n_{\varepsilon}-\hat{n}+m_{\varepsilon}-\hat{m})^2\right)^{\frac{1}{2}}\left(\int_{\Omega} |u_{\varepsilon}|^2\right)^{\frac{1}{2}}}\\
\leq&\disp{-
\int_{\Omega}|\nabla u_{\varepsilon}|^2+K_1\left(\int_{\Omega}|n_{\varepsilon}-\hat{n}|^2+\int_{\Omega}|m_{\varepsilon}-\hat{m}|^2\right)^{\frac{1}{2}}\left(\int_{\Omega} |u_{\varepsilon}|^2\right)^{\frac{1}{2}},}\\
\end{array}
\label{cddddz2.51kkk1ssssdfssddjjkkkggsddffggg4ddfggg114}
\end{equation}
where we have used the fact that $\hat{n}=-\hat{m}$ as well as $\nabla\cdot u_{\varepsilon} = 0$ and $u_{\varepsilon} |_{\partial\Omega} = 0$.
Due to the Poincar\'{e} inequality again, we have
$$ \eta_0\int_{\Omega} |u_{\varepsilon}|^2\leq \int_{\Omega}|\nabla u_{\varepsilon}|^2,$$
therefore, collecting \dref{11111hhxxcdfvhhhvssssssssscccjjghjjsdggggdddfffddffssddcssdz2.5} and  \dref{11111hhxxcdfvhhhvsssssssssjjghjjsdggggdddfffddffssddcssdz2.5}, we derive from \dref{cddddz2.51kkk1ssssdfssddjjkkkggsddffggg4ddfggg114} that
\begin{equation}
\begin{array}{rl}
\disp\disp \lim_{t\rightarrow+\infty}\int_{\Omega}|u_{\varepsilon}(x,t)|^2dx&=\disp{0}\\
\end{array}
\label{hhxxcdfvhhhvsssssssssjjdfffddffssddcssdz2.5}
\end{equation}
and
\begin{equation}\int_{t}^{t+1}\|\nabla u_{\varepsilon}\|_{L^2(\Omega)}^2dx\rightarrow 0 ~~~\mbox{as}~~~t\rightarrow\infty
\label{11111hhxxcdfvhhhvddddsssssddfffssddffsscccjjghjjsdggggdddfffddffssddcssdz2.5}
\end{equation}
and thereby proves \dref{11111hhxxcdfvhhhvssssssssscccjjghjjsdgggddddgdddfffddffssddcssdz2.5}--\dref{11111hhxxcdfvhhhvsssssddfffssddffsscccjjghjjsdggggdddfffddffssddcssdz2.5}.
Finally, we make use of the embedding
$W^{1,2}(\Omega)\hookrightarrow L^q (\Omega)$ (for any $q\in[1,6)$) and the
Young inequality  to find 
 that \dref{11111hhxxcdfvhhhvsssddffssddfffssddffsscccjjghjjsdggggdddfffddffssddcssdz2.5} and \dref{11111hhxxcdfvhhhvddfffsssddffssddfffssddffsscccjjghjjsdggggdddfffddffssddcssdz2.5} hold.
\end{proof}
Using thge decay property of $m_\varepsilon(\cdot,t)-\hat{m}+n_\varepsilon(\cdot,t)
-\hat{n}$ (see Lemmas \ref{ssdddlemmddddaddffffdfffgg4sssdddd5630} and \ref{11aaalemdfghkkmaddffffdfffgg4sssdddd5630}), by means of a contraction mapping argument, we may  derive a certain eventual regularity and
decay of $u_\varepsilon$ in $L^p(\Omega)$ with some $p\geq6$.
\begin{lemma}\label{aaalemmaddffffsddddfffgg4sssdddd5630}
 For any $p \in[6,\infty)$ and $\eta>0$, there are $T > 0$ and $\varepsilon_0>0$ such that for any $t > T$ and
such that for any $\varepsilon\in(0,\varepsilon_0 )$
 \begin{equation}\|u_{\varepsilon}(\cdot,s)\|_{L^p(\Omega)}<\eta~~~\mbox{for any}~~~ s\in[t,t + 1].
\label{11111hhxxcdfvhhhvssssssssscccjjghjjsdgggddddgdddffddfffddffssddcssdz2.5}
\end{equation} 
\end{lemma}
\begin{proof}
We let $p\geq 6$ and choose $q \in (3,6)$ which is close to $6$ (e.g. $q=\frac{6p}{p+2}$) such that
\begin{equation}
\begin{array}{rl}
\disp{\frac{1}{2}-\frac{3}{2p}-3(\frac{1}{q}-\frac{1}{p})} >\disp{0.}\\
\end{array}
\label{hhxxcdfvhhhvsssjjdfffssssddssdddcz2.5}
\end{equation}
Now, we define $\gamma_0=\frac{3}{2}(\frac{1}{q}-\frac{1}{p})$.
Next, in view of \dref{hhxxcdfvhhhvsssjjdfffssssddssdddcz2.5} and using  $p\geq6$, $q\in(3,6)$, we have
 $$-2\gamma_0-\frac{1}{2}-\frac{3}{2p}>-1~~~\mbox{and}~~ -\frac{1}{2}-\frac{3}{2p}>-1.$$
 Therefore,
\begin{equation}
2\kappa_2\int_0^3s^{-2\gamma_0-\frac{1}{2}-\frac{3}{2p}}ds<+\infty~~~\mbox{and}~~ 2\kappa_2\int_0^3s^{-\frac{1}{2}-\frac{3}{2p}}ds<+\infty,
 \label{2222dddddf1.1ddfghssdddyddffssddduisda}
\end{equation}
where $\kappa_2$ is give by Lemma \ref{llssdrffmmggnnccvvccvvkkkkgghhkkllvvlemma45630}.
Thus for any $\eta>0$, we any  choose  $\eta_0\in(0,\eta)$ small enough such that
\begin{equation}
\eta_0<\min\{\frac{1}{2\kappa_2\int_0^3s^{-2\gamma_0-\frac{1}{2}-\frac{3}{2p}}ds},\frac{1}{2\kappa_2\int_0^3s^{-\frac{1}{2}-\frac{3}{2p}}ds}\}.
 \label{2222dddddf1.1ddfghssdddysddddddffssddduisda}
\end{equation}
On the other hand, by \dref{11111hhxxcdfvhhhvddfffsssddffssddfffssddffsscccjjghjjsdggggdddfffddffssddcssdz2.5}, Lemmas \ref{ssdddlemmddddaddffffdfffgg4sssdddd5630}--\ref{11aaalemmaddffffdsddfffffgg4sssdddd5630}, we then pick $T_0$ and $\varepsilon_0>0$ such that for any $t > T$ and
such that for any $\varepsilon\in(0,\varepsilon_0 )$ 
\begin{equation}\int_{t}^{t+1}\|u_\varepsilon\|_{L^q(\Omega)}dx<\frac{\eta_0}{3\kappa_1}~~~
\mbox{and}~~\|m_\varepsilon(\cdot,t)-\hat{m}+n_\varepsilon(\cdot,t)-\hat{n}\|_{L^p(\Omega)}<\frac{\eta_0}{3^{\gamma_0+1}\kappa_3\|\nabla\phi\|_{L^\infty(\Omega)}},
\label{11111hhxxcdfvhhhvddfssdddffsssddffssddfffssddffsscccjjghjjsdggggdddfffddffssddcssdz2.5}
\end{equation}
where $\kappa_1$ and $\kappa_3$ are given by Lemma  \ref{llssdrffmmggnnccvvccvvkkkkgghhkkllvvlemma45630}.
In view of \dref{11111hhxxcdfvhhhvddfssdddffsssddffssddfffssddffsscccjjghjjsdggggdddfffddffssddcssdz2.5}, for any $t_1>T_0$ and $\varepsilon\in (0,\varepsilon_{t_1})$, one can find $\tilde{t}_0\in(t_1,t_1+1)$ such that
\begin{equation}\|u_\varepsilon(\cdot,\tilde{t}_0)\|_{L^q(\Omega)}dx<\frac{\eta_0}{3\kappa_1}.
\label{11111hhxxcdfvhhhvdssddddfssdddffsssddffssddfffssddffsscccjjghjjsdggggdddfffddffssddcssdz2.5}
\end{equation}
Now, we define
$$T=T_0+2.$$
In the following, we will prove that \dref{11111hhxxcdfvhhhvssssssssscccjjghjjsdgggddddgdddffddfffddffssddcssdz2.5} holds for any $t>T.$
To this end, for the above $\tilde{t}_0,p,\gamma_0$ and $\eta_0$, we let
\begin{equation}X=\{v:\Omega\times(\tilde{t}_0,\tilde{t}_0+3)\rightarrow \mathbb{R};\sup_{s\in(0,3)}s^{\gamma_0}\|v(\tilde{t}_0+s)\|_{L^p(\Omega)}\leq\eta_0\}.
\label{cz2.571hhhhh51lllllccvvddfgghddfccvvhjjjkkhhggjjlsdddlll}
\end{equation}
Then we consider the mapping
$\varphi: X\rightarrow \mathbb{R}$ defined by
$$\varphi(v) = e^{-tA}u_{\varepsilon}(\tilde{t}_0) +\int_{\tilde{t}_0}^te^{-(t-\tau)A}
\mathcal{P}\left[-\kappa\nabla\cdot(Y_\varepsilon v\otimes v) (\tau)+(n_\varepsilon (\tau)+m_\varepsilon (\tau))\nabla\phi\right]d\tau.$$
Now, we will show that  $\varphi$  is a contraction on $X$.
In fact,  in view of Lemma \ref{llssdrffmmggnnccvvccvvkkkkgghhkkllvvlemma45630}, for any $s>1$ and for any such $v$  we may
derive from the H\"{o}lder inequality that
\begin{equation}
\begin{array}{rl}
&\|\varphi(v) (\cdot, t)\|_{L^p(\Omega)}\\
\leq&\disp{\kappa_1(t-\tilde{t}_0)^{-\gamma_0}\|u_\varepsilon(\tilde{t}_0)\|_{L^q(\Omega)} +
\kappa_2\int_{\tilde{t}_0}^t(t-\tau)^{-\frac{1}{2}-\frac{3}{2}(\frac{2}{p}-\frac{1}{p})}\|v\oplus v\|_{L^{\frac{p}{2}}(\Omega)}d\tau}\\
&\disp{+\kappa_3\|\nabla\phi\|_{L^\infty(\Omega)}\int_{\tilde{t}_0}^t\|m_\varepsilon(\cdot,\tau)-\hat{m}+n_\varepsilon(\cdot,\tau)-\hat{n}\|_{L^p(\Omega)}d\tau}\\
\leq&\disp{\kappa_1(t-\tilde{t}_0)^{-\gamma_0}\|u_\varepsilon(\tilde{t}_0)\|_{L^q(\Omega)} +
\kappa_2\int_{\tilde{t}_0}^t(t-\tau)^{-\frac{1}{2}-\frac{3}{2}(\frac{2}{p}-\frac{1}{p})}\|v\|_{L^{p}(\Omega)}^2d\tau}\\
&\disp{+\kappa_3\|\nabla\phi\|_{L^\infty(\Omega)}\int_{\tilde{t}_0}^{\tilde{t}_0+3}\|m_\varepsilon(\cdot,\tau)-\hat{m}+n_\varepsilon(\cdot,\tau)
-\hat{n}\|_{L^p(\Omega)}d\tau.}\\
\end{array}
\label{cz2.571hhhhh51lllllccvvhddfccvvhjjjkkhhggjjlsdddlll}
\end{equation}
Therefore, in light of \dref{2222dddddf1.1ddfghssdddyddffssddduisda}, \dref{11111hhxxcdfvhhhvdssddddfssdddffsssddffssddfffssddffsscccjjghjjsdggggdddfffddffssddcssdz2.5} as well as \dref{cz2.571hhhhh51lllllccvvddfgghddfccvvhjjjkkhhggjjlsdddlll} and \dref{11111hhxxcdfvhhhvddfssdddffsssddffssddfffssddffsscccjjghjjsdggggdddfffddffssddcssdz2.5},  we see that for every $t\in (\tilde{t}_0 ,\tilde{t}_0 + 3)$ and every $v\in X$
\begin{equation}
\begin{array}{rl}
&(t-\tilde{t}_0)^{\gamma_0}\|\varphi(v) (\cdot, t)\|_{L^p(\Omega)}\\
\leq&\disp{\kappa_1\|u_{\varepsilon}(\tilde{t}_0)\|_{L^q(\Omega)} +
\kappa_2(t-\tilde{t}_0)^{\gamma_0}\int_{\tilde{t}_0}^t(t-\tau)^{-\frac{1}{2}-\frac{3}{2}(\frac{2}{p}-\frac{1}{p})}\|v\oplus v\|_{L^{\frac{p}{2}}(\Omega)}d\tau}\\
&\disp{+\kappa_3(t-\tilde{t}_0)^{\gamma_0}\|\nabla\phi\|_{L^\infty(\Omega)}\int_{\tilde{t}_0}^{\tilde{t}_0+3}\|m_\varepsilon(\cdot,\tau)-\hat{m}+
n_\varepsilon(\cdot,\tau)-\hat{n}\|_{L^p(\Omega)}d\tau}\\
\leq&\disp{\kappa_1\|u_{\varepsilon}(\tilde{t}_0)\|_{L^q(\Omega)} +
\kappa_2\eta_0^23^{\gamma_0}\int_{0}^3(t-\tau)^{-\frac{1}{2}-\frac{3}{2p}}\tau^{-2\gamma_0}d\tau}\\
&\disp{+\kappa_33^{\gamma_0}\|\nabla\phi\|_{L^\infty(\Omega)}\frac{\eta_0}{3^{\gamma_0+1}\kappa_3\|\nabla\phi\|_{L^\infty(\Omega)}}}\\
<&\disp{\eta_0.}\\
\end{array}
\label{cz2.571hhhhh51lllllccvvhddfccvvhjjjkkhhggjjllll}
\end{equation}
from which it readily follows that $\varphi(X)\subset X$.
 Likewise, for $v\in X$  and $w\in X$ we can
use Lemma \ref{llssdrffmmggnnccvvccvvkkkkgghhkkllvvlemma45630} to find that
\begin{equation}
\begin{array}{rl}
&\|\varphi(v) (\cdot, t)-\varphi(w)(\cdot, t) \|_{L^p(\Omega)}\\
\leq&\disp{
\kappa_2\int_{\tilde{t}_0}^t(t-\tau)^{-\frac{1}{2}-\frac{3}{2p}}\|v\oplus v-w\oplus w\|_{L^{\frac{p}{2}}(\Omega)}d\tau}\\
=&\disp{
\kappa_2\int_{\tilde{t}_0}^t(t-\tau)^{-\frac{1}{2}-\frac{3}{2p}}\|v\oplus (v-w)+(v-w)\oplus w\|_{L^{\frac{p}{2}}(\Omega)}d\tau}\\
\leq&\disp{
\kappa_2\int_{\tilde{t}_0}^t(t-\tau)^{-\frac{1}{2}-\frac{3}{2p}}(\|v\|_{L^{p}}+\|w\|_{L^{p}})\|v-w\|_{L^{p}(\Omega)}d\tau}\\
\leq&\disp{2
\kappa_2\eta_0\int_{0}^3\tau^{-\frac{1}{2}-\frac{3}{2p}}d\tau\|v-w\|_{L^{\infty}((0,3);L^{p}(\Omega)}.}\\
\end{array}
\label{234cz2.571hhhhh51lllllccvvhsdddddfccvvhjjjkkhhggjjlsdddlll}
\end{equation}
On the other hand, \dref{2222dddddf1.1ddfghssdddysddddddffssddduisda} implies that
$$2
\kappa_2\eta_0\int_{0}^3\tau^{-\frac{1}{2}-\frac{3}{2p}}d\tau<1,$$
whence \dref{234cz2.571hhhhh51lllllccvvhsdddddfccvvhjjjkkhhggjjlsdddlll} shows that
$\varphi$ acts as a contraction on
$X$ and hence possesses a unique fixed point  on $X$, which,
in view of  the definition of $\varphi$, must coincide with the unique weak solution $u_\varepsilon$ of fourth equation of \dref{1.1fghyuisda}
 on $(\tilde{t}_0 ,\tilde{t}_0+3)$ (see e.g.
Thm. V.2.5.1 of \cite{Sohr}). Now, by  \dref{cz2.571hhhhh51lllllccvvhddfccvvhjjjkkhhggjjllll}, we also derive that
\begin{equation}\|u_\varepsilon(\cdot,t)\|_{L^p(\Omega)}dx<\eta_0,~~\mbox{for all}~~t\in(\tilde{t}_0+1 ,\tilde{t}_0+3),
\label{11111hhxxcdfvhhhssddvdssddddfssdddffsssddffssddfffssddffsscccjjghjjsdggggdddfffddffssddcssdz2.5}
\end{equation}
from  $(\tilde{t}_0+1 ,\tilde{t}_0+3)\supset(t_1+2,t_1+3)$ we readily derive \dref{11111hhxxcdfvhhhvssssssssscccjjghjjsdgggddddgdddffddfffddffssddcssdz2.5}. 

\end{proof}
In the following lemmas, we next plan to prove  H\"{o}lder regularity
 of the components of a solution on
intervals of the form $(T_0 ,T_0 + 1)$ for $T_0> 0$ by using the maximal Sobolev regularity. 
To this end, we introduce the following cut-off functions,  which will play a key role in deriving higher order regularity for solution of problem \dref{1.1fghyuisda}.
\begin{definition}\label{aaalemmaddffffdssfffgg4sssdddd5630}
Let $\xi: \mathbb{R} \rightarrow [0,1]$ be a smooth, monotone function, satisfying $\xi\equiv 0$ on $(-\infty,0]$ and
$\xi\equiv 0$ on $(1,\infty)$ and for any $t_0\in \mathbb{R}$ we let $\xi_{t 0}:= \xi (t -t_0 )$.
\end{definition}
Due to the  above cut-off function,
it follows from maximal Sobolev regularity that the solution $(n_\varepsilon, c_\varepsilon, m_\varepsilon,u_\varepsilon)$ even satisfies estimates in
appropriate H\"{o}lder spaces:
\begin{lemma}\label{lemma45630hhuujjuuyy}
Let $\alpha>0$.
Then one can find $\mu\in(0, 1)$ and $T,\varepsilon_0,C>0$ such that for all $\varepsilon\in(0,\varepsilon_0)$
\begin{equation}
\|u_\varepsilon(\cdot,t)\|_{C^{1+\mu,\frac{\mu}{2}}(\bar{\Omega}\times[t,t+1])} \leq C ~~\mbox{for all}~~ t>T,
\label{zjscz2.5297x9630111kkhhffrroojj}
\end{equation}
%
%

\begin{equation}
\|c_\varepsilon(\cdot,t)\|_{C^{1+\mu,\frac{\mu}{2}}(\bar{\Omega}\times[t,t+1])}  \leq C ~~\mbox{for all}~~ t>T
\label{zjscz2.5297x9630111kkhhiioo}
\end{equation}
as well as
\begin{equation}
\|m_\varepsilon(\cdot,t)\|_{C^{1+\mu,\frac{\mu}{2}}(\bar{\Omega}\times[t,t+1])}  \leq C ~~\mbox{for all}~~ t>T
\label{ddhhhzjscz2.5297x9630111kkhhiioo}
\end{equation}
and
\begin{equation}
\|n_\varepsilon(\cdot,t)\|_{C^{1+\mu,\frac{\mu}{2}}(\bar{\Omega}\times[t,t+1])} \leq C ~~\mbox{for all}~~ t>T.
\label{zjscz2.5297x96dfgg30111kkhhffrroojj}
\end{equation}
\end{lemma}
\begin{proof}
Firstly, for  any $T_0 > 0$ let $\xi:=\xi_{T_0}$ and $v:=\xi u_\varepsilon.$
$$
 \left\{\begin{array}{ll}
   v_t-\Delta v=h~~~\mbox{in}~~\Omega\times (T_0,\infty),\\
\nabla\cdot v=0,~~~\mbox{in}~~\Omega\times (T_0,\infty),\\
 \disp{v(x,T_0)=0},\quad
x\in \Omega,\\
 \disp{v=0},\quad
\mbox{on}~~ \partial\Omega\times(T_0,\infty),\\
 \end{array}\right.
 $$
 where $$h=-\kappa(Y_\varepsilon u_\varepsilon \cdot \nabla)v+\nabla (\xi P_\varepsilon)+\xi(n_\varepsilon+m_\varepsilon)\nabla \phi+\xi' u_\varepsilon.$$
  To estimate the inhomogeneity $h$ herein,  we first note that  the known maximal Sobolev regularity estimate for the Stokes semigroup (\cite{Gigass12176}) yields a
constant $k_1 > 0$ such that
\begin{equation}
\begin{array}{rl}
&\disp\int_{T_0}^{T_0+2}\|v_t\|_{L^s(\Omega)}^s+\int_{T_0}^{T_0+2}\|D^2v\|_{L^s(\Omega)}^s\\
\leq&\disp{k_1\int_{T_0}^{T_0+2}\|\mathcal{P}(\xi Y_\varepsilon u_\varepsilon\cdot)u_\varepsilon\|_{L^s(\Omega)}^s+k_1\int_{T_0}^{T_0+2}
\left(\|\mathcal{P}\xi(m_\varepsilon(\cdot,\tau)-\hat{m}+n_\varepsilon(\cdot,\tau)-\hat{n})\nabla\phi\|_{L^s(\Omega)}^s\right)}\\
&\disp{+k_1\int_{T_0}^{T_0+2}\|\mathcal{P}\xi'  u_\varepsilon\|_{L^s(\Omega)}^s.}\\
\end{array}
\label{cz2.571hhhhh51lllllccvvhsdddddfccvvhjjjkkhhggjjlsdddlll}
\end{equation}
From the boundedness of the Helmholtz projection in $L^s$-spaces and the H\"{o}lder inequality we derive from Lemma \ref{aaalemmaddffffsddddfffgg4sssdddd5630} that there exist positive constant  $k_2 $ and $k_3$ such that
 for any $ T_0 > T$
\begin{equation}
\begin{array}{rl}
&\disp k_1\int_{T_0}^{T_0+2}\|\mathcal{P}(Y_\varepsilon u_\varepsilon\cdot)v\|_{L^s(\Omega)}^s\\
\leq&\disp{k_2\int_{T_0}^{T_0+2}\left(\|Y_\varepsilon u_\varepsilon\|_{L^{l'}(\Omega)}^s\|\nabla v\|_{L^{l}(\Omega)}^s\right) }\\
\leq&\disp{k_2\int_{T_0}^{T_0+2}\left(\| u_\varepsilon\|_{L^{l'}(\Omega)}^s\|\nabla v\|_{L^{l}(\Omega)}^s\right) }\\
\leq&\disp{k_3\int_{T_0}^{T_0+2}\|\nabla v\|_{L^{l}(\Omega)}^s~~~~~\mbox{for any}~~\varepsilon\in(0,\varepsilon_{T_0}),}\\
\end{array}
\label{cz2.57ssdd1hhhhh51lllllccvvhsddddddfdfccvvhjjjkkhhggjjlsdddlll}
\end{equation}
where $l>2s$ and $\frac{1}{l}+\frac{1}{l'}=1.$
Thanks to the Gagliardo-Nirenberg inequality, from Lemma \ref{aaalemmaddffffsddddfffgg4sssdddd5630} again, 
we can  estimate the right of \dref{cz2.57ssdd1hhhhh51lllllccvvhsddddddfdfccvvhjjjkkhhggjjlsdddlll} by following:
\begin{equation}
\begin{array}{rl}
&\disp k_3\int_{T_0}^{T_0+2}\|\nabla v\|_{L^{l}(\Omega)}^s\\
\leq&\disp{k_4\int_{T_0}^{T_0+2}\left(\|D^2 v\|_{L^s(\Omega)}^{as}\| v\|_{L^{r_0}(\Omega)}^{(1-a)s}\right) }\\
\leq&\disp{k_4\int_{T_0}^{T_0+2}\left(\|D^2 v\|_{L^s(\Omega)}^{as}\| u_\varepsilon\|_{L^{r_0}(\Omega)}^{(1-a)s}\right) }\\
\leq&\disp{k_5\int_{T_0}^{T_0+2}\|D^2 v\|_{L^s(\Omega)}^{as}~~~
~~\mbox{for any}~~\varepsilon\in(0,\varepsilon_{T_0}),}\\
\end{array}
\label{cz2.57ssdd1hhhhh51lllllccvvhsddddddfdfccvvhjddffjjkkhhgdddgjjlsdddlll}
\end{equation}
where $a\in(0,1)$ satisfies
$$\frac{1}{l}-\frac{1}{3}=a(\frac{1}{s}-\frac{2}{3})+(1-a)\frac{1}{r_0}.$$
 Therefore, inserting \dref{cz2.57ssdd1hhhhh51lllllccvvhsddddddfdfccvvhjddffjjkkhhgdddgjjlsdddlll} into \dref{cz2.57ssdd1hhhhh51lllllccvvhsddddddfdfccvvhjjjkkhhggjjlsdddlll} and
 applying the  Young inequality, we find $k_6> 0$ such that for all $t_0 > T$
 \begin{equation}
\begin{array}{rl}
\disp k_1\int_{T_0}^{T_0+2}\|\mathcal{P}(Y_\varepsilon u_\varepsilon\cdot)v\|_{L^s(\Omega)}^s
\leq&\disp{\frac{1}{2}\int_{T_0}^{T_0+2}\|D^2 v\|_{L^s(\Omega)}^{s}+k_6~~~~~\mbox{for any}~~\varepsilon\in(0,\varepsilon_{T_0}).}\\
\end{array}
\label{cz2.57ssdd1hhhhh51lllllkkllccvvhsddddddfdfccvvhjjjkkhhggjjlsdddlll}
\end{equation}
Moreover, we derive from  Definition \ref{aaalemmaddffffdssfffgg4sssdddd5630} and Lemmas \ref{ssdddlemmddddaddffffdfffgg4sssdddd5630}, \ref{11aaalemmaddffffdsddfffffgg4sssdddd5630} and  \ref{aaalemmaddffffsddddfffgg4sssdddd5630}, there is $k_7> 0$ such that
\begin{equation}
\begin{array}{rl}
&\disp k_1\int_{T_0}^{T_0+2}\|\mathcal{P}\xi(m_\varepsilon(\cdot,\tau)-\hat{m}+n_\varepsilon(\cdot,\tau)-\hat{n})\nabla\phi\|_{L^s(\Omega)}^s
+k_1\int_{T_0}^{T_0+2}\|\mathcal{P}\xi'  u_\varepsilon\|_{L^s(\Omega)}^s\\
\leq&\disp{
k_7,}\\
\end{array}
\label{cz2.571hhhhh51lllllccvvhsdddddfccvvhjjjkkhhgssdgjjlsdddlll}
\end{equation}
so that invoking \dref{cz2.57ssdd1hhhhh51lllllkkllccvvhsddddddfdfccvvhjjjkkhhggjjlsdddlll} and \dref{cz2.571hhhhh51lllllccvvhsdddddfccvvhjjjkkhhggjjlsdddlll} we can estimate
\begin{equation}
\begin{array}{rl}
\disp\int_{T_0}^{T_0+2}\|v_t\|_{L^s(\Omega)}^s+\int_{T_0}^{T_0+2}\|D^2v\|_{L^s(\Omega)}^s
\leq&\disp{k_8.}\\
\end{array}
\label{cz2.571hhhhh51lllllccddfffvvhsdddddfccvvhjjjkkhhggjjlsdddlll}
\end{equation}
Therefore, by the definition of $\xi$, for any $s > 1$, there exist positive constants $C $ and $T$ such that for any
$t > T$ there is $\varepsilon_0 > 0$ satisfying that for any $\varepsilon\in (0,\varepsilon_0)$
\begin{equation}
\|u_\varepsilon(\cdot,t)\|_{L^{s}((t,t+1);W^{2,s}(\Omega))} +\|u_{\varepsilon t}(\cdot,t)\|_{L^{s}(\Omega\times(t,t+1))}\leq C ~~\mbox{for all}~~ t\geq T,
\label{zjscz2.5297x9dddjkkkkkd630111kkhhffrddroojj}
\end{equation}
which in view of a known embedding result (\cite{AmannAmannmo1216}) implies that for all $ t_0 > 0$,
we can find $\theta_1\in (0, 1)$ and $C_{10}$ such  that
\begin{equation}
\|u_\varepsilon(\cdot,t)\|_{C^{1+\theta_1,\theta_1}(\bar{\Omega}\times[t,t+1])} \leq C_{10} ~~\mbox{for all}~~ t> T.
\label{zjscz2.5297x9ddddfggdd630111kkhhffrfffroojj}
\end{equation}
 Likewise, again using the  maximal Sobolev regularity estimates and the Gagliardo-Nirenberg inequality,
 we can claim that \dref{zjscz2.5297x9630111kkhhiioo}--\dref{zjscz2.5297x96dfgg30111kkhhffrroojj}  hold by applying Lemmas \ref{fvfgsdfggfflemma45}, \ref{lemmaghjffggssddgghhmk4563025xxhjklojjkkk} and \ref{aaalemmaddffffsddddfffgg4sssdddd5630}.
\end{proof}
Straightforward applications of standard Schauder estimates for the Stokes evolution equation and the
heat equation, respectively, finally yield eventual smoothness of the solution $(n_\varepsilon, c_\varepsilon, m_\varepsilon,u_\varepsilon)$.

\begin{lemma}\label{lemma45630hhuujjsdfffggguuyy}
Let $\alpha>0$.
Then one can find $\mu\in(0, 1)$ and $T>0$ such that for some $C > 0$
\begin{equation}
\|u_\varepsilon(\cdot,t)\|_{C^{2+\mu,1+\frac{\mu}{2}}(\bar{\Omega}\times[t,t+1];\mathbb{R}^3)} \leq C ~~\mbox{for all}~~ t>T,
\label{222zjscz2.5297x9630111kkhhffrroojj}
\end{equation}
%
%

\begin{equation}
\|c_\varepsilon(\cdot,t)\|_{C^{2+\mu,1+\frac{\mu}{2}}(\bar{\Omega}\times[t,t+1])}  \leq C ~~\mbox{for all}~~ t>T
\label{222zjscz2.5297x9630111kkhhiioo}
\end{equation}
as well as
\begin{equation}
\|m_\varepsilon(\cdot,t)\|_{C^{2+\mu,1+\frac{\mu}{2}}(\bar{\Omega}\times[t,t+1])}  \leq C ~~\mbox{for all}~~ t>T
\label{222ddhhhzjscz2.5297x9630111kkhhiioo}
\end{equation}
and
\begin{equation}
\|n_\varepsilon(\cdot,t)\|_{C^{2+\mu,1+\frac{\mu}{2}}(\bar{\Omega}\times[t,t+1])} \leq C ~~\mbox{for all}~~ t>T.
\label{222zjscz2.5297x96dfgg30111kkhhffrroojj}
\end{equation}
\end{lemma}
\begin{proof}
We first combine Lemma \ref{lemma45630hhuujjuuyy} to infer
the existence of $\alpha_1\in(0,1)$, $T_1 > 0$ and $C_1 > 0$ such that  for all $t>T_1$,
\begin{equation}
\begin{array}{rl}
&\|u_\varepsilon\cdot \nabla c_\varepsilon\|_{C^{\alpha_1,\frac{\alpha_1}{2}}(\bar{\Omega}\times[t,t+1])}+\|u_\varepsilon\cdot \nabla m_\varepsilon\|_{C^{\alpha_1,\frac{\alpha_1}{2}}(\bar{\Omega}\times[t,t+1])}\\
 &+\|n_\varepsilon m_\varepsilon\|_{C^{\alpha_1,\frac{\alpha_1}{2}}(\bar{\Omega}\times[t,t+1])}+\| m_\varepsilon\|_{C^{\alpha_1,\frac{\alpha_1}{2}}(\bar{\Omega}\times[t,t+1])}\\
 \leq &C_1.
\end{array}\label{222zjscz2.52ssdd97x9630111ffggkkhhffrroojj}
\end{equation}
Standard parabolic Schauder estimates applied to the second and third  equation in \dref{1.1fghyuisda} (\cite{Ladyzenskajaggk7101})
thus provide $C_2 > 0$ fulfilling
\begin{equation}
\|c_\varepsilon\|_{C^{2+\alpha_1,1+\frac{\alpha_1}{2}}(\bar{\Omega}\times[t,t+1])}+\| m_\varepsilon\|_{C^{2+\alpha_1,1+\frac{\alpha_1}{2}}(\bar{\Omega}\times[t,t+1])}\leq C_2~~~\mbox{for all}~~ t > T_1 + 1.
\label{222zjscz2.52ssdd97x9630111kddffkhhffrroojj}
\end{equation}
According to Lemma \ref{lemma45630hhuujjuuyy}, it is possible to fix $\alpha_2\in (0,1)$,
$T_2> 0$ and $C_3> 0$ such that
\begin{equation}
\|n_\varepsilon\|_{C^{1+\alpha_2,\frac{\alpha_2}{2}}(\bar{\Omega}\times[t,t+1])}+\| u_\varepsilon\|_{C^{1+\alpha_2,\frac{\alpha_2}{2}}(\bar{\Omega}\times[t,t+1];\mathbb{R}^3)}\leq C_3~~~\mbox{for all}~~ t > T_2.
\label{222zjscz2.52ssdd97x9630111kddffkhhffrssddroojj}
\end{equation}
We next set $T:= T_2+1$ and let $t_0 > T$ be given. Then with $\xi_{t_0}$ taken from Definition \ref{aaalemmaddffffdssfffgg4sssdddd5630},
we again use that $v(\cdot,t) := \xi_{t_0}u_\varepsilon(\cdot,t), (x\in\Omega,t> t_0-1)$, is a solution of
\begin{equation}
 \left\{\begin{array}{ll}
   v_t-\Delta v=h~~~~x\in\Omega, t>t_0-1,\\
 \disp{v(x,t_0-1)=0},\quad
x\in \Omega,\\
 \end{array}\right.
\label{222zjscz2.52ssdd9ddff7x9630111kddffkhhffrssddroojj}
\end{equation}
 where $$h_\varepsilon=-\kappa(Y_\varepsilon u_\varepsilon \cdot \nabla)v+\nabla (\xi P_\varepsilon)+\xi(n_\varepsilon+m_\varepsilon)\nabla \phi+\xi' u_\varepsilon.$$
 Now from \dref{222zjscz2.52ssdd97x9630111kddffkhhffrssddroojj} and the smoothness of $\xi$ we readily obtain $\alpha_3\in (0,1)$
and $C_4> 0$ fulfilling
\begin{equation}
\|h_\varepsilon\|_{C^{\alpha_3,\frac{\alpha_3}{2}}(\bar{\Omega}\times[t_0-1,t_0+1];\mathbb{R}^3)}\leq C_4,
\label{222zjscz2.52ssdd97x9dddd630111kddffkhhffrssddroojj}
\end{equation}
so that regularity estimates from Schauder theory for the Stokes evolution equation
(\cite{SolonnikovSolonnikov1216}) ensure that \dref{222zjscz2.52ssdd9ddff7x9630111kddffkhhffrssddroojj} possesses a classical solution $\bar{v}\in {C^{2+\alpha_3,1+\frac{\alpha_3}{2}}(\bar{\Omega}\times[t_0-1,t_0+1])}$ satisfying
\begin{equation}
\|\bar{v}\|_{C^{2+\alpha_3,1+\frac{\alpha_3}{2}}(\bar{\Omega}\times[t_0-1,t_0+1];\mathbb{R}^3)}\leq C_5
\label{222zjscz2.52ssdd97x9dddd63011ssdd1kddffkhhffrssddroojj}
\end{equation}
with some $C_5 > 0$ which is independent of $t_0$. This combined with the
uniqueness property of \dref{222zjscz2.52ssdd9ddff7x9630111kddffkhhffrssddroojj}, one can prove
\begin{equation}
\|u_\varepsilon\|_{C^{2+\alpha_3,1+\frac{\alpha_3}{2}}(\bar{\Omega}\times[t,t+1];\mathbb{R}^3)}\leq C_6.
\label{222zjscz2.52ssdd9ddff7x9dddd63011ssdd1kddffkhhffrssddroojj}
\end{equation}
Again relying on Lemma \ref{lemma45630hhuujjuuyy}, this in turn warrants that for
some $\alpha_4\in(0,1), T_4> 0$ and $C_7> 0$ such that for all $t>T_4$
\begin{equation}
\|\nabla\cdot(n_{\varepsilon}S_\varepsilon(x, n_{\varepsilon}, c_{\varepsilon})\nabla c_{\varepsilon})\|_{C^{\alpha_4,\frac{\alpha_4}{2}}(\bar{\Omega}\times[t,t+1])}+\|u_{\varepsilon}\cdot\nabla n_{\varepsilon}\|_{C^{\alpha_4,\frac{\alpha_4}{2}}(\bar{\Omega}\times[t,t+1])}+
\|n_{\varepsilon}m_{\varepsilon}\|_{C^{\alpha_4,\frac{\alpha_4}{2}}(\bar{\Omega}\times[t,t+1])}\leq C_7,
\label{222zjscz2.52ssdd9ddff7x9ddkklldd63011ssdd1kddffkhhffrssddroojj}
\end{equation}
which along with the Schauder theory says establishes 
\begin{equation}
\|n_\varepsilon\|_{C^{2+\alpha_4,1+\frac{\alpha_4}{2}}(\bar{\Omega}\times[t,t+1])}\leq C_8.
\label{222zjscz2.52ssdd9ddff7x9dsddddd63011ssdd1kddffkhhffrssddroojj}
\end{equation}
Finally, choose $T=\max\{T_1,T_1,T_2,T_3,T_4\}$ and $\mu=\min\{\alpha_1,\alpha_2,\alpha_3,\alpha_4\}$, then \dref{222zjscz2.52ssdd97x9630111kddffkhhffrroojj}, \dref{222zjscz2.52ssdd9ddff7x9dddd63011ssdd1kddffkhhffrssddroojj}, \dref{222zjscz2.52ssdd9ddff7x9dsddddd63011ssdd1kddffkhhffrssddroojj} imply \dref{222zjscz2.5297x9630111kkhhffrroojj}--\dref{222zjscz2.5297x96dfgg30111kkhhffrroojj}.
\end{proof}

Having found uniform H\"{o}lder bounds on $n_\varepsilon, c_\varepsilon, m_\varepsilon$ and  $u_\varepsilon$ for $\varepsilon> 0$ in the previous three lemmas (see Lemmas \ref{lemma45630hhuujjuuyy} and \ref{lemma45630hhuujjsdfffggguuyy}), also $n, c,m$
and $u$ share this regularity and these bounds.

\begin{lemma}\label{lemma45630223}
Assume that   $\alpha>0$. There exist $\theta\in (0,1)$ as well as   $T_0 > 0$, $(\varepsilon_j)_{j\in \mathbb{N}}\subset (0, 1)$ of the sequence
from Lemma \ref{lemma45hyuuuj630223} such that for any $t>T_0 $
\begin{equation}
 \left\{\begin{array}{ll}
 n\in C^{2+\theta,1+\frac{\theta}{2}}(\bar{\Omega}\times[t,t+1]),\\
  c\in  C^{2+\theta,1+\frac{\theta}{2}}(\bar{\Omega}\times[t,t+1]),\\
  m\in  C^{2+\theta,1+\frac{\theta}{2}}(\bar{\Omega}\times[t,t+1]),\\
  u\in  C^{2+\theta,1+\frac{\theta}{2}}(\bar{\Omega}\times[t,t+1];\mathbb{R}^3),\\
   \end{array}\right.\label{1.ffhhh1hhhjjkdffggdfghyuisda}
\end{equation}
 that $\varepsilon_j\searrow 0$ as $j\rightarrow\infty$  and
\begin{equation}
 \left\{\begin{array}{ll}
 n_\varepsilon\rightarrow n~~\in C^{1+\theta,\frac{\theta}{2}}(\bar{\Omega}\times[t,t+1]),\\
  c_\varepsilon\rightarrow c~~\in C^{1+\theta,\frac{\theta}{2}}(\bar{\Omega}\times[t,t+1]),\\
   m_\varepsilon\rightarrow m~~\in C^{1+\theta,\frac{\theta}{2}}(\bar{\Omega}\times[t,t+1]),\\
 u_\varepsilon\rightarrow u~~\in C^{1+\theta,\frac{\theta}{2}}(\bar{\Omega}\times[t,t+1];\mathbb{R}^3)\\
   \end{array}\right.
   and\label{1.ffgghhhhh1dffggdfghyuisda}
\end{equation}
as $\varepsilon=\varepsilon_j\searrow 0$.
Moreover, there is $C > 0$ such that
\begin{equation}
\|c(\cdot,t)\|_{C^{2+\theta,1+\frac{\theta}{2}}(\bar{\Omega}\times[t,t+1])}+ \|m\cdot,t)\|_{C^{2+\theta,1+\frac{\theta}{2}}(\bar{\Omega}\times[t,t+1])}\leq C ~~\mbox{for all}~~ t>T_0
\label{222zjscz2.5297x9630111kkhhfsddddfrroojj}
\end{equation}
as well as
\begin{equation}
\|n(\cdot,t)\|_{C^{2+\theta,1+\frac{\theta}{2}}(\bar{\Omega}\times[t,t+1])}+\|u(\cdot,t)\|_{C^{2+\theta,1+\frac{\theta}{2}}(\bar{\Omega}\times[t,t+1]);\mathbb{R}^3)}\leq C ~~\mbox{for all}~~ t>T_0.
\label{sss222zjscz2.5297x9sdddd630111kkhhfsddddfrroojj}
\end{equation}

\end{lemma}
\begin{proof}

In conjunction with Lemmas \ref{lemma45630hhuujjsdfffggguuyy} and \ref{lemma45hyuuuj630223}  and the standard compactness arguments (see \cite{Simon}), we can thus
find a sequence $(\varepsilon_j)_{j\in \mathbb{N}}\subset (0, 1)$ such that $\varepsilon_j\searrow 0$ as $j\rightarrow\infty$, and such that
\dref{1.ffhhh1hhhjjkdffggdfghyuisda}--\dref{sss222zjscz2.5297x9sdddd630111kkhhfsddddfrroojj} hold.
The proof of Lemma \ref{lemma45630223} is completed.
\end{proof}

\begin{lemma}\label{sssslemma45ssddddff630hhuujjsdfffggguuyy}
Let $\alpha>0$.
Then one can find $\theta\in(0, 1)$ and $T>0$ such that 
\begin{equation}
\|c(\cdot,t)\|_{C^{2+\theta,1+\frac{\theta}{2}}(\bar{\Omega}\times[T,\infty))} \leq C
\label{222zjscffgggz2.5fff297x9630111kkhhffrroojj}
\end{equation}
%
%
as well as
\begin{equation}
\|u(\cdot,t)\|_{C^{2+\theta,1+\frac{\theta}{2}}(\bar{\Omega}\times[T,\infty);\mathbb{R}^3)} \leq C
\label{222zjscffgggz2.5sdddff297x9630111kkhhffrroojj}
\end{equation}
and
\begin{equation}
\|n(\cdot,t)\|_{C^{2+\theta,1+\frac{\theta}{2}}(\bar{\Omega}\times[T,\infty))} +
\|m(\cdot,t)\|_{C^{2+\theta,1+\frac{\theta}{2}}(\bar{\Omega}\times[T,\infty))}   \leq C.
\label{222zjscz2ggggg.5297x9630111kkhhiioo}
\end{equation}
\end{lemma}
\begin{proof}
Let $g := -\xi c+m\xi -\xi u\cdot\nabla c+c\xi'$, where $\xi:=\xi_{T_0}$ is given by {Definition} \ref{aaalemmaddffffdssfffgg4sssdddd5630}
and $T_0$ is  same as   the previous lemmas.
Then we  consider the following problem
\begin{equation}
 \left\{\begin{array}{ll}
   \tilde{c}_t-\Delta \tilde{c}=g~~~~x\in\Omega, t>T_0,\\
 \disp{\tilde{c}(T_0)=0},\quad
x\in \Omega,\\
\disp{\frac{\partial\tilde{c}}{\partial \nu}=0},\quad
x\in \partial\Omega,\\
 \end{array}\right.
\label{222zjscz2.52ssdd9dssddddff7x9630111kddffkhhffssdddrssddroojj}
\end{equation}
In view of Lemma \ref{lemma45630223} and {Definition} \ref{aaalemmaddffffdssfffgg4sssdddd5630}, we drive that
$$g~~\mbox{is bounded in}~~ C^{\theta} (\bar{\Omega}\times(T, \infty)),$$
 so that, regularity estimates from Schauder theory for the  parabolic equation
(see e.g. III.5.1 of \cite{Ladyzenskajaggk7101}) ensure that  problem \dref{222zjscz2.52ssdd9dssddddff7x9630111kddffkhhffssdddrssddroojj} admits a unique solution $\tilde{c}\in C^{2+\theta,1+\frac{\theta}{2}}(\bar{\Omega}\times[T_0+1,\infty)).$
This combined with the property of $\xi$ implies that
\begin{equation}c\in C^{2+\theta,1+\frac{\theta}{2}}(\bar{\Omega}\times[T_0+1,\infty)).
\label{222zjscz2.52ssdd9dssddddff7x9630111kddffkhhffssddddffdrssddroojj}
\end{equation} Applying the same argument  one can derive the third equation of \dref{334451.1fghyuisda} that
\begin{equation}m\in C^{2+\theta,1+\frac{\theta}{2}}(\bar{\Omega}\times[T_0+1,\infty)).
\label{222zjscz2.52ssdd9dssddddff7x9630111kddffkhhffssdddrssddsdfgghhroojj}
\end{equation}

Finally,
employing almost exactly the same arguments as in the proof of Lemma  \ref{lemma45630hhuujjsdfffggguuyy} (the minor necessary changes are left as an easy exercise to the reader), and taking advantage of \dref{sss222zjscz2.5297x9sdddd630111kkhhfsddddfrroojj}, we conclude that
\begin{equation}u\in C^{2+\theta,1+\frac{\theta}{2}}(\bar{\Omega}\times[T_0+1,\infty);\mathbb{R}^3)
\label{222zjscz2.52ssdd9dssddddff7x9630111kddddffffkhhffssdddrssddsdfgghhroojj}
\end{equation}
and
 \begin{equation}n\in C^{2+\theta,1+\frac{\theta}{2}}(\bar{\Omega}\times[T_0+1,\infty)),
 \label{222zjscz2.52ssdd9dssddddff7x9630111kddddffffkhhffssdddrssddssddffdfgghhroojj}
\end{equation}
whence combining the result of \dref{222zjscz2.52ssdd9dssddddff7x9630111kddffkhhffssddddffdrssddroojj} with \dref{222zjscz2.52ssdd9dssddddff7x9630111kddffkhhffssdddrssddsdfgghhroojj} completes the proof.
\end{proof}

%
%

On the basis of the eventual uniform continuity properties implied by the estimates in this
section (see Lemma \ref{sssslemma45ssddddff630hhuujjsdfffggguuyy}), by using the interpolation inequality, we can now turn the weak stabilization properties of $n,c,m$ and $u$ from Lemmas \ref{ssdddlemmddddaddffffdfffgg4sssdddd5630}--\ref{11aaalemmaddffffdsddfffffgg4sssdddd5630} into convergence with regard to the norm in $L^\infty(\Omega)$.

\begin{lemma}\label{lemma4dd5630hhuujjuuyy}
Let $\alpha>0$. The solution $(n,c,m,u)$ of \dref{1.1fghyuisda} constructed in Lemma \ref{lemma45hyuuuj630223} satisfies
$$n(\cdot,t)\rightarrow \hat{n},
m(\cdot,t)\rightarrow \hat{m}~~\mbox{as well as } ~~~c(\cdot,t)\rightarrow \hat{m}~~\mbox{and}~~~u(\cdot,t)\rightarrow0
~~\mbox{in}~~~L^\infty(\Omega),$$
where $\hat{n}=\frac{1}{|\Omega|}\{\int_{\Omega}n_0-\int_{\Omega}m_0\}_{+}$ and $\hat{m}=\frac{1}{|\Omega|}\{\int_{\Omega}m_0 -\int_{\Omega}n_0\}_{+}$.
\end{lemma}
\begin{proof}
Firsly, due to Lemmas \ref{ssdddlemmddddaddffffdfffgg4sssdddd5630}--\ref{11aaalemmaddffffdsddfffffgg4sssdddd5630}, we derive from Lemma \ref{sssslemma45ssddddff630hhuujjsdfffggguuyy} that
\begin{equation}n(t)\rightarrow \hat{n},
c(t)\rightarrow \hat{m}, m(t)\rightarrow \hat{m}~~\mbox{and}~~u(t)\rightarrow0~~~\mbox{in}~~~L^2(\Omega),
\label{aahhxxcdfvhhhvssssssssdsssjjdfffddffssllllddcssdz2.ssdd5}
\end{equation}
where $\hat{m}$ and $\hat{n}$ are given by \dref{1111hhxxcdfvhhhvsddfffgssjjdfffsfffsddcsssz2.5} and \dref{1111hhxxcddffdfvhhhvsddfffgssdfffjjdfffssddcsssz2.5}, respectively.
Next, due to Lemma \ref{lemma45630223}, one can obtain there exist positive  constants $\kappa_1$ and $T$ such that for all $t>T$
\begin{equation}\|n (\cdot,t)\|_{C^{2+\theta}(\bar{\Omega})} +\|c (\cdot,t)\|_{C^{2+\theta}(\bar{\Omega})}+\|m (\cdot,t)\|_{C^{2+\theta}(\bar{\Omega})}+\|u (\cdot,t)\|_{C^{2+\theta}(\bar{\Omega})}  \leq \kappa_1.
\label{aahhxxcdfvhhhvssssssssdsssjjdfffddffssddcssdz2.ssdd5}
\end{equation}
Therefore,
for any  $\eta > 0$, we may use the compactness of the first of the embeddings
$C^{2+\theta}(\bar{\Omega}) \hookrightarrow\hookrightarrow L^\infty (\Omega)\hookrightarrow L^2 (\Omega)$ to fix,
through an associated Ehrling lemma, a constant $\kappa_{2} > 0$
such that
 \begin{equation}\|n(\cdot,t)-\hat{n}\|_{L^{\infty}(\Omega)}\leq\frac{\eta}{2\kappa_{1} }\|n(\cdot,t)-\hat{n}\|_{C^{2+\theta}(\bar{\Omega})}+\kappa_{2}\|n(\cdot,t)-\hat{n}\|_{L^{2}(\Omega)}
\label{233ddxcvbbggdddddddfghhdfgcz2vv.5ghju4ss8cfg9ddsddddffff24ssdddghddfgggyddfggusdffji}
\end{equation}
 \begin{equation}\|c(\cdot,t)-\hat{m}\|_{L^{\infty}(\Omega)}\leq\frac{\eta}{2\kappa_{1} }\|c(\cdot,t)-\hat{m}\|_{C^{2+\theta}(\bar{\Omega})}+\kappa_{2}\|c(\cdot,t)-\hat{m}\|_{L^{2}(\Omega)}
\label{ghhddxcvbbggdddddddfghhdfgcz2vv.5ghju4ss8cfg9ddsddddffff24ssdddghddfgggyddfggusdffji}
\end{equation}
as well as
 \begin{equation}\|m(\cdot,t)-\hat{m}\|_{L^{\infty}(\Omega)}\leq\frac{\eta}{2\kappa_{1} }\|m(\cdot,t)-\hat{m}\|_{C^{2+\theta}(\bar{\Omega})}+\kappa_{2}\|m(\cdot,t)-\hat{m}\|_{L^{2}(\Omega)}
\label{klllddxcvbbggdddddddfghhdfgcz2vv.5ghju4ss8cfg9ddsddddffff24ssdddghddfgggyddfggusdffji}
\end{equation}
and
 \begin{equation}\|u(\cdot,t)\|_{L^{\infty}(\Omega)}\leq\frac{\eta}{2\kappa_{1} }\|u(\cdot,t)\|_{C^{2+\theta}(\bar{\Omega})}+\kappa_{2}\|u(\cdot,t)\|_{L^{2}(\Omega)}.
\label{dllkoooopdxcvbbggdddddddfghhdfgcz2vv.5ghju4ss8cfg9ddsddddffff24ssdddghddfgggyddfggusdffji}
\end{equation}
Now due to \dref{aahhxxcdfvhhhvssssssssdsssjjdfffddffssllllddcssdz2.ssdd5}, we may choose $t_0 > \max\{1,T\}$ large enough
such that for all $t>t_0$,
 \begin{equation}\|n(\cdot,t)-\hat{n}\|_{L^2(\Omega)}+\|c(\cdot,t)-\hat{m}\|_{L^2(\Omega)}+\|m(\cdot,t)-\hat{m}\|_{L^2(\Omega)}+\|u(\cdot,t)\|_{L^2(\Omega)}<\frac{\eta}{2\kappa_{2}}.
\label{234ddxcvbbggdddddddfghhdfgcz2vv.5ghju4ss8cfg9ddsddddffff24ghddfgggyddfggusdffji}
\end{equation}
Combined with \dref{233ddxcvbbggdddddddfghhdfgcz2vv.5ghju4ss8cfg9ddsddddffff24ssdddghddfgggyddfggusdffji}--\dref{234ddxcvbbggdddddddfghhdfgcz2vv.5ghju4ss8cfg9ddsddddffff24ghddfgggyddfggusdffji}, this shows that in fact
$$
\begin{array}{rl}\|n(\cdot,t)-\hat{n}\|_{L^{\infty}(\Omega)}\leq&\disp{\frac{\eta}{2\kappa_{1} }\|n(\cdot,t)-\hat{n}\|_{C^{2+\theta}(\bar{\Omega})}+\kappa_{2}\|n(\cdot,t)-\hat{n}\|_{L^{2}(\Omega)}}\\
<&\disp{\frac{\eta}{2\kappa_{1} }\kappa_{1}+\kappa_{2}\frac{\eta}{2\kappa_{2} }}\\
=&\eta~~~\mbox{for all}~~ t > t _0,\\
\end{array}
$$
$$
\begin{array}{rl}\|c(\cdot,t)-\hat{m}\|_{L^{\infty}(\Omega)}\leq&\disp{\frac{\eta}{2\kappa_{1} }\|c(\cdot,t)-\hat{m}\|_{C^{2+\theta}(\bar{\Omega})}+\kappa_{2}\|c(\cdot,t)-\hat{m}\|_{L^{2}(\Omega)}}\\
<&\disp{\frac{\eta}{2\kappa_{1} }\kappa_{1}+\kappa_{2}\frac{\eta}{2\kappa_{2} }}\\
=&\eta~~~\mbox{for all}~~ t > t _0\\
\end{array}
$$
as well as
$$
\begin{array}{rl}\|m(\cdot,t)-\hat{m}\|_{L^{\infty}(\Omega)}\leq&\disp{\frac{\eta}{2\kappa_{1} }\|m(\cdot,t)-\hat{m}\|_{C^{2+\theta}(\bar{\Omega})}+\kappa_{2}\|m(\cdot,t)-\hat{m}\|_{L^{2}(\Omega)}}\\
<&\disp{\frac{\eta}{2\kappa_{1} }\kappa_{1}+\kappa_{2}\frac{\eta}{2\kappa_{2} }}\\
=&\eta~~~\mbox{for all}~~ t > t _0\\
\end{array}
$$
and
$$
\begin{array}{rl}\|u(\cdot,t)\|_{L^{\infty}(\Omega)}\leq&\disp{\frac{\eta}{2\kappa_{1} }\|u(\cdot,t)\|_{C^{2+\theta}(\bar{\Omega})}+\kappa_{2}\|u(\cdot,t)\|_{L^{2}(\Omega)}}\\
<&\disp{\frac{\eta}{2\kappa_{1} }\kappa_{1}+\kappa_{2}\frac{\eta}{2\kappa_{2} }}\\
=&\eta~~~\mbox{for all}~~ t > t _0,\\
\end{array}
$$
which together with the fact that $\eta> 0$ was arbitrary implies the 
claimed estimates.
\end{proof}
%
%
%
%
%
%
%
%
%
%
%
%
%
%
%
%
%
%
%
%

In order to prove Theorem \ref{thaaaeorem3}, we now only have to collect the results prepared during this section:

{\bf Proof of Theorem  \ref{thaaaeorem3}.} 


\begin{proof}
Combining Lemmas \ref{lemma45630223}--\ref{lemma4dd5630hhuujjuuyy} this convergence statement results immediately.
\end{proof}

{\bf Acknowledgement}:
This work is partially supported by  the National Natural
Science Foundation of China (No. 11601215), Shandong Provincial
Science Foundation for Outstanding Youth (No. ZR2018JL005) and Project funded by China
Postdoctoral Science Foundation (No. 2019M650927, 2019T120168).

\end{document}